\documentclass[notheoremnums,url=hyphens]{dpreprint}

\newtheorem{maintheorem}{Theorem}

\newtheorem{maincorollary}[maintheorem]{Corollary}

\newtheorem{theorem}{Theorem}[section]

\newtheorem{lemma}[theorem]{Lemma}
\newtheorem{corollary}[theorem]{Corollary}

\newtheorem{proposition}[theorem]{Proposition}

\theoremstyle{definition}
\newtheorem{definition}[theorem]{Definition}
\newtheorem{notation}[theorem]{Notation}
\newtheorem{example}[theorem]{Example}
\newtheorem{remark}[theorem]{Remark}

\renewcommand{\ceil}[1]{\left\lceil #1 \right\rceil}
\renewcommand{\floor}[1]{\left\lfloor #1 \right\rfloor}

\usepackage{xcolor}
\usepackage{faktor}
\usepackage{pgf, tikz}
\usetikzlibrary{arrows, automata}
\hypersetup{pdfpagemode=UseNone}

\newcommand{\len}[1]{| #1 |}

\newcommand{\bigslant}[2]{{\raisebox{.2em}{$#1$}\left/\raisebox{-.2em}{$#2$}\right.}}

\begin{document}

\dtitle[On minimal subshifts of linear word complexity with slope less than $3/2$]{{\LARGE On minimal subshifts of linear word complexity with slope less than $3/2$}}
\dauthorone[D.~Creutz]{Darren Creutz}{darren.a.creutz@vanderbilt.edu}{Vanderbilt University}{}
\dauthortwo[R.~Pavlov]{Ronnie Pavlov}{rpavlov@du.edu}{University of Denver}{The second author gratefully acknowledges the support of a Simons Foundation Collaboration Grant.}
\datewritten{\today}

\keywords{Symbolic dynamics, word complexity, discrete spectrum, S-adic, Pisot conjecture, Sarnak conjecture, nilmanifolds, adeles}
\subjclass{Primary: 37B10; Secondary 37A25}

\dabstract{%
We prove that every infinite minimal subshift with word complexity $p(q)$ satisfying $\limsup p(q)/q < 3/2$ is
measure-theoretically isomorphic to its maximal equicontinuous factor; in particular, it has measurably discrete spectrum. Among other applications, this provides a proof of Sarnak's conjecture for all subshifts with $\limsup p(q)/q < 3/2$ (which can be thought of as a much stronger version of zero entropy). We also show that all such subshifts are balanced, and that orbit equivalence between them is characterized by the additive eigenvalue group. %The bound of $3/2$ is optimal, since \cite{creutzpavlov} contains an example with $\limsup p(q)/q = 3/2$ which is weakly mixing.

As in \cite{creutzpavlov}, our main technique is proving that all low-complexity minimal subshifts have a specific type of representation via a sequence $\{\tau_k\}$ of substitutions, usually called an S-adic decomposition. The maximal equicontinuous factor is the (possibly twisted) product of an odometer with a rotation on a compact abelian connected one-dimensional group (either a circle or solenoid), for which we can give an explicit description in terms of the substitutions $\tau_k$. We also prove that all such odometers and groups may appear for minimal subshifts with $\limsup p(q)/q = 1$, demonstrating that lower complexity thresholds do not further restrict the possible structure.
}

\makepreprint

\section{Introduction}

Symbolic dynamics is concerned with the study of dynamical systems called \textbf{subshifts} consisting of sequences on a finite alphabet. Various results show that subshifts with very low word complexity exhibit `simple' behavior, including extensive work on subshifts generated by substitutions. In particular, for so-called Pisot substitutions, the induced subshifts have natural connections to algebraic number theory. In this work, we demonstrate that these connections can surprisingly be made for all subshifts of low enough word complexity. Our main result (see Section~\ref{int3}) shows that every minimal subshift with word complexity function (see Section~\ref{int1}) of very slow growth is measurably isomorphic to a specific rotation of a compact abelian group called its 
\textbf{maximal equicontinuous factor} or \textbf{MEF} (which can be defined as the character group of its eigenvalue group; see Section~\ref{int2}). 

What's more, the group in question has a very specific structure as the twisted product of an odometer and a compact abelian connected one-dimensional group about which we establish specific number-theoretic properties. %a number-theoretic object known as an adelic $1$-step (i.e. abelian) one-dimensional nilmanifold. 
Section \ref{int6} gives more details about the latter object, but for context we mention that a simple example is irrational rotation of the unit circle 
$S^1 = \mathbb{R} / \mathbb{Z}$, and all of the groups in question can be thought of as rotations on $p$-adic extensions of $S^{1}$.

In this introduction, we will describe our main results and consequences/connections to several disparate areas, including Sarnak's conjecture (see Section~\ref{int3.5}), Pisot's substitution conjecture and so-called S-adic representations of subshifts (see Section~\ref{int4}), and word complexity thresholds (see Section~\ref{int5}). To explain these, we need some brief definitions/background; for full formal definitions, see Section~\ref{defs}.

\subsection{Topological dynamics, subshifts, and word complexity}\label{int1}

A \textbf{topological dynamical system (TDS)} is a pair $(X, T)$ where $X$ is a compact metric space and $T: X \rightarrow X$ is a homeomorphism. A TDS is \textbf{transitive} if there exists $x \in X$ for which $X = \overline{\{T^n x\}}$; every TDS throughout this work will be assumed transitive to avoid degenerate examples such as disjoint unions, for which it is impossible to give a unified structure. A TDS is called \textbf{minimal} if it does not properly contain any nonempty TDS.

A \textbf{subshift} is a TDS given by a finite set $A$ (called the \textbf{alphabet}) and a set $X \subset A^\mathbb{Z}$ which is closed (thereby compact) in the product topology and invariant under all powers of the left shift $\sigma$ (i.e. any shift of a sequence in $X$ must also be in $X$). Every subshift is endowed with the action by the homeomorphism $\sigma$, and so we generally refer to a subshift as $X$ instead of $(X, \sigma)$. 

The \textbf{word complexity function} $p(q)$ of a subshift $X$ simply counts the number of $q$-letter contiguous strings which appear within at least one $x \in X$. For instance, if $X$ is the so-called golden mean shift, consisting of $0$-$1$ sequences without consecutive $1$s, then it's not hard to check that $p(1) = 2$, $p(2) = 3$, $p(3) = 5$, and in fact $p(q)$ is the Fibonacci sequence.

The well-known Morse-Hedlund theorem (\cite{HM}) states that if there exists any $q$ for which $p(q) \leq q$, then $X$ is a finite set of periodic sequences, i.e. any infinite $X$ must have $p(q) \geq q+1$ for all $q$. This means that for infinite $X$, word complexity grows at least linearly. Ferenczi (\cite{ferenczirank}) proved that minimal subshifts with linear complexity have a recursive structure given by a sequence $\tau_k$ of \textbf{substitutions} satisfying specific combinatorial properties (such subshifts are now called \textbf{S-adic}); see Section~\ref{defs} for formal definitions. 

This structure already restricts the dynamical behavior of such a subshift significantly; for instance $X$ must have finite topological rank (\cite{DDMP2}), finitely many ergodic invariant measures (\cite{boshernitzan}, \cite{CK2}, \cite{DMP}), and cannot support a strongly mixing measure (\cite{ferenczirank}). (All $\sigma$-invariant measures on subshifts are assumed to be Borel probability measures.)

In \cite{creutzpavlov}, we showed that any minimal subshift $X$ with $\limsup p(q)/q < 4/3$ has a quite restrictive S-adic structure, which implies (measurable) \textbf{discrete spectrum}, meaning that $X$ is measurably isomorphic to the rotation of a compact abelian group.
In this work, we improve that result in multiple ways. First, we increase the threshold from $4/3$ to $3/2$, which is optimal since \cite{creutzpavlov} also contains an example with $\limsup p(q)/q = 3/2$ which is (measurably) \textbf{weakly mixing}, which is antithetical to discrete spectrum. The main complication in this improvement is that new substitutions can appear above $4/3$; see Corollary~\ref{taus}. Second, we describe the exact group in question (in terms of the S-adic decomposition) and show all measurable eigenvalues are continuous, implying that the canonical continuous projection from $X$ equipped with the shift map to the group rotation is a measure-theoretic isomorphism. (This also implies Weyl/mean equicontinuity; see Theorem 2.1 of \cite{DG}.) This in a sense shows that $X$ is as close as possible to a group rotation (since an infinite subshift cannot be topologically isomorphic to a rotation).  

We note for future reference that $\limsup p(q)/q < 3$ is known to imply uniqueness of invariant measure for minimal subshifts (\cite{boshernitzan}), and so when we refer to `the measure' on such a subshift $X$, there is no ambiguity.

\subsection{Eigenvalues, characters, MEFs, and Sturmian and Toeplitz subshifts}\label{int2}

We say that a continuous function $f : X \to \mathbb{C}$ is a \textbf{continuous eigenfunction} of the TDS $(X, T)$ with \textbf{continuous additive eigenvalue} 
$\gamma$ if $f(Tx) = e^{2\pi i \gamma} f(x)$ for all $x \in X$. The continuous additive eigenvalues form a subgroup $E_{X}$ of $(\mathbb{R}, +)$ containing $\mathbb{Z}$; the \textbf{continuous multiplicative eigenvalues} $\mathcal{E}_{X} = \{ \exp(2\pi i \gamma) : \gamma \in E_{X} \}$ form a subgroup of the unit circle $(S^{1}, \cdot)$ in the complex plane, which is always isomorphic to $E_{X}/\mathbb{Z}$.

For any minimal TDS, its MEF is a rotation of the \textbf{dual group} or \textbf{character group} of the continuous multiplicative eigenvalues, i.e. the group $\widehat{\mathcal{E}_{X}}$ (under pointwise multiplication) of homomorphisms from $\mathcal{E}_{X}$ to $S^{1}$, see e.g.~\cite{DFM} and \cite{MR3160543}.

Two well-studied classes of subshifts with known MEF are the \textbf{Sturmian subshifts} and \textbf{Toeplitz subshifts}. Sturmian subshifts are particularly relevant for our purposes since they are subshifts of minimal word complexity, i.e. $p(q) = q+1$ for all $q$. In addition, any Sturmian subshift $S$ has $E_{S} = \mathbb{Z} \alpha + \mathbb{Z}$ for some $\alpha \notin \mathbb{Q}$, and so 
$\mathcal{E}_S = \exp(2\pi i \mathbb{Z} \alpha) \approx \mathbb{Z} \alpha$, which is isomorphic to $\mathbb{Z}$. The MEF of a Sturmian shift is therefore $\widehat{\mathbb{Z}}$ equipped with the natural action coming from the shift on $X$, which is an (irrational) rotation on the unit circle. 

A subshift $T$ is Toeplitz if it is an almost $1$-$1$ extension of its MEF and has $E_{T}$ a subgroup of $(\mathbb{Q}, +)$ (which must contain $\mathbb{Z}$). Then $\mathcal{E}_{T}$ is isomorphic to $E_T/\mathbb{Z}$, and so
the MEF of $T$ is given by $\widehat{E_T/\mathbb{Z}}$. One way of viewing any such dual group is as an \textbf{odometer}, which is defined as coordinatewise addition by $1$ in an inverse limit of the form $\varprojlim_k \mathbb{Z}/o_1 o_2 \ldots o_k \mathbb{Z}$.
(In a slight abuse of notation, we sometimes also use the term `odometer' to refer to the group itself and not the rotation; since the only rotation ever considered on such groups in this work is the coordinatewise addition by $1$, we hope this will not cause ambiguity.)

Unlike Sturmian subshifts, Toeplitz subshifts need not have low word complexity (it can even grow exponentially) and may have many invariant measures; in fact one of the most celebrated results about this class (\cite{MR1135237}) is that their measure-theoretic structure can be that of an arbitrary Choquet simplex, and so measure-theoretically Toeplitz subshifts are no more restrictive than general topological dynamical systems. 

As will be seen in the next section, we prove that every minimal subshift $X$ with sufficiently low word complexity is in some sense a combination of a (possibly degenerate) Toeplitz subshift and a generalized Sturmian subshift in the sense that its MEF is a (possibly twisted) product of an odometer and a rotation of a compact abelian connected one-dimensional group which factors onto an irrational circle rotation; see Theorem \ref{MEFprod}.  %This is in sharp contrast to the general situation where $X$ equipped with the shift action has no nontrivial relationship to Sturmian nor Toeplitz subshifts, see e.g.~\cite{???}.

\subsection{Our main results}\label{int3}

\begin{maintheorem}\label{main}
Let $X$ be an infinite minimal subshift with $\limsup p(q)/q < 3/2$. Then, if $(M, \eta)$ is the maximal equicontinuous factor of $(X, \sigma)$ and $\phi: (X, \sigma) \rightarrow (M, \eta)$ is the associated factor map, 

\begin{itemize}
\item $\phi$ is a measure-theoretic isomorphism with respect to the unique invariant measures on $(X, \sigma)$ and $(M,\eta)$ (Proposition~\ref{almostoneone}),
\item the additive continuous eigenvalue group of $(X, \sigma)$ is $E_X = \{ q\alpha + \sum \{ \{ q \}_{p} e_{p} \}_{p} : q \in Q \} + R$ for some $\alpha \notin \mathbb{Q}$, $e_{p} \in \mathbb{Q}_{p}$ and 
$Q,R$ subgroups of $\mathbb{Q}$ containing $\mathbb{Z}$ (Theorem~\ref{evalgroup}) where $\{ \cdot \}_{p}$ is the $p$-adic part (see Section \ref{eigs}), and
\item $(M, \eta)$ is isomorphic to a twisted product of 
a (possibly finite) odometer $\mathcal{O}_{X}$ (controlling the rational continuous eigenvalues) and 
a rotation of a compact abelian connected one-dimensional group (i.e., a circle or solenoid) $\mathcal{M}_{X}$
(Theorem~\ref{MEFprod}).
\end{itemize}

In particular, $X$ has measurably discrete spectrum for its unique invariant measure (Theorem~\ref{discspec}), factors onto an irrational circle rotation (Corollary \ref{circlefactor}), and every eigenfunction is continuous (Theorem~\ref{evalgroup}). 

We also show that $X$ is balanced on words, which implies that any two such subshifts are orbit equivalent iff they are strong orbit equivalent iff they have the same additive eigenvalue group.
\end{maintheorem}

For subshifts with well-behaved rational eigenvalue groups, we have the stronger statement that the MEF is a direct product:
\begin{maincorollary}
Let $X$ be an infinite minimal subshift with $\limsup p(q)/q < 3/2$. If there are only finitely many primes $p$ for which there exists $0 < t < \infty$ where $1/p^{t}$ is an eigenvalue but $1/p^{t+1}$ is not (for example when $X$ is totally ergodic) then the MEF is the direct product of the odometer $\mathcal{O}_{X}$ and the compact abelian connected one-dimensional group $\mathcal{M}_{X}$.
\end{maincorollary}

We also show that further restrictions on word complexity do not restrict the structure of the MEF as a twisted product:

\begin{maintheorem}\label{main2}
For every odometer $\mathcal{O}$ and every group $\mathcal{M}$ which is the dual of any subgroup of $\mathbb{Q}$ there exists a minimal subshift with 
$\lim p(q)/q = 1$ for which the MEF is the twisted product of $\mathcal{O}$ and a rotation on $\mathcal{M}$ (Theorem~\ref{nobetter}). In fact, $\limsup p(q)/q$ may take any prescribed value in $[1, 3/2)$ (except when $\mathcal{O}$ is finite and $\mathcal{M}$ is a finite extension of a circle in which case $\lim p(q)/q = 1$ is forced).
\end{maintheorem}

We mention that we do not know what rotations on such groups can occur as the action on the MEF, specifically we do not know what values of $(\alpha,(e_{p})$ nor which twisting cocycles can occur in such examples. However, in the concrete cases most of interest we can exhibit subshifts where the MEF is a direct product:

\begin{maincorollary}
For every odometer $\mathcal{O}$ for which only finitely many primes $p$ appear with bounded nonzero degree in the construction and every  $\mathcal{M}$ which is the dual of any subgroup of $\mathbb{Q}$ there exists a minimal subshift with 
$\lim p(q)/q = 1$ for which the MEF is the direct product of $\mathcal{O}$ and a rotation on $\mathcal{M}$. In fact, $\limsup p(q)/q$ may take any prescribed value in $[1, 3/2)$ (except when $\mathcal{O}$ is finite and $\mathcal{M}$ is a finite extension of a circle in which case $\lim p(q)/q = 1$ is forced).
\end{maincorollary}

%In addition, we prove (Theorem~\ref{nobetter}) that for every adelic one-step nilsystem $M$ and odometer $\mathcal{O}$, there exists a minimal subshift $X$ with $\lim p(q)/q = 1$ such that $X$ has MEF as in Theorem~\ref{main} with $\mathcal{O}_{X} = \mathcal{O}$ and $M_{X} = M$.

\subsection{The Sarnak conjecture}\label{int3.5}

The celebrated \textbf{Sarnak conjecture} states that for any zero entropy TDS $(X,T)$, any $f \in C(X)$, and any $x \in X$,  
\begin{equation}\label{sar}
\frac{1}{N} \sum_{n = 0}^{N-1} f(T^n x) \mu(n) \rightarrow 0,
\end{equation}
where $\mu$ is the M\"{o}bius function. A simple application of Theorem~\ref{main} is the following.

\begin{corollary}\label{sarcor}
For any subshift $(X,\sigma)$ with $\limsup p(q)/q < 3/2$, any $f \in C(X)$, and any $x \in X$, (\ref{sar}) holds.
\end{corollary}

\begin{proof}
Assume that $(X,\sigma)$ has $\limsup p(q)/q < 3/2$, and consider $f \in C(X)$ and $x \in X$. The subsystem 
$Y = \overline{\{\sigma^n x\}}$ is transitive by definition. If $Y$ is finite, then it is a finite union of periodic orbits, and it is a simple consequence of the Prime Number Theorem that (\ref{sar}) holds for $x$ in this case. If $Y$ is infinite, transitive, and non-minimal, then by \cite{ormespavlov}, $x$ is eventually periodic in both directions, in which case (\ref{sar}) holds for $x$ for the same reason. 

Finally, if $Y$ is infinite and minimal, then by Theorem~\ref{main}, it is uniquely ergodic with unique invariant measure having discrete spectrum. By Theorem 1.2 of \cite{HWY}, Sarnak's conjecture holds for such a system.
\end{proof}

Corollary~\ref{sarcor} is, to our knowledge, the first result to show that sufficiently low word complexity (which is just a stronger version of zero entropy) implies the conclusion of Sarnak's conjecture. 

\begin{remark}
We note that the previously mentioned \cite{HWY} proves that a different sort of low complexity implies discrete spectrum, and so also implies Sarnak's conjecture. Their property is called \textbf{measure complexity}, and is a bit too long to define fully here. Roughly, for any $n$, they look at the minimum size of a collection of points $S$ so that most $x \in X$ (in the sense of the invariant measure) maintain a small average distance from some $s \in S$ over the first $n$ iterates. We do not believe that our result implies theirs or the converse; their measure complexity is clearly bounded from above by word complexity, but can be much smaller. In fact, their required upper bound is that measure complexity grows more slowly than $n^{\epsilon}$ for all $\epsilon > 0$, and their proof actually shows that discrete spectrum implies that measure complexity is bounded for some metric.
\end{remark}

\subsection{Adelic groups, nilmanifolds, and solenoids}\label{int6}

The MEF of a low-complexity minimal subshift can be interpreted in purely algebraic terms as a group rotation, a viewpoint which relates to both class field theory and Lie theory.  Since $\alpha \mathbb{Z} + \mathbb{Z} \subseteq E_{X} \subseteq \alpha \mathbb{Q} + \mathbb{Q}$ for the natural rotation by (irrational) $\alpha$ on $S^{1}$, it is necessary to establish the general structure of duals of arbitrary subgroups of $\mathbb{Q}$. The classical theory \cite{conrad} is the study of the ring of adeles modulo $\mathbb{Q}$ as the dual of $\mathbb{Q}$ itself.  %%https://kconrad.math.uconn.edu/blurbs/gradnumthy/characterQ.pdf

The \textbf{ring of adeles} $\mathbb{A}$ over $\mathbb{Q}$ is $\mathbb{R} \times \prod_{p} \mathbb{Q}_{p}$ where $p$ ranges over the primes and $\mathbb{Q}_{p}$ are the $p$-adic numbers, restricted to elements where all but finitely many terms lie in $\mathbb{Z}_{p}$. The field $\mathbb{Q}$ sits naturally as a lattice (discrete co-compact subgroup) diagonally in $\mathbb{A}$, and its character group is $\widehat{\mathbb{Q}} = \mathbb{A} / \mathbb{Q}$. 

The eigenvalue groups of the low complexity subshifts in Theorem~\ref{main} involve arbitrary subgroups of $\mathbb{Q}$ containing $\mathbb{Z}$, and describing the MEF via their character groups requires more refined techniques to identify the duals of arbitrary subgroups of $\mathbb{Q}$. It's not hard to check that such subgroups are in one-to-one correspondence with sequences $(\ell_p)$ in $\mathbb{Z}_{\geq 0} \cup \{ \infty \}$ indexed by primes $p$, where $Q_{(\ell_p)} := \{ \frac{m}{n} : n = \prod_{p} p^{t_{p}}~\text{such that}~0\leq t \leq \ell_{p} \}$. This case (where infinitely many $\ell_p$ are allowed to be nonzero) is often called the \textbf{adelic} case in the literature. 

Adapting relevant proofs, treating solenoids as quotients of $p$-adic groups by lattices, to the adelic setting requires a bit of care since $Q_{(\ell_p)}$ generally does not sit as a lattice in $\mathbb{A}$ (being of infinite covolume). We can define a natural substitute $\mathbb{A}_{(\ell_p)}$ for $\mathbb{A}$ and lattice $\Gamma_{(\ell_{p})}$ in $\mathbb{A}_{(\ell_{p})}$, and we then verify that $\widehat{Q_{(\ell_p)}}$ is the (fully nonsplit) central extension of $\mathbb{A}_{(\ell_{p})} / \Gamma_{(\ell_{p})}$ by a compact abelian group $K_{(\ell_{p})}$. This also explains why adelic subgroups arise naturally in connection to odometers (Proposition \ref{odo}): for any odometer $\mathcal{O} = \varprojlim \mathbb{Z}/o_{k}\mathbb{Z}$, if $\ell_{p} = \sup \{ t : p^{t}~\text{divides}~o_{k}~\text{for some $k$} \}$, then $\mathcal{O} \simeq \widehat{Q_{(\ell_p)} / \mathbb{Z}}$. 

More generally, such a quotient $G/\Gamma$ for a nilpotent Lie group $G$ (above, $\mathbb{A}_{(\ell_p)}$) and lattice $\Gamma$ is called a \textbf{nilmanifold}, introduced by Mal'cev \cite{malcev}. When $G$ is a $1$-step nilpotent group (i.e., abelian), a nilmanifold is also a compact abelian group. However, abelian nilmanifolds are a strict subclass of compact abelian groups, since not all groups have the necessary Lie structure. (Any compact abelian group rotation, however, can be viewed as an inverse limit of rotations on abelian nilmanifolds.) Theorem~\ref{main} shows that for subshifts of low word complexity, the character groups in question cannot be purely $p$-adic, i.e. must have nontrivial real component. The phenomenon of nontrivial real component being more `natural' in the study of Lie groups/lattices is not new; one example is the generalization of the Margulis Arithmeticity Theorem (\cite{Margulis91}) proved in (\cite{MR1872530}).

Our Lie group $G$ can contain $p$-adic parts for infinitely many primes, and so this case could be referred to as an \textbf{adelic nilmanifold}. This is itself a generalization of the so-called `$S$-adic' theory of nilmanifolds ($S$ being a finite set of primes), in which $p$-adic parts can exist for finitely many primes, and which was studied for instance in \cite{bekka} and in \cite{MR3627126}.% in connection with solenoids.

The quotient group yielding the MEF in our setting is one-dimensional, and must be the (twisted) product of an odometer (which may be degenerate) with either the circle $S^1 := \mathbb{R}/\mathbb{Z}$ or a non-degenerate real adelic group. Such a group is well-studied in dynamical systems as a so-called (Vietoris) \textbf{solenoid}; we illustrate this with a (relatively) simple example.

\begin{example}\label{e1}

If the additive eigenvalue group is $\mathbb{Z} + \mathbb{Z}[\nicefrac{1}{2}] \alpha $ for an irrational $\alpha$, the corresponding MEF is given by the dual of $\mathbb{Z}[\nicefrac{1}{2}]\alpha$, which is the group $(\mathbb{R} \times \mathbb{Q}_{2})/\mathbb{Z}[\nicefrac{1}{2}]$ under a rotation by a group element. This group is sometimes called the dyadic solenoid in dynamical systems, and has several different equivalent definitions. The most natural is as the inverse limit of an inverse system of circles under the bonding maps $x \mapsto 2x \pmod 1$. Another is via the fundamental domain $\mathcal{M}_{2} = S^1 \times \mathbb{Z}_{2}$ (here and elsewhere, $\mathbb{Z}_{2}$ are the $2$-adic integers). The restriction of the specific rotation $(\alpha,0)$ to this domain is given by the following action:
\[
(\theta,z) \mapsto \left\{ \begin{array}{ll} (\theta+\alpha,z) &\text{if $\theta + \alpha < 1$} \\
(\theta+\alpha-1,z+1) &\text{otherwise} \end{array}\right.
\]
The rotational component on $S^1$ is just rotation, and the $\mathbb{Z}_2$ component is also a rotation, dependent on the value of $\theta \in S^1$. This type of map is called a \textbf{skew product}, and is another common interpretation of the solenoid in dynamical systems.

%This action is not a rotation on $\mathcal{M}_2$ viewed as a product group, but it can be viewed as the restriction of the rotation by $(\alpha, a_2)$ on $(\mathbb{R} \times \mathbb{Q}_{2})/\mathbb{Z}[\nicefrac{1}{2}]$ to its natural fundamental domain $\mathcal{M}_2 = S^1 \times \mathbb{Z}_{2}$. 

%We note that the projection to the real coordinate is precisely the factor map onto $S^{1}$ under rotation by $\alpha$.

\end{example} 

%\subsection{Nilmanifolds and nilsystems}\label{int62}

%Such rotations are called $1$-step nilsystems; more generally, nilsystems are actions on higher-rank nilmanifolds, which may not even be group rotations, and have been used in breakthrough work by Host-Kra \cite{MR2150389} and Bergelson-Host-Kra \cite{MR2138068} to prove convergence of nonconventional ergodic averages. Since then, they have proved invaluable in ergodic theory and dynamical systems (e.g.~\cite{MR2122919}, \cite{MR2257397}, \cite{MR2950773}, \cite{MR2901353},  \cite{MR2912715}, \cite{MR3728628}, \cite{MR3839640}, \cite{MR3866908}). 

%Nilsystems are often thought of as the `simplest' type of dynamical system. Given this heuristic, the phenomenon that some irrational eigenvalues yield lower complexity than rational alone makes sense; when all eigenvalues are rational, the character group is purely $p$-adic, and so cannot have nilmanifold structure.

%{In analogy with nilpotent groups, $1$-step nilsystems are abelian, and all examples in this work will be $1$-step.} However, not every compact abelian group rotation (such as the MEF of a TDS) is a nilsystem, though all are inverse limits of $1$-step nilsystems. However, in our case we prove that the MEF is itself a ($1$-step) one-dimensional nilsystem. Given the heuristic that nilsystems are the `simplest' type of dynamical systems, the phenomenon that irrational eigenvalues yield lower complexity than rational alone makes sense; when all eigenvalues are rational, the character group is purely $p$-adic, and so cannot have nilmanifold structure.

\subsection{Substitutions and the Pisot conjecture}\label{int4}

As mentioned earlier, Ferenczi proved that all minimal subshifts of linear complexity have a so-called S-adic structure (`S' standing for `substitution'), meaning that all $x \in X$ have a recursive structure coming from a sequence $(\rho_k)$ of substitutions with specific combinatorial properties from a finite set. As was done in \cite{creutzpavlov}, the main component of the proof of Theorem~\ref{main} is a proof (Corollary~\ref{taus}) that low word complexity implies a very special type of S-adic structure, where all substitutions (denoted by $\tau_{m_k,n_k,r_k}$ in the proofs) have a very particular form.

Connections between substitutive structure and discrete spectrum have been known for many years, and the most famous such connection is the so-called \textbf{Pisot conjecture}. A full treatment is beyond our scope here, but informally it states that if $X$ is defined by a single substitution (i.e. all $\rho_k$ are the same in the description above) and if that substitution has the Pisot property (this means that its associated adjacency matrix has Perron eigenvalue which is a Pisot-Vijarayaghavan number) and is algebraically irreducible, then $X$ has measurably discrete spectrum. The conjecture remains open, though there has been substantial progress, including a complete solution for $X$ for alphabet of size $2$ (\cite{bargediamond}, \cite{2pisot}). 

Much more difficult is the general S-adic case; even finding a proper plausible formulation seems quite difficult. There have been multiple impressive recent results in this direction, including a version of S-adic Pisot for two-letter alphabets (\cite{sadic2pisot}). However, their result includes several hypotheses which cannot hold even for all Sturmian subshifts, most notably recurrence, which means that for every $m$, there exists $n$ so that $\rho_k = \rho_{n+k}$ for $1 \leq k \leq m$. In particular, their results seemingly cannot be used to verify discrete spectrum under any complexity hypothesis alone. 

Another hypothesis required for previous versions of the S-adic Pisot conjecture is that such subshifts are \textbf{balanced on words}, see Section~\ref{defs} for a definition. This property is often difficult to verify, but we do so (Theorem~\ref{balanced}) in the course of finding the \textbf{dimension group} for minimal subshifts of low word complexity (Theorem~\ref{dimgroup}), which can be used to characterize \textbf{orbit equivalence} and \textbf{strong orbit equivalence} for these shifts.

Our proof of Theorem~\ref{main} can then be, at least in part, thought of as a direct verification of a version of the S-adic Pisot conjecture for only sequences of substitutions which can appear in our S-adic decomposition (a complete characterization of which sequences are possible seems quite hard, but for a sufficient condition see Corollary~\ref{sum2}). In fact, the eigenvalue group and the abelian nilmanifold $\mathcal{M}_X$ and odometer $\mathcal{O}_X$ appearing in the MEF can be explicitly defined in terms of these substitutions; this is too technical to describe here, but is done in Theorems~\ref{evalgroup} and \ref{MEFprod}. 

As mentioned earlier, all two-letter Pisot substitutions are known to have measurable discrete spectrum (\cite{2pisot}), but the original proofs do not give an explicit description of the MEF. There are some geometric sufficient conditions which yield a description of the MEF as a domain exchange, similar to the classical Rauzy fractal (see \cite{siegel}, \cite{bbk}, 
\cite{MinThus}). Our results describe the group in terms of Pontryagin duality, and and define some simple new examples (including traditional substitutions rather than S-adic) with prescribed MEF. (See Section~\ref{all} for proofs.)

\begin{example}\label{ex4}
If $\rho$ is the substitution on $\{0,1\}$ defined by $\rho(0) = 001$ and $\rho(1) = 00001$, and we define 
$X = \overline{\{\sigma^n x\}}$ by $x = \lim_k \rho^k(0)$, then its MEF is a rotation of the dyadic solenoid from Example~\ref{e1} (meaning in the above context that $\mathcal{M}_{X}$ is the dyadic solenoid and $\mathcal{O}_{X}$ is trivial). We emphasize that here, the additive eigenvalue group is not generated by $\alpha/2^n$ as in Example~\ref{e1}, but by a more complicated
set $(\alpha + k_n)/2^n$, which induces a rotation by $(\alpha, e_2)$ for a nonzero $e_2 \in \mathbb{Q}_2$.

We also note that this particular example was specifically given by Siegel in \cite{siegel}. 
\end{example}

\begin{example}\label{ex5}
If $\rho$ is the substitution on $\{0,1\}$ defined by $\rho(0) = 00000011$ and $\rho(1) = 0000000011$, and we define $X = \overline{\{\sigma^n x\}}$ by $x = \lim_k \rho^k(0)$, then its 
 MEF is a rotation of the product of the dyadic odometer with the dyadic solenoid, i.e.~$\mathcal{M}_{X}$ is again the dyadic solenoid and $\mathcal{O}_{X}$ is the dyadic odometer.
\end{example}

\begin{example}\label{ex6}
Let $\pi$ be the substitution on $\{0,1\}$ defined by $\pi(0) = a$ and $\pi(1) = ab$, let $\omega_1$ be the substitution on $\{0,1\}$ defined by $\omega_1(0) = 001$ and $\omega_1(1) = 00001$, and let $\omega_2$ be the substitution on $\{0,1\}$ defined by $\omega_2(0) = 00001$ and $\omega_2(1) = 0000001$. Define a sequence of substitutions $\rho_k \in \{\omega_1, \omega_2\}$ by 
$\rho_k = \omega_1$ if $2^{k+2}$ divides the length of $(\pi \circ \rho_0 \circ \cdots \circ \rho_{k-1})(1)$, and $\omega_2$ otherwise. (For instance, $\rho_0$ is $\omega_2$ since $2^2$ does not divide the length $2$ of $\pi(1) = ab$, and 
$\rho_1$ is $\omega_1$ since $2^3$ does divide the length $8$ of $(\pi \circ \rho_0)(1) = aaaaaaab$.)
If we define $X = \overline{\{\sigma^n x\}}$ by $x = \lim_k (\pi \circ \rho_0 \circ \cdots \circ \rho_k)(0)$, then 
its MEF is the product of the dyadic odometer and an irrational circle rotation, i.e.~$\mathcal{M}_{X}$ is $S^{1}$ and $\mathcal{O}_{X}$ is the dyadic odometer.
\end{example}

All of these examples satisfy $\limsup p(q)/q < 3/2$ by Corollaries ~\ref{sum} and ~\ref{sum2}, and for Example~\ref{ex4} we explicitly calculate this limit to be $\frac{105+\sqrt{17}}{86} \approx 1.269$. We do not provide exact values for the other two examples since the calculations would be quite technical.

We note that the weak mixing example of \cite{creutzpavlov} in fact was generated by a sequence of Pisot substitutions, and so each individual substitution being Pisot is (unsurprisingly) not enough to guarantee discrete spectrum. This was not the first such weak mixing example; \cite{cassetal} contains one on a three-letter alphabet involving the so-called Arnoux-Rauzy substitutions (and word complexity $p(n) = 2n$). We are not, however, aware of an earlier example with a two-letter alphabet.

\subsection{\texorpdfstring{\nicefrac{3}{2}}{3/2} as a threshold}\label{int5}

Several recent works have demonstrated that $\limsup p(q)/q = 3/2$ is an important threshold for several different types of dynamical properties. First, Theorems 1.2 and 1.3 of \cite{ormespavlov} imply that if a subshift $X$ is transitive and nonminimal and has $\limsup p(q)/q < 3/2$, then it is the orbit closure of a sequence which is eventually periodic in both directions. (In particular, this means that Theorem~\ref{main} automatically applies to all transitive shifts not of this degenerate form.) We can rewrite as the following threshold result.

\begin{theorem}
\[
3/2 = \min\{\limsup p(q)/q \ : \ X\textrm{ is transitive, not minimal, and contains a non-eventually periodic sequence}\}.
\]
\end{theorem}

In \cite{Creutz2022b}, it was shown (Theorem C) that every aperiodic rank-one subshift satisfies $\limsup p(q)/q \geq 3/2$, and an example was given there (Theorem D) of an aperiodic rank-one subshift with $\limsup p(q)/q = 3/2$ (a rank-one subshift is a subshift on two letters defined as the orbit closure of the limit of words $B_{n}$ such that $B_{0} = 0$ and $B_{n+1}$ is a concatenation of $B_{n}$ and $1$, equivalently a subshift measure-theoretically isomorphic to a rank-one transformation in the cutting-and-stacking sense). This immediately implies the following.

\begin{corollary}
\[
3/2 = \min\{\limsup p(q)/q \ : \ X\textrm{ is an aperiodic rank-one subshift}\}.
\]
\end{corollary}

Theorem~\ref{main} implies similar results for different dynamical properties. Theorem~\ref{main}, when combined with the weakly mixing example from \cite{creutzpavlov} with $\limsup p(q)/q = 3/2$ mentioned above, yields the following result, which shows that any bound on $\limsup p(q)/q$ which implies existence of eigenvalues automatically implies discrete spectrum.

\begin{corollary}\label{maincor1}
\begin{multline*}
3/2 = \min\{\limsup p(q)/q \ : \ X \textrm{ is minimal and does not have discrete spectrum}\} \\
= \min\{\limsup p(q)/q \ : \ X\textrm{ is minimal and has a nontrivial weakly mixing measure}\}.
\end{multline*}
\end{corollary}

Surprisingly, the same number is also the complexity threshold for Toeplitz subshifts. Our results already show that $\limsup p(q)/q < 3/2$ precludes $X$ being Toeplitz; all Toeplitz shifts are minimal, and have no irrational continuous eigenvalues, so cannot have the structure of Theorem~\ref{main}. In the other direction, \cite{MR4092862} gives word complexity estimates for a subclass called simple Toeplitz subshifts, and those estimates show that there exist simple Toeplitz subshifts with $\limsup p(q)/q = 3/2$ (this happens whenever the parameter sequence $(n_k)$ from that paper is unbounded). We now have the following.

\begin{theorem}\label{maincor2}
\[
3/2 = \min\{\limsup p(q)/q \ : \ X\textrm{ is Toeplitz}\}.
\]
\end{theorem}

A nearly identical proof shows that $3/2$ is a threshold for irrational continuous eigenvalues.

\begin{corollary}
\[
3/2 = \min\{\limsup p(q)/q \ : \ X\textrm{ is minimal, infinite, and has no irrational continuous eigenvalue}\}.
\]
\end{corollary}

\subsection{Summary}

Section~\ref{defs} contains definitions not fully given in the introduction. In Section~\ref{subs}, we describe and prove the S-adic structure for minimal subshifts of low complexity. Sections~\ref{discspecsec}, \ref{eigs}, \ref{adelicsection}, \ref{MEF}, and \ref{all}
contain, respectively, proofs of discrete spectrum and balancedness, the eigenvalue group, general proofs about dual groups and adelic theory, the structure of the MEF, and realization results.

\section*{Acknowledgements} 
The authors gratefully acknowledge the extremely useful reports from the anonymous referees, which helped greatly in improving this manuscript.
We also acknowledge the usage of ChatGPT-6 Astra, in conjunction with the referees' comments, in the preparation of the final revisions of this paper, for proofreading, clarity, and checking proofs. All writing in the article is by the authors, who take full responsibility for the mathematical correctness and final text of the published article.

\section{Definitions}\label{defs}

Let $\mathcal{A}$ be a finite subset of $\mathbb{Z}$; the \textbf{full shift} is the set $\mathcal{A}^\mathbb{Z}$ equipped with the product topology and $\sigma$ is the left shift homeomorphism on $\mathcal{A}^\mathbb{Z}$. 
A \textbf{subshift} is a closed $\sigma$-invariant subset $X \subset \mathcal{A}^\mathbb{Z}$. The \textbf{orbit} of $x \in X$ is the 
set $\{\sigma^n x\}_{n \in \mathbb{Z}}$.
In a slight abuse of notation, we sometimes define $X$ as the orbit closure of a one-sided sequence $y \in \mathcal{A}^\mathbb{N}$; this can be interpreted in the obvious way using natural extensions. 

A \textbf{word} is any element of $\mathcal{A}^n$ for some $n \in \mathbb{N}$, referred to as its \textbf{length} and denoted by
$\len{w}$. For any  
word $v$, the number of occurrences of $v$ as a subword of $w$ is denoted $|w|_{v}$.  We say $\mathcal{A}^* = \bigcup_{n \geq 1} \mathcal{A}^n$ and represent the concatenation of words $w_1, w_2, \ldots, w_n$ by $w_1 w_2 \ldots w_n$.

The \textbf{language} of a subshift $X$ on $\mathcal{A}$, denoted $L(X)$, is the set of all finite words appearing as subwords of points in $X$. For any $q \in \mathbb{N}$, we denote $L_q(X) = L(X) \cap \mathcal{A}^q$, the set of $q$-letter words in $L(X)$, and define the \textbf{word complexity function} of $X$ to be $p(q) := |L_{q}(X)|$. 
For a subshift $X$ and a word $w \in L(X)$ we denote by $[w]$ the clopen subset in $X$ consisting of all $x \in X$ such that $x_{0}\ldots x_{|w|-1} = w$.

A \textbf{substitution} (sometimes called a morphism) is a map 
$\tau: \mathcal{A} \rightarrow \mathcal{B}^*$ for finite alphabets $\mathcal{A}$ and $\mathcal{B}$. Substitutions can be composed when viewed as homomorphisms on the monoid of words under concatenation, i.e. if $\tau: \mathcal{A} \rightarrow \mathcal{B}^*$ and
$\pi : \mathcal{B} \rightarrow \mathcal{C}^*$, then $\pi \circ \tau: \mathcal{A} \rightarrow \mathcal{C}^*$ can be defined by
$(\pi \circ \tau)(a) = \pi(b_1) \pi(b_2) \ldots \pi(b_k)$, where $\tau(a) = b_1 \ldots b_k$.
When a sequence of substitutions $\tau_k: \mathcal{A} \rightarrow \mathcal{A}^*$ shares the same alphabet, and when there exists $a \in \mathcal{A}$ for which $\tau_k(a)$ begins with $a$ for all $k$, clearly
$(\tau_1 \circ \cdots \circ \tau_k)(a)$ is a prefix of $(\tau_1 \circ \cdots \circ \tau_{k+1})(a)$ for all $k$. In this situation one may then speak of the (right-infinite) limit of $(\tau_1 \circ \cdots \circ \tau_k)(a)$.

For any subshift $X$, there is a convenient way to represent the $n$-language and possible transitions between words in points of $X$ by the \textbf{Rauzy graphs}: the $n$th Rauzy graph of $X$ is the directed graph $G_{X, n}$ with vertex
set $L_{n}(X)$ and directed edges from $w_1 \ldots w_{n}$ to $w_2 \ldots w_{n+1}$ for all $w_1 \ldots w_{n+1} \in L_{n+1}(X)$. 
Then, a vertex with multiple outgoing edges corresponds to a word $w = w_1 \ldots w_n$ which is \textbf{right-special}, meaning that there exist letters $a \neq b$ for which $wa, wb \in L(X)$. \textbf{Left-special} words are defined similarly, and a word is \textbf{bi-special} if it is both left- and right-special. The reader is referred to Section 1 of \cite{creutzpavlov} for more details.

We will sometimes endow a subshift $X$ with a measure $\mu$; any such $\mu$ is understood to be a Borel probability measure which is invariant under $\sigma$. 
A subshift has \textbf{(measurably) discrete spectrum} with respect to a measure $\mu$ when the eigenfunctions span $L^{2}(X,\mu)$.
It is well-known, e.g.~(\cite{MR648108}, Theorem 3.4), that:
\begin{theorem}\label{isochars}
An ergodic transformation on a standard probability space with discrete spectrum is measure-theoretically isomorphic to the space of characters of its (multiplicative) eigenvalue group, endowed with the Haar measure, under the ergodic ``rotation'' of multiplication by the identity homomorphism, i.e.~multiplication by the character $\iota : \mathcal{E}_{X} \to S^{1}$ given by $\iota(\gamma) = \gamma$ for all $\gamma$.
\end{theorem}

A subshift $X$ is \textbf{balanced on words} when for every word $v \in \mathcal{L}(X)$, there exists $C_{v} > 0$ such that for any $w,w^{\prime} \in \mathcal{L}(X)$ with $\len{w} = \len{w^{\prime}}$, $||w|_{v} - |w^{\prime}|_{v}| < C_{v}$, i.e.~the number of occurrences of $v$ in any two words of the same length differs by less than $C_{v}$. We say that $X$ is \textbf{balanced on letters} if the above holds whenever $v$ has length $1$.

\section{Substitutive structure of minimal subshifts with low complexity}\label{subs}

This section is devoted to establishing the following proposition, which establishes a substitutive structure for minimal subshifts with complexity $\limsup \frac{p(q)}{q} < 1.5$.

\begin{proposition}\label{words}
Let $X$ be an infinite minimal subshift with $\limsup \frac{p(q)}{q} < 1.5$.  Then there exist $\nu > 0$ and words $u_{k}$ and $v_{k}$ for $k \geq 0$ such that, writing $p_{k}$ for the maximal common prefix of $u_{k}$ and $v_{k}$ and $s_{k}$ for the maximal common suffix of $v_{k}^{\infty}$ and $v_{k}^{\infty}u_{k}$, the following hold:
\begin{itemize}
\item every $x \in X$ is uniquely decomposable as a concatenation of $v_{k}$ and $u_{k}$;
\item $\len{v_{k}} < \len{u_{k}}$ and $v_{k}$ is a suffix of $u_{k}$;
\item $\len{p_{k}} + \len{s_{k}} < \len{v_{k}} + \len{u_{k}}$;
\item for $k \geq 1$, also $\len{p_{k}} + \len{s_{k}} < 2\len{v_{k}} + \len{v_{k-1}} < 3\len{v_{k}}$; and
\item $p(q) < (1.5-\nu)q$ for all $q \geq \len{v_{0}}$.
\end{itemize}
For each $k$, exactly one of the following holds:
\begin{itemize}
\item there exist positive integers $m_{k} < n_{k}$ such that \[ v_{k+1} = v_{k}^{m_{k}-1}u_{k} \quad\text{and}\quad u_{k+1} = v_{k}^{n_{k}-1}u_{k}; \text{or} \]
\item there exist positive integers $r_{k} < m_{k} < n_{k}$ such that \[ v_{k+1} = v_{k}^{m_{k}-1}u_{k}v_{k}^{r_{k}-1}u_{k} \quad\text{and}\quad u_{k+1} = v_{k}^{n_{k}-1}u_{k}v_{k}^{r_{k}-1}u_{k}. \]
\end{itemize}
\end{proposition}

\begin{notation}
For $k$ such that $r_{k}$ does not exist, set $r_{k} = 0$ and for all $k$ set
\[
\bbone_{r_{k}} :=  \left\{ \begin{array}{ll} 1 & r_{k} > 0 \\ 0 & r_{k} = 0 \end{array} \right. .
\]
\end{notation}

The substitutive structure can be explicitly stated as follows:

\begin{corollary}\label{taus}
For integers $0 \leq r < m < n$, define the substitutions $\tau_{m,n,r} :  \{ 0, 1 \} \to \{ 0, 1 \}^{*}$ by
\begin{align*}
\tau_{m,n,0} &:  \left\{ \begin{array}{lr} 0 \mapsto &0^{m-1}1 \\ 1 \mapsto &0^{n-m}0^{m-1}1 \end{array} \right. \quad\quad
\tau_{m,n,r} :  \left\{ \begin{array}{lr} 0 \mapsto &0^{m-1}10^{r-1}1 \\ 1 \mapsto &0^{n-m}0^{m-1}10^{r-1}1 \end{array} \right.~\text{when $r > 0$}.
\end{align*}
Then $X$ is the orbit closure of $\lim \pi \circ \tau_{m_{0},n_{0},r_{0}} \circ \cdots \circ \tau_{m_{k},n_{k},r_{k}}(0)$ for $\pi : \{ 0, 1 \} \to \mathcal{A}^{*}$ for some finite $\mathcal{A}$.
\end{corollary}

\begin{remark}
The substitutions $\tau_{m,n,r}$ are right-proper (meaning that the final letter of $\tau_{m,n,r}(a)$ is constant over all letters $a$), which will later be useful for proving that $X$ is balanced on words. We also remark that in \cite{creutzpavlov}, we derived a substitutive structure when $\limsup \frac{p(q)}{q} < 4/3$ involving only the substitutions $\tau_{m,n,0}$.

It is useful to consider the simplest possible case of such a decomposition, which is $r = 0$ for all $k$ and $n_k = m_k + 1$ for all $k$, which yields the Sturmian subshift with rotation number $\alpha = [m_1, m_2, \ldots]$. Informally, we can say that the 
$(m_k)$ and $(n_k)$ together determine the irrational circle rotation $\alpha$ which $X$ factors onto (see (\ref{ab}) and Theorem~\ref{evalgroup}), and $n_k$ large in comparison to $m_k$ and nonzero values of $r_k$ yield additional structure of the group beyond the circle, which can manifest either as rational multiples of $\alpha$ (solenoidal behavior) or rational numbers (odometer behavior).  
\end{remark}

\begin{proof}
Define $\pi(0) = v_{0}$ and $\pi(1) = u_{0}$.  Write $\xi_{k} = \pi \circ \tau_{m_{0},n_{0},r_{0}} \circ \cdots \tau_{m_{k-1},n_{k-1},r_{k-1}}$.

It follows immediately from Proposition \ref{words} that if $v_{k} = \xi_{k}(0)$ and $u_{k} = \xi_{k}(1)$ then, when $r_{k} > 0$,
\[
v_{k+1} = v_{k}^{m_{k}-1}u_{k}v_{k}^{r_{k}-1}u_{k}
= (\xi_{k}(0))^{m_{k}-1}\xi_{k}(1)(\xi_{k}(0))^{r_{k}-1}\xi_{k}(1)
= \xi_{k} \circ \tau_{m_{k},n_{k},r_{k}}(0) = \xi_{k+1}(0)
\]
and similarly for $u_{k+1}$.  Similar reasoning applies when $r_{k}=0$.  The claim then follows by induction.
\end{proof}

We first collect several basic facts established in previous work of the authors, though we note that these lemmas could alternately be derived from Fine and Wilf's Theorem. For a general introduction to this style of word-combinatorial argument, see \cite{lothaire}. 

\begin{definition}
A word $v$ is a \textbf{root} of $w$ if $\len{v} \leq \len{w}$ and $w$ is a suffix of the left-infinite word $v^{\infty}$.
\end{definition}

\begin{lemma}[\cite{Creutz2022} Lemma 5.7]\label{B}
If $w$ and $v$ are words with $\len{v} \leq \len{w}$ such that $wv$ has $w$ as a suffix then $v$ is a root of $w$.
\end{lemma}

\begin{lemma}[\cite{creutzpavlov} Lemma 2.5]\label{s}
Let $u$ and $v$ be words with $\len{v} < \len{u}$ and let $s$ be the maximal common suffix of $v^{\infty}$ and $v^{\infty}u$. If $\len{s} \geq \len{vu}$ then $u$ and $v$ are multiples (powers) of the same word.
\end{lemma}

\begin{lemma}[\cite{creutzpavlov} Lemma 2.6]\label{s2}
Let $v$ and $u$ be words with $\len{v} < \len{u}$ which are not multiples of the same word and where $v$ is a suffix of $u$.  Let $s$ be the maximal common suffix of $v^{\infty}$ and $v^{\infty}u$.  Then $s$ is a suffix of any left-infinite concatenation of $u$ and $v$.
\end{lemma}

\begin{lemma}[\cite{creutzpavlov} Lemma 2.7]\label{s3}
Let $v$ and $u$ be words and $s$ be the maximal common suffix of $v^{\infty}$ and $v^{\infty}u$.  Let $y$ and $z$ be suffixes of some (possibly distinct) concatenations of $u$ and $v$, both of length at least $\len{s}$.  Then for any word $w$, the maximal common suffix of $yvw$ and $zuw$ is $sw$.
\end{lemma}

\begin{lemma}[\cite{creutzpavlov}, previously proved in \cite{CN10}]\label{RSlem}
Let $X$ be a subshift on alphabet $\mathcal{A}$, for all $n$ let $RS_n(X)$ denote the set of right-special words of length $n$ in the language of $X$, and for all right-special $w$, let
$F(w)$ denote the set of letters which can follow $w$, i.e. $\{a \in \mathcal{A} \ : \ wa \in L(X)\}$. Then, for all $q > r$,
\[
p(q) = p(r) + \sum_{i = r}^{q-1} \sum_{w \in RS_i(X)} (|F(w)| - 1).
\]
\end{lemma}

\begin{lemma}\label{2rs}
Let $w$ and $y$ be right-special words with $\len{w} \leq \len{y}$ and maximal common suffix $s$.  Then
\[
\frac{p(\len{w})}{\len{w}} \geq 1 + \frac{\len{w} - \len{s}}{\len{w}}.
\]
\end{lemma}
\begin{proof}
For all $\len{s} < q \leq \len{w}$, the suffixes of $w$ and $y$ of length $q$ are distinct and are both right-special so $|\{w \in RS_{q}(X)\}| \geq 2$.  By Lemma \ref{RSlem}, then $p(\len{w}) \geq p(\len{s}) + \sum_{q=\len{s}+1}^{\len{w}} 2 = p(\len{s}) + 2(\len{w}-\len{s})$.
\end{proof}

\begin{lemma}[\cite{creutzpavlov} Lemma 2.8]\label{6}
If $p(q + 1) - p(q) = 1$ then there exists a bi-special word which has length in $[q, q+p(q)]$, has exactly two successors, and is the unique right-special word of its length and also the unique left-special word of its length.  %In particular, this %forces a two-letter alphabet (before the application of the final substitution $\pi$).
\end{lemma}

In fact, inductive usage of Lemma~\ref{6} is the key to obtaining an $S$-adic description for $X$ on a two-letter alphabet 
(before the application of the final substitution $\pi$), since it implies that the associated Rauzy graph has two loops.

The starting point for our construction of the words $v_{k}$ and $u_{k}$ is the following lemma.  %We note that the main idea in the following is the behavior of bispecial words in the Rauzy graphs.

\begin{lemma}\label{u0v0}
Let $X$ be an infinite minimal subshift satisfying $\limsup \frac{p(q)}{q} < 2$.  For any $Q > 0$, there exist words $u$ and $v$ with $\len{v} \geq Q$ such that, writing $s$ for the maximal common suffix of $v^{\infty}$ and $v^{\infty}u$,
\begin{itemize}
 \item $u$ and $v$ begin with different letters;
 \item $v$ is a proper suffix of $u$;
 \item $s$ is the unique left-special and unique right-special word of its length;
 \item every word in $\mathcal{L}(X)$ which has $s$ as a suffix is a suffix of a concatenation of $u$ and $v$; and
\item  every $x \in X$ can be written in exactly one way as a concatenation of $u$ and $v$.
  \end{itemize}
\end{lemma}
\begin{proof}
Since $\limsup \frac{p(q)}{q} < 2$, there exist infinitely many $q$ such that $p(q + 1) - p(q) = 1$ and eventually $p(q) < 2q$.  By Lemma \ref{6}, there are then infinitely many $q$ such that there exists a word $w_{q}$ which is the unique left-special and unique right-special word of length $q$ (which will have exactly two successors).

Since $w_{q}$ is the only vertex with more than one outgoing or incoming edge, each of the two edges leaving $w_{q}$ extends uniquely into a cycle based at $w_{q}$.  Let $v_{q}$ and $u_{q}$ with $\len{v_{q}} \leq \len{u_{q}}$ be the words read along these cycles, so that every $x \in X$ can be written uniquely as a bi-infinite concatenation of $v_{q}$ and $u_{q}$.

Then $v_{q}$ and $u_{q}$ begin with different letters and every $x \in X$ can be written in exactly one way as a concatenation of $v_{q}$ and $u_{q}$ since every $x \in X$ must label a path in the Rauzy graph.  Since $2\len{v_{q}} \leq \len{v_{q}} + \len{u_{q}} \leq p(\len{w_{q}}) + 1 < 2\len{w_{q}} + 1$, we have $\len{v_{q}} \leq \len{w_{q}}$ so by Lemma \ref{B}, $v_{q}$ is a root of $w_{q}$.   Since $w_{q}u_{q}$ has $w_{q}$ as a suffix, it has $v_{q}$ as a suffix and as $\len{v_{q}} \leq \len{u_{q}}$ then $v_{q}$ is a suffix of $u_{q}$.

If $\len{v_{q}} = \len{u_{q}}$ then $\len{u_{q}} \leq \len{w_{q}}$ so both $v_{q}$ and $u_{q}$ are suffixes of $w_{q}$ of the same length which would imply $v_{q} = u_{q}$ so $\len{v_{q}} < \len{u_{q}}$.
Since $v_{q}$ is a suffix of $w_{q}$, it is also a suffix of $w_{q^{\prime}}$ for $q^{\prime} > q$.  Then $v_{q}$ must be a suffix of $v_{q^{\prime}}$ since $v_{q^{\prime}}$ is a return word for $w_{q^{\prime}}$ hence for $w_{q}$ so $\len{v_{q}} \leq \len{v_{q^{\prime}}}$.  Suppose $\len{v_{q}}$ is bounded.  Then $v_{q} = v_{Q}$ for some fixed $Q$ eventually but that would make $v_{Q}$ a root of $w_{q}$ for arbitrarily long $w_{q}$ making the bi-infinite word $v_{Q}^{\pm \infty} \in X$ which contradicts that $X$ is infinite and minimal.

Let $s_{q}$ be the maximal common suffix of $v_{q}^{\infty}$ and $v_{q}^{\infty}u_{q}$.  Since $v_{q}$ is a path from $w_{q}$ to itself in the Rauzy graph, it follows that $w_{q}v_{q}$ has $w_{q}$ as a suffix.  Then $w_{q}v_{q}^{t}$ has $w_{q}$ as a suffix for every $t \geq 1$.  Likewise, since $u_{q}$ is a return word for $w_{q}$, it follows that $w_{q}v_{q}^{t}u_{q}$ has $w_{q}$ as a suffix.  Therefore both $v_{q}^{\infty}$ and $v_{q}^{\infty}u_{q}$ have $w_{q}$ as a suffix.

%Since $v_{q}$ is a return word to $w_{q}$, then $w_{q}$ is a suffix of $w_{q}v_{q}$ hence of $v_{q}^{\infty}$.  Since $u_{q}$ is a return word as well, then $w_{q}u_{q}$ and hence $v_{q}^{\infty}u_{q}$ has $w_{q}$ as a suffix.

%Since $v_{q}$ and $u_{q}$ are return words for $w_{q}$, then $w_{q}v_{q}^{t}u_{q}$ has $w_{q}$ as a suffix for all $t \geq 0$ so $w_{q}$ is a suffix of $v_{q}^{\infty}u_{q}$.  Since $v_{q}$ is a root of $w_{q}$, then $w_{q}$ is a suffix of $s_{q}$.  

Since $w_{q}$ is left-special, and since the two paths from $w_{q}$ to itself are labeled $v_{q}$ and $u_{q}$, it follows that 
$w_q v_{q}$ and $w_q u_{q}$ each have suffix $w_q$ and must differ on the letter prior to $w_{q}$ as otherwise both paths would pass through the same vertex prior to reaching $w_{q}$.  Therefore $s_{q} = w_{q}$.  Then any word which has $s_{q}$ as a suffix must be a suffix of a concatenation of $v_{q}$ and $u_{q}$ as those are the labels of the two return paths.
\end{proof}

The inductive step in the construction of $v_{k}$ and $u_{k}$ comes from our next lemma.

\begin{lemma}\label{nextstep}
Let $X$ be an infinite minimal subshift, and let $u$ and $v$ be words where all $x \in X$ can be written as a concatenation of $u$ and $v$, $\len{v} < \len{u}$, $v$ is a suffix of $u$ but not a prefix of $u$, and $p(q)/q < 1.5$ for all $q \geq \len{v}$. Let $p$ be the maximal common prefix of $u$ and $v$ and $s$ be the maximal common suffix of $v^{\infty}$ and $v^{\infty}u$.
Provided that $\len{p} + \len{s} < \len{u} + \len{v}$, exactly one of the following holds:
\begin{itemize}
\item there exist positive integers $m < n$ such that every $x \in X$ can be written as a concatenation of $v^{m-1}u$ and $v^{n-1}u$; or
\item there exist positive integers $r < m < n$ such that every $x \in X$ can be written as a concatenation of $v^{m-1}uv^{r-1}u$ and $v^{n-1}uv^{r-1}u$.
\end{itemize}
\end{lemma}
\begin{proof}
For brevity, we will use `concatenation' to refer to a concatenation of $u$ and $v$ corresponding to some $x \in X$.
Consider the set $S = \{ t \geq 0 : uv^{t}u \text{ appears in a concatenation} \}$.  If $|S| = 1$ then, by minimality, $X$ would consist solely of words which are concatenations of $uv^{t}$ for some fixed $t$ forcing $X$ to be a finite set of periodic words, contradicting that $X$ is infinite.  If $|S| = 2$ then the first of the two possible conclusions hold by setting $m = \min(S) + 1$ and $n = \max(S) + 1$ since every concatenation is of the form $\ldots u v^{i_{1}} u v^{i_{2}} u v^{i_{3}} u \ldots$ where all $i_{j}$ are either $m-1$ or $n-1$.

So we may assume that $uv^{x}u, uv^{y}u, uv^{z}u$ for $x < y < z$ all appear in the concatenations and take $x$ to be the minimal such value and $y$ to be the next smallest value.

Suppose that $uv^{x}uv^{x}u$ appears in a concatenation.  Then, using Lemma \ref{s}, $sv^{x}uv^{x}uv^{x}p$ is right-special as $uv^{x}uv^{x}u$ must be preceded by $v^{x}$ due to the minimality of $x$ and the $v^{x}uv^{x}u$ pattern cannot continue forever (by minimality) and when it is broken we see $v^{x}uv^{y}$.  Also $sv^{x}uv^{y}p$ is right-special due to $x$ being minimal and $z > y$.  By Lemma \ref{s3}, the maximal common suffix of $sv^{x}uv^{x}uv^{x}p$ and $sv^{x}uv^{y}p$ is $sv^{x}p$.

In the case when $\len{sv^{x}uv^{y}p} \leq \len{sv^{x}uv^{x}uv^{x}p}$, since $x \leq y-1$ and $\len{p} + \len{s} < \len{v} + \len{u}$, by Lemma \ref{2rs},
\begin{align*}
\frac{p(\len{sv^{x}uv^{y}p})}{\len{sv^{x}uv^{y}p}} &\geq 1 + \frac{\len{u} + y\len{v}}{\len{p} + \len{s} + \len{u} + (x+y)\len{v}} 
> 1 + \frac{\len{u} + y\len{v}}{2\len{u} + (x+y+1)\len{v}} \\
&\geq 1 + \frac{\len{u} + y\len{v}}{2\len{u} + (y-1+y+1)\len{v}} = \frac{3}{2}.
\end{align*}

In the case when $\len{sv^{x}uv^{x}uv^{x}p} < \len{sv^{x}uv^{y}p}$, since $\len{u} > \len{v}$, by Lemma \ref{2rs},
\begin{align*}
\frac{p(\len{sv^{x}uv^{x}uv^{x}p})}{\len{sv^{x}uv^{x}uv^{x}p}} &\geq 1 + \frac{2\len{u} + 2x\len{v}}{\len{p} + \len{s} + 2\len{u} + 3x\len{v}} \\
&> 1 + \frac{2\len{u} + 2x\len{v}}{3\len{u} + (3x+1)\len{v}}
> 1 + \frac{2\len{u} + 2x\len{v}}{4\len{u} + 3x\len{v}} \geq \frac{3}{2}.
\end{align*}

Since $p(q) < 1.5q$ for $q \geq \len{v}$, this is a contradiction and therefore $uv^{x}uv^{x}u$ never appears in a concatenation.

Now suppose that $v^{y}uv^{y}$ appears in a concatenation.  Then $sv^{y}uv^{x}p$ is right-special as $uv^{x}u$ must be preceded by $v^{y}$ as $y$ is the next smallest value (and $uv^{x}uv^{x}u$ does not appear).  Also $sv^{x}uv^{y}p$ is right-special as $uv^{z}u$ must be preceded by $v^{x}$ by minimality of $x$.  By Lemma \ref{s3}, the maximal common suffix of $sv^{y}uv^{x}p$ and $sv^{x}uv^{y}p$ is $sv^{x}p$.  Therefore, as $x \leq y-1$, by Lemma \ref{2rs},
\begin{align*}
\frac{p(\len{sv^{y}uv^{x}p})}{\len{sv^{y}uv^{x}p}} &\geq 1 + \frac{\len{u} + y\len{v}}{\len{p} + \len{s} + \len{u} + (x+y)\len{v}} \\
&> 1 + \frac{\len{u} + y\len{v}}{2\len{u} + (x+y+1)\len{v}}
\geq 1 + \frac{\len{u} + y\len{v}}{2\len{u} + (y-1+y+1)\len{v}} = \frac{3}{2}
\end{align*}
which again contradicts that $p(q) < 1.5q$ for $q \geq \len{v}$; therefore $v^{y}uv^{y}$ never appears.  Then every appearance of $uv^{w}u$ for $w > x$ appears as part of $uv^{x}uv^{w}uv^{x}u$.  As $uv^{x}uv^{x}u$ never appears, then every occurrence of $uv^{w}u$ appears as part of $v^{y}uv^{x}uv^{w}uv^{x}uv^{y}$ by the minimality of $y$ as the second smallest possible value.  By Lemma \ref{s2} then $uv^{w}u$ for $w > x$ always appears as part of $sv^{y}uv^{x}uv^{w}uv^{x}uv^{y}p$.

Since $z > x$, then $sv^{y}uv^{x}uv^{z}uv^{x}uv^{y}p$ appears in a concatenation.  That word has $sv^{y}uv^{x}uv^{y}v$ as a prefix (as $z > y$) and $sv^{y}uv^{x}uv^{y}uv^{x}uv^{y}p$, which also appears, has $sv^{y}uv^{x}uv^{y}u$ as a prefix so $sv^{y}uv^{x}uv^{y}p$ is right-special.  If $uv^{w}u$ for $w > z$ also appears then by the same reasoning, $sv^{y}uv^{x}uv^{z}p$ is right-special.  By Lemma \ref{s3}, the maximal common suffix of $sv^{y}uv^{x}uv^{y}p$ and $sv^{y}uv^{x}uv^{z}p$ is $sv^{y}p$ so we would have, by Lemma \ref{2rs},
\begin{align*}
&\frac{p(\len{sv^{y}uv^{x}uv^{y}p})}{\len{sv^{y}uv^{x}uv^{y}p}}
\geq 1 + \frac{(x + y)\len{v} + 2\len{u}}{(x + 2y)\len{v} + 2\len{u} + \len{p} + \len{s}} \\
&\quad\quad > 1 + \frac{(x + y)\len{v} + 2\len{u}}{(x + 2y + 1)\len{v} + 3\len{u}} 
= \frac{3}{2} + \frac{(\frac{1}{2}x-\frac{1}{2})\len{v} + \frac{1}{2}\len{u}}{(x + 2y + 1)\len{v} + 3\len{u}} 
\geq \frac{3}{2} + \frac{\frac{1}{2}\len{u} - \frac{1}{2}\len{v}}{(x + 2y + 1)\len{v} + 3\len{u}} > \frac{3}{2}
\end{align*}
contradicting that $p(q) < 1.5q$ for $q \geq \len{v}$; therefore $|S| = 3$.  Therefore every concatenation is of the form
\[
\ldots uv^{x}uv^{i_{0}}uv^{x}uv^{i_{1}}uv^{x}ux^{i_{2}}uv^{x}uv^{i_{3}}uv^{x}u \ldots
\]
where $i_{j}$ are all either $y$ or $z$.  Setting $r = x+1$ and $m = y+1$ and $n = z+1$ then proves the claim.
\end{proof}

We are now prepared to prove Proposition~\ref{words}.

\begin{proof}[Proof of Proposition \ref{words}]
Since $\limsup \frac{p(q)}{q} < 1.5$, there exists $\nu > 0$ and $q_{0}$ so that $p(q) < (1.5-\nu)q$ for all $q \geq q_{0}$. 
Let $u_{0}$ and $v_{0}$ be the words guaranteed by Lemma \ref{u0v0} such that $\len{v_{0}} \geq q_{0}$.  Then $p_{0}$ is empty as $u_{0}$ and $v_{0}$ begin with different letters.  Lemma \ref{s} implies $\len{s_{0}} < \len{u_{0}v_{0}}$ so $\len{p_{0}} + \len{s_{0}} = \len{s_{0}} < \len{u_{0}} + \len{v_{0}}$.

Proceed by induction assuming we have constructed the words $u_{k}$ and $v_{k}$.  By Lemma \ref{nextstep}, there either exist positive integers $m_{k} < n_{k}$ such that every $x \in X$ can be written as a concatenation of $v_{k+1} := v_{k}^{m_{k}-1}u_{k}$ and $u_{k+1} := v_{k}^{n_{k}-1}u_{k}$ or there exist positive integers $r_{k} < m_{k} < n_{k}$ such that every $x \in X$ can be written as a concatenation of $v_{k+1} := v_{k}^{m_{k}-1}u_{k}v_{k}^{r_{k}-1}u_{k}$ and $u_{k+1} := v_{k}^{n_{k}-1}u_{k}v_{k}^{r_{k}-1}u_{k}$.

Since $v_{k+1}$ has $v_{k}^{m_{k}-1}u_{k}$ as a prefix and $u_{k+1}$ has $v_{k}^{m_{k}-1}v_{k}$ as a prefix (as $m_{k} < n_{k}$), $p_{k+1} = v_{k}^{m_{k}-1}p_{k}$.  Since $v_{k+1}^{\infty}$ has $u_{k}v_{k+1}$ as a suffix and $u_{k+1} = v_{k}^{n_{k}-m_{k}}v_{k+1}$ has $v_{k}v_{k+1}$ as a suffix, by Lemmas \ref{s2} and \ref{s3}, we have $s_{k+1} = s_{k}v_{k+1}$.  Therefore
\[
\len{p_{k+1}} + \len{s_{k+1}} = (m_{k}-1)\len{v_{k}} + \len{v_{k+1}} + \len{p_{k}} + \len{s_{k}}
< (m_{k}-1)\len{v_{k}} + \len{v_{k+1}} + \len{u_{k}} + \len{v_{k}}
= 2\len{v_{k+1}} + \len{v_{k}}
\]
and
as $\len{u_{k+1}} \geq \len{v_{k+1}} + \len{v_{k}}$, then $\len{p_{k+1}} + \len{s_{k+1}} < \len{u_{k+1}} + \len{v_{k+1}}$.

By induction on $k$, each $x \in X$ can be decomposed uniquely into words $v_{k}$ and $u_{k}$. For $k = 0$, this follows from Lemma \ref{u0v0} since
$v_0$ and $u_0$ were constructed using that lemma.  If $x$ can be uniquely represented as a concatenation of $v_{k}$ and $u_{k}$ then the same must be true of $v_{k}^{m_{k-1}}u_{k}$ and $v_{k}^{n_{k}-1}u_{k}$, or of both followed immediately by $v_{k}^{r_{k}-1}u_{k}$ for $k$ for which $r_{k}$ exists.
\end{proof}

\begin{remark}\label{pands}
For all $k$, $p_{k+1} = v_{k}^{m_{k}-1}p_{k}$ and $s_{k+1} = s_{k}v_{k+1}$ as shown in the proof of Proposition \ref{words}.
\end{remark}

\subsection{Complexity estimates}

Having established the substitutive structure of low complexity minimal subshifts, we can now determine what their right-special words are, which allows us to estimate word complexity using Lemma~\ref{RSlem}.

\begin{proposition}\label{rs}
Let $X$ be an infinite minimal subshift satisfying the conclusions of Proposition \ref{words}.  For the words $\{ u_{k} \}$ and $\{ v_{k} \}$, the following hold:
\begin{itemize}
\item the left-infinite word $p_{\infty} = \lim s_{k}p_{k} = \lim s_{0}v_{1}\cdots v_{k+1}v_{k}^{m_{k}-1}v_{k-1}^{m_{k-1}-1}\cdots v_{0}^{m_{0}-1}$ is right-special;
\item for each $k$, the word $s_{k}v_{k}^{n_{k}-2}p_{k}$ is right-special and the maximal common suffix of it and $p_{\infty}$ is $s_{k}v_{k}^{m_{k}-1}p_{k}$; and
\item for $k$ such that $r_{k} > 0$, the word $s_{k}v_{k}^{r_{k}-1}u_{k}v_{k}^{r_{k}-1}p_{k}$ is right-special and the maximal common suffix of it and $p_{\infty}$ and of it and $s_{k}v_{k}^{n_{k}-2}p_{k}$ is $s_{k}v_{k}^{r_{k}-1}p_{k}$.
\end{itemize}
\end{proposition}
\begin{proof}
Clearly $p_{k}$ is right-special as $u_{k} \ne v_{k}$ and $p_{k}$ must be followed by different letters in each due to maximality so Lemma \ref{s2} implies $s_{k}p_{k}$ is right-special.  By Remark \ref{pands}, $s_{k+1}p_{k+1} = s_{k}v_{k+1}v_{k}^{m_{k}-1}p_{k}$, and as this, by Lemma \ref{s2}, has $s_{k}p_{k}$ as a suffix, $p_{\infty}$ exists and is right-special.

Since $v_{k}^{n_{k}-1}u_{k}$ appears in a concatenation (if not then $u_{k+1}$ never appears so the subshift would be periodic) and is preceded by a concatenation of $u_{k}$ and $v_{k}$, by Lemma \ref{s2}, $s_{k}v_{k}^{n_{k}-1}u_{k}$ appears.  Since $s_{k}v_{k}^{n_{k}-2}v_{k}$ is a prefix of that word and $s_{k}v_{k}^{n_{k}-2}u_{k}$ is a suffix of it, $s_{k}v_{k}^{n_{k}-2}p_{k}$ is right-special. 

Since $s_{k+1}p_{k+1} = s_{k}v_{k+1}v_{k}^{m_{k}-1}p_{k}$ has $s_{k}u_{k}v_{k}^{m_{k}-1}p_{k}$ as a suffix, by Lemma \ref{s3}, the maximal common suffix of $p_{\infty}$ and $s_{k}v_{k}^{n_{k}-2}p_{k}$ is then $s_{k}v_{k}^{m_{k}-1}p_{k}$.

For $k$ such that $r_{k} > 0$, the word $v_{k+1} = v_{k}^{m_{k}-1}u_{k}v_{k}^{r_{k-1}}u_{k}$ appears in a concatenation preceded by a concatenation of $u_{k}$ and $v_{k}$ showing that $s_{k}v_{k}^{r_{k}-1}u_{k}v_{k}^{r_{k}-1}u_{k}$ appears (as $r_{k} < m_{k}$).  The word $u_{k}v_{k}^{r_{k-1}}u_{k}v_{k}^{m_{k}-1}$ appears in a concatenation (in fact with the second $u_{k}$ being a suffix of any $v_{k+1}$ or $u_{k+1}$ that appears in a $(k+1)$-concatenation) showing that $s_{k}v_{k}^{r_{k}-1}u_{k}v_{k}^{r_{k}-1}v_{k}$ appears (as $r_{k} < m_{k}$).  Therefore $s_{k}v_{k}^{r_{k}-1}u_{k}v_{k}^{r_{k}-1}p_{k}$ is right-special.

As both $p_{\infty}$ and $s_{k}v_{k}^{n_{k}-2}p_{k}$ have $v_{k}^{m_{k}-1}p_{k}$ as a suffix and $r_{k} < m_{k}$, both have $v_{k}v_{k}^{r_{k}-1}p_{k}$ as a suffix.  By Lemma \ref{s3}, the maximal common suffix of $s_{k}v_{k}^{r_{k}-1}u_{k}v_{k}^{r_{k}-1}p_{k}$ and either of them is then $s_{k}v_{k}^{r_{k}-1}p_{k}$.
\end{proof}

\begin{lemma}\label{only}
Under the hypotheses of Proposition~\ref{rs}, every right-special word of length at least $\len{s_{0}}$ is a suffix of one of those from that proposition, and so for $q > \len{s_0}$, 
\[
p(q+1) - p(q) = 1 + \sum_{k = 0}^{\infty} \bbone_{(\len{s_k v_k^{m_k-1} p_k}, \len{s_k v_k^{n_k-2} p_k}]}(q) + 
\sum_{k = 0}^{\infty} \bbone_{r_{k}} \bbone_{(\len{s_k v_k^{r_k-1} p_k}, \len{s_k v_k^{r_k-1} u_k v_k^{r_k - 1} p_k}]}(q).
\]
\end{lemma}
\begin{proof}
By Lemma \ref{u0v0}, $s_{0}$ is the unique right-special word of its length.  Therefore every right-special word of length at least $\len{s_{0}}$ has $s_{0}$ as a suffix.  Lemma \ref{u0v0} also implies every right-special word of length at least $\len{s_{0}}$ is a suffix of a concatenation of $u_{0}$ and $v_{0}$.  Assume that every right-special word of length at least $s_{k}p_{k}$ is a suffix of a concatenation of $u_{k}$ and $v_{k}$ followed by $p_{k}$.  Let $w$ be a right-special word of length at least $\len{s_{k+1}p_{k+1}}$.  Then $w$ is a suffix of a concatenation of $u_{k}$ and $v_{k}$ followed by $p_{k}$ so has $s_{k}p_{k}$ as  a suffix.

By Remark \ref{pands}, $s_{k}p_{k} = s_{k-1}v_{k}p_{k}$ so $w$ has $v_{k}p_{k}$ as a suffix and also $\len{w} \geq \len{s_{k+1}p_{k+1}} = \len{s_{k}v_{k+1}v_{k}^{m_{k}-1}p_{k}}$.  Since $v_{k}u_{k}$ only appears in a concatenation of $v_{k}$ and $u_{k}$ as a suffix either of $v_{k+1}$ or of $v_{k+1}v_{k}^{m_{k-1}}u_{k}$, then $w$ either has $v_{k+1}v_{k}^{m_{k-1}}p_{k}$ as a suffix or has $v_{k+1}p_{k}$ as a suffix.  Since $v_{k+1}u_{k}$ only appears when $m_{k} = 1$ and since $w$ is right-special, in both cases $w$ has $v_{k+1}v_{k}^{m_{k}-1}p_{k}$ as a suffix.  As the $v_{k+1}$ is preceded by a concatenation of $u_{k+1}$ and $v_{k+1}$ of length at least $\len{s_{k}}$, by Remark \ref{pands}, then $w$ has $s_{k}v_{k+1}v_{k}^{m_{k}-1}p_{k} = s_{k+1}p_{k+1}$ as a suffix.  By induction, every right-special word with length at least $\len{s_{k}p_{k}}$ has $s_{k}p_{k}$ as a suffix and is a suffix of a concatenation of $u_{k}$ and $v_{k}$ followed by $p_{k}$.

Let $w$ be any right-special word with $\len{w} \geq \len{s_{0}} = \len{s_{0}p_{0}}$ which is not a suffix of $p_{\infty}$.  Let $k$ maximal such that $\len{w} \geq \len{s_{k}p_{k}}$.  Then $w = yp_{k}$ where $y$ is a suffix of a concatenation of $u_{k}$ and $v_{k}$ of length at least $s_{k}$.  Since $w$ is right-special, $yu_{k}$ and $yv_{k}$ must both appear in a concatenation.  So $y$ must share a suffix either with $v_{k}^{m_{k}-1}$ or with $v_{k}^{r_{k}-1}u_{k}v_{k}^{r_{k}-1}$ (in which case $r_{k} > 0$).

When $y$ shares a suffix with $v_{k}^{m_{k}-1}$, as $w$ is not a suffix of $p_{\infty}$, then $y$ has $s_{k}v_{k}^{m_{k}-1}$ as a proper suffix and shares a suffix with $s_{k}v_{k}^{n_{k}-1}$.  If $\len{y} > \len{s_{k}v_{k}^{n_{k}-2}}$ then $w$ being right-special would force $s_{k}v_{k}^{n_{k-1}}p_{k}$ or $s_{k}u_{k}v_{k}^{n_{k}-1}p_{k}$ to be right-special but $s_{k}v_{k}^{n_{k}}$  never appears in a concatenation.  So when $y$ shares a suffix with $v_{k}^{m_{k}-1}$, $w$ is a suffix of $s_{k}v_{k}^{n_{k}-2}p_{k}$.

When $y$ shares a suffix with $v_{k}^{r_{k-1}}u_{k}v_{k}^{r_{k}-1}$, since $w$ is not a suffix of $p_{\infty}$ it must be that $y$ has $s_{k}v_{k}^{r_{k-1}}$ as a proper suffix.  If $w$ is not a suffix of $s_{k}v_{k}^{r_{k}-1}u_{k}v_{k}^{r_{k-1}}p_{k}$ then $y$ must have either $u_{k}v_{k}^{r_{k}-1}u_{k}v_{k}^{r_{k}-1}$ or $v_{k}^{r_{k}}u_{k}v_{k}^{r_{k-1}}$ as a suffix.  In the first case $w$ being right-special would force $u_{k}v_{k}^{r_{k}-1}u_{k}v_{k}^{r_{k}-1}u_{k}$ to appear, which is impossible, and in the second case it would force $v_{k}^{r_{k}}u_{k}v_{k}^{r_{k}}$ to appear in a concatenation, also impossible.  So in the case $y$ shares a suffix with $v_{k}^{r_{k}-1}u_{k}v_{k}^{r_{k}-1}$, $w$ is a suffix of $s_{k}v_{k}^{r_{k}-1}u_{k}v_{k}^{r_{k}-1}p_{k}$.

The right-special words of any length $n > \len{s_0}$ are then: the suffix of length $n$ of $p_{\infty}$, 
the suffix of length $n$ of some $s_k v_k^{n_k-2} p_k$ (which exists and is different from the first word
iff $n \in (\len{s_k v_k^{m_k-1} p_k}, \len{s_k v_k^{n_k-2} p_k}]$), and 
the suffix of length $n$ of some $(\len{s_k v_k^{r_k-1} p_k}, \len{s_k v_k^{r_k-1} u_k v_k^{r_k - 1} p_k}]$ for which $r_k > 0$
(which exists and is different from the first and second words iff $n \in (\len{s_k v_k^{r_k-1} p_k}, \len{s_k v_k^{r_k-1} u_k v_k^{r_k - 1} p_k}]$). 
The complexity difference formula is now an immediate consequence of Lemma~\ref{RSlem}, along with the observation that
there is no overlap between the right-special words for distinct $k$, since by Lemma \ref{s3}, the maximal common suffix of $s_{k}v_{k}p_{k}$ and $s_{k+1}v_{k+1}p_{k+1}$ is $s_{k}p_{k}$ (recall that $v_{k+1}$ has $u_{k}$ as a suffix).
\end{proof}

By using this set of right-special words, we can write an explicit formula for the complexity function at some specific lengths which allow for the computation of $\limsup p(q)/q$. 

\begin{corollary}\label{sum}
Let $X$ be an infinite minimal subshift satisfying the conclusions of Proposition \ref{words}.  Set $C = p(\len{s_{0}}) - \len{s_{0}}$.  Then for every $k$, writing $\ell_{k} = \left\{ \begin{array}{ll} \min(\len{v_{k}^{n_{k}-r_{k}-1}},\len{v_{k}^{r_{k}-1}u_{k}}) &\text{if $r_{k} > 0$} \\ 0 &\text{otherwise} \end{array}\right.$,
\[
p(\len{s_{k}v_{k}^{n_{k}-2}p_{k}})
= \len{s_{k}v_{k}^{n_{k}-2}p_{k}} + \bbone_{r_{k}}\ell_{k} + \sum_{j=0}^{k} (n_{j} - m_{j} - 1)\len{v_{j}} + \sum_{j=0}^{k-1} ((r_{j}-1)\len{v_{j}} + \len{u_{j}})\bbone_{r_{j}} + C
\]
and for $k$ such that $r_{k} > 0$,
\begin{multline*}
p(\len{s_{k}v_{k}^{r_{k}-1}u_{k}v_{k}^{r_{k}-1}p_{k}})
= \len{s_{k}v_{k}^{r_{k}-1}u_{k}v_{k}^{r_{k}-1}p_{k}} + \sum_{j=0}^{k-1} (n_{j} - m_{j} - 1)\len{v_{j}} + \sum_{j=0}^{k} ((r_{j}-1)\len{v_{j}} + \len{u_{j}})\bbone_{r_{j}} \\
+ \left|(\len{s_k v_k^{m_k-1} p_k}, \len{s_k v_k^{n_k-2} p_k}] \cap (1, \len{s_k v_k^{r_k-1} u_k v_k^{r_k - 1} p_k}] \right| + C
\end{multline*}
and $\limsup \frac{p(q)}{q}$ is attained along some subsequence of these values. 
\end{corollary}
\begin{proof}
This is a fairly immediate corollary of Lemmas \ref{RSlem} and \ref{only}; 
we note only that $\ell_k$ is, when $r_k > 0$, the number of elements of 
$(\len{s_k v_k^{r_k-1} p_k}, \len{s_k v_k^{r_k-1} u_k v_k^{r_k - 1} p_k}]$ which are less than $\len{s_k v_k^{n_k-2} p_k}$.

The limsup must be attained along a subsequence of the indicated sequences since 
they are the right endpoints of the intervals in the characteristic functions from Lemma~\ref{only}.
\end{proof}

\begin{remark}\label{infsum}
We could make a similar formulation of $\liminf p(q)/q$ using the left endpoints of the intervals from Lemma~\ref{only}, but since we do not have need of that in this work, we do not do so here.
\end{remark}

The following corollary is easily checked; we do not at all claim that it is a complete description, but yields at least a simple sufficient condition for $\limsup p(q)/q < 1.5$.

\begin{corollary}\label{sum2}
If $X$ is generated by the substitutions $\tau_{m_k, n_k, 0}$ and $\sup n_k/m_k < 2$, then $\limsup p(q)/q < 1.5$.
\end{corollary}

\subsection{Restrictions on the substitutions}\label{ineqs}

By Corollary~\ref{taus}, the complexity hypothesis $\limsup \frac{p(q)}{q} < 1.5$ ensures that $X$ is defined
by substitutions $\tau_{m_{k}, n_{k}, r_{k}}$. In this section, we give some restrictions on how these integers are related.

However, the restrictions we need are not expressible on the level of individual substitutions, and instead require
analysis simultaneously across multiple ``levels" of the substitutions (making use of Corollary \ref{sum}) to establish lower bounds on $p(q)/q$ which contradict the assumption that $\limsup p(q)/q < 3/2$. 
We mention that this multi-level analysis is crucial to our approach, and is a technique that does not often appear in the literature, highlighting the subtlety inherent in this class of subshfits.  

\begin{lemma}\label{pandsbounds}
Let $X$ be an infinite minimal subshift with $\limsup \frac{p(q)}{q} < 1.5$ and $\nu > 0$ and the sequences of words $\{ u_{k} \}, \{ v_{k} \}, \{ p_{k} \}, \{ s_{k} \}$ and integers $\{ m_{k} \}, \{ n_{k} \}, \{ r_{k} \}$ be those from Proposition \ref{words}.

For all $k$,
\begin{align*}
\len{s_{k}} + \len{p_{k}} &< (m_{k-3}+2)\len{v_{k-3}} + m_{k-2}\len{v_{k-2}} + m_{k-1}\len{v_{k-1}} + \len{v_{k}}; \\
\len{s_{k}} + \len{p_{k}} &< (m_{k-2}+2)\len{v_{k-2}} + m_{k-1}\len{v_{k-1}} + \len{v_{k}};~\text{and} \\
\len{s_{k}} + \len{p_{k}} &< (m_{k-1}+2)\len{v_{k-1}} + \len{v_{k}}.
\end{align*}
\end{lemma}
\begin{proof}
By Remark \ref{pands} applied three times and that $\len{s_{k-3}} + \len{p_{k-3}} < 3\len{v_{k-3}}$,
\begin{align*}
\len{s_{k}} + \len{p_{k}} &=  \len{s_{k-1}} + \len{v_{k}} + \len{p_{k-1}} + (m_{k-1}-1)\len{v_{k-1}} \\
&= \len{s_{k-2}} + \len{v_{k-1}} + \len{v_{k}} + \len{p_{k-2}} + (m_{k-2}-1)\len{v_{k-2}} + (m_{k-1}-1)\len{v_{k-1}}  \\
&= \len{s_{k-3}} + (m_{k-3}-1)\len{v_{k-3}} + (m_{k-2}-1)\len{v_{k-2}} + m_{k-1}\len{v_{k-1}} + \len{v_{k}} + \len{p_{k-3}} + \len{v_{k-2}} \\
&< (m_{k-3}+2)\len{v_{k-3}} + m_{k-2}\len{v_{k-2}} + m_{k-1}\len{v_{k-1}} + \len{v_{k}}.
\end{align*}
Since $\len{s_{k-1}} + \len{p_{k-1}} < 3\len{v_{k-1}}$, $\len{s_{k-1}} + (m_{k-1}-1)\len{v_{k-1}} + \len{p_{k-1}} + \len{v_{k}} < (m_{k-1}+2)\len{v_{k-1}} + \len{v_{k}}$ and since $\len{s_{k-2}} + \len{p_{k-2}} < 3\len{v_{k-2}}$, $\len{s_{k-2}} + (m_{k-2}-1)\len{v_{k-2}} + m_{k-1}\len{v_{k-1}} + \len{p_{k-2}} + \len{v_{k}} < (m_{k-2}+2)\len{v_{k-2}} + m_{k-1}\len{v_{k-1}} + \len{v_{k}}$.
\end{proof}

\begin{proposition}\label{svp}
Let $X$ be an infinite minimal subshift with $\limsup \frac{p(q)}{q} < 1.5$ and $\nu > 0$ and the sequences of words $\{ u_{k} \}, \{ v_{k} \}, \{ p_{k} \}, \{ s_{k} \}$ and integers $\{ m_{k} \}, \{ n_{k} \}, \{ r_{k} \}$ be those from Proposition \ref{words}.

For $k \geq 2$ such that $n_{k} > 2m_{k}$, exactly one of the following holds:
\begin{enumerate}[(i)\hspace{5pt}]
\item\label{svp+2} $n_{k} = 2m_{k}+2$, $n_{k-1} = m_{k-1} + 1$, $n_{k-2} \leq \frac{4}{3}m_{k-2}+1$ and $r_{k} = 0$ and $r_{k-1} = 0$;
\item\label{svp+1a} $n_{k} = 2m_{k} + 1$, $n_{k-1} \leq 2m_{k-1}$ and $r_{k} = 0$; or
\item\label{svp+1b} $n_{k} = 2m_{k} + 1$, $m_{k-1} = 1$, $n_{k-1} = 3$, $n_{k-2} = m_{k-2} + 1$ and $r_{k} = r_{k-1} = r_{k-2} = 0$.
\end{enumerate}
\end{proposition}
\begin{proof}
Let $k$ such that $n_{k} > 2m_{k}$.
Since $\len{p_{k}} + \len{s_{k}} < 3\len{v_{k}}$, Corollary \ref{sum} implies
\begin{align*}
\frac{p(\len{s_{k}v_{k}^{n_{k}-2}p_{k}})}{\len{s_{k}v_{k}^{n_{k}-2}p_{k}}}
&\geq 1 + \frac{(n_{k} - m_{k} - 1 + \bbone_{r_{k}})\len{v_{k}}}{(n_{k}-2)\len{v_{k}} + \len{p_{k}} + \len{s_{k}}} 
> 1 + \frac{n_{k} - m_{k} - 1 + \bbone_{r_{k}}}{n_{k}-2 + 3}
= \frac{3}{2} + \frac{\frac{1}{2}n_{k} - m_{k} - \frac{3}{2} + \bbone_{r_{k}}}{n_{k}+1}
\end{align*}
and therefore $n_{k} - 2m_{k} - 3 + 2\cdot\bbone_{r_{k}} < 0$.  So if $r_{k} > 0$ then $n_{k} < 2m_{k} + 1$, a contradiction, and if not then $n_{k} < 2m_{k} + 3$.

By Corollary \ref{sum} and Lemma \ref{pandsbounds},
\begin{align*}
\frac{p(\len{s_{k}v_{k}^{n_{k}-2}p_{k}})}{\len{s_{k}v_{k}^{n_{k}-2}p_{k}}}
&\geq 1 + \frac{(n_{k} - m_{k} - 1)\len{v_{k}} + (n_{k-1} - m_{k-1} - 1 + \bbone_{r_{k-1}})\len{v_{k-1}}}{(n_{k}-2)\len{v_{k}} + \len{p_{k}} + \len{s_{k}}} \\
&> 1 + \frac{(n_{k} - m_{k} - 1)\len{v_{k}} + (n_{k-1} - m_{k-1} - 1 + \bbone_{r_{k-1}})\len{v_{k-1}}}{(n_{k}-2)\len{v_{k}} + \len{v_{k}} + (m_{k-1}+2)\len{v_{k-1}}} \\
&= \frac{3}{2} + \frac{(\frac{1}{2}n_{k} - m_{k} - \frac{1}{2})\len{v_{k}} + (n_{k-1} - \frac{3}{2}m_{k-1} - 2 + \bbone_{r_{k-1}})\len{v_{k-1}}}{(n_{k}-1)\len{v_{k}} + (m_{k-1}+2)\len{v_{k-1}}}
\end{align*}
and therefore $(\frac{1}{2}n_{k} - m_{k} - \frac{1}{2})\len{v_{k}} + (n_{k-1} - \frac{3}{2}m_{k-1} - 2 + \bbone_{r_{k-1}})\len{v_{k-1}} < 0$.

If $n_{k} = 2m_{k} + 2$ then $\frac{1}{2}\len{v_{k}} + (n_{k-1} - \frac{3}{2}m_{k-1} - 2 + \bbone_{r_{k-1}})\len{v_{k-1}} < 0$ and since, $\len{v_{k}} = \len{v_{k-1}^{m_{k-1}-1}u_{k-1}} > m_{k-1}\len{v_{k-1}}$, then $n_{k-1} - m_{k-1} - 2 + \bbone_{r_{k-1}} < 0$ so $r_{k-1} = 0$ and $n_{k-1} = m_{k-1} + 1$.  By Corollary \ref{sum} and Lemma \ref{pandsbounds},
\begin{align*}
\frac{p(\len{s_{k}v_{k}^{n_{k}-2}p_{k}})}{\len{s_{k}v_{k}^{n_{k}-2}p_{k}}}
&\geq 1 + \frac{(m_{k}+1)\len{v_{k}} + (n_{k-2} - m_{k-2} - 1)\len{v_{k-2}}}{(n_{k}-2)\len{v_{k}} + \len{v_{k}} + \len{v_{k-1}} + (m_{k-2}+2)\len{v_{k-2}}} \\
&= \frac{3}{2} + \frac{\frac{1}{2}\len{v_{k}} - \frac{1}{2}\len{v_{k-1}} + (n_{k-2} - \frac{3}{2}m_{k-2} - 2)\len{v_{k-2}}}{(2m_{k}+1)\len{v_{k}} + \len{v_{k-1}} + (m_{k-2}+2)\len{v_{k-2}}}.
\end{align*}
Since $\len{v_{k}} - \len{v_{k-1}} \geq \len{u_{k-1}} - \len{v_{k-1}} = (n_{k-2} - m_{k-2})\len{v_{k-2}}$, this implies $\frac{3}{2}n_{k-2} - 2m_{k-2} - 2 < 0$ meaning that $n_{k-2} \leq \frac{4}{3}m_{k-2} + 1$, putting us in case (\ref{svp+2}).

So we may assume from here on that $n_{k} = 2m_{k} + 1$.  Since $r_{k} = 0$, if $n_{k-1} \leq 2m_{k-1}$ then we are in case (\ref{svp+1a}).  So we may assume from here on that $n_{k-1} = 2m_{k-1} + a + 1$ for some $a \geq 0$.
The above gives that $n_{k-1} - \frac{3}{2}m_{k-1} - 2 + \bbone_{r_{k-1}} < 0$ so $\frac{1}{2}m_{k-1} + a - 1 + \bbone_{r_{k-1}} < 0$.  Then $r_{k-1} = 0$ and $\frac{1}{2}m_{k-1} + a < 1$ meaning that $m_{k-1} = 1$ and $a = 0$ so $n_{k-1} = 3$.  
By Corollary \ref{sum} and Lemma \ref{pandsbounds},
\begin{align*}
\frac{p(\len{s_{k}v_{k}^{n_{k}-2}p_{k}})}{\len{s_{k}v_{k}^{n_{k}-2}p_{k}}}
&\geq 1 + \frac{m_{k}\len{v_{k}} + (n_{k-1}-m_{k-1}-1)\len{v_{k-1}} + (n_{k-2} - m_{k-2} - 1 + \bbone_{r_{k-2}})\len{v_{k-2}}}{(n_{k}-2)\len{v_{k}} + \len{p_{k}} + \len{s_{k}}} \\
&> 1 + \frac{m_{k}\len{v_{k}} + \len{v_{k-1}} + (n_{k-2} - m_{k-2} - 1 + \bbone_{r_{k-2}})\len {v_{k-2}}}{(2m_{k}-1)\len{v_{k}} + (m_{k-2} + 2)\len{v_{k-2}} + \len{v_{k-1}} + \len{v_{k}}} \\
&= \frac{3}{2} + \frac{\frac{1}{2}\len{v_{k-1}} + (n_{k-2} - \frac{3}{2}m_{k-2} - 2 + \bbone_{r_{k-2}})\len{v_{k-2}}}{2m_{k}\len{v_{k}} + \len{v_{k-1}} + (m_{k-2}+2)\len{v_{k-2}}}
\end{align*}
and, since $\len{v_{k-1}} > m_{k-2}\len{v_{k-2}}$, therefore $n_{k-2} - m_{k-2} - 2 + \bbone_{r_{k-2}} < 0$.  Then $r_{k-2} = 0$ and $n_{k-2} = m_{k-2} + 1$, putting us in case (\ref{svp+1b}).
\end{proof}

\begin{proposition}\label{rk}
Let $X$ be an infinite minimal subshift with $\limsup \frac{p(q)}{q} < 1.5$ and $\nu > 0$ and the sequences of words $\{ u_{k} \}, \{ v_{k} \}, \{ p_{k} \}, \{ s_{k} \}$ and integers $\{ m_{k} \}, \{ n_{k} \}, \{ r_{k} \}$ be those from Proposition \ref{words}.

For $k \geq 2$ such that $r_{k+1} > 0$, $n_{k} \leq \frac{4}{3}m_{k} + 1$ and exactly one of the following holds:
\begin{enumerate}[(i)\hspace{5pt}]
\item\label{rk3/2} $n_{k} \leq \frac{3}{2}m_{k}$;
\item\label{rk35} $m_{k}=3$, $n_{k}=5$, $n_{k-1} = m_{k-1} + 1$ and $r_{k} = 0$ and $r_{k-1} = 0$;
\item\label{rk12a} $m_{k}=1$, $n_{k}=2$, $n_{k-1} \leq 2m_{k-1}$ and $r_{k} = 0$; or
\item\label{rk12b} $m_{k}=1$, $n_{k}=2$, $m_{k-1}=1$, $n_{k-1}=3$, $n_{k-2} = m_{k-2} + 1$ and $r_{k} = 0$ and $r_{k-1} = 0$.
\end{enumerate}
\end{proposition}
\begin{proof}
Let $k \geq 2$ such that $r_{k+1} > 0$.  For brevity, we will write
\[
P_{k} := \frac{p(\len{s_{k+1}v_{k+1}^{r_{k+1}-1}u_{k+1}v_{k+1}^{r_{k+1}-1}p_{k+1}})}{\len{s_{k+1}v_{k+1}^{r_{k+1}-1}u_{k+1}v_{k+1}^{r_{k+1}-1}p_{k+1}}}.
\]
By Corollary \ref{sum} and Lemma \ref{pandsbounds},
\begin{align*}
P_{k} &> 1 + \frac{(r_{k+1}-1)\len{v_{k+1}} + \len{u_{k+1}} + (n_{k} - m_{k} - 1 + \bbone_{r_{k}})\len{v_{k}}}{2(r_{k+1}-1)\len{v_{k+1}} + \len{u_{k+1}} + \len{v_{k+1}} + (m_{k}+2)\len{v_{k}}}  \\
&= \frac{3}{2} + \frac{\frac{1}{2}\len{u_{k+1}}-\frac{1}{2}\len{v_{k+1}} + (n_{k} - \frac{3}{2}m_{k} - 2 + \bbone_{r_{k}})\len{v_{k}}}{(2r_{k+1}-1)\len{v_{k+1}} + \len{u_{k+1}} + (m_{k}+2)\len{v_{k}}}
\end{align*}
and, as $\len{u_{k+1}} - \len{v_{k+1}} = (n_{k}-m_{k})\len{v_{k}}$, therefore $\frac{3}{2}n_{k} - 2m_{k} - 2 + \bbone_{r_{k}} < 0$.  As $n_{k}$ is an integer, then $3n_{k} + 2\cdot\bbone_{r_{k}} \leq 4m_{k} + 3$ so in particular $n_{k} \leq \frac{4}{3}m_{k} + 1$.

Assume that $n_{k} \geq \frac{3m_{k}+1}{2}$ as otherwise we are in case (\ref{rk3/2}).  Then $\frac{9m_{k}+3}{2} + 2\cdot\bbone_{r_{k}} \leq 4m_{k} + 3$ meaning that $\frac{m_{k}}{2} + 2 \cdot \bbone_{r_{k}} \leq \frac{3}{2}$.  So $r_{k} = 0$ and $m_{k} \leq 3$.  The only possibilities for $(m_{k},n_{k})$ are then $(3,5)$ or $(1,2)$ since $3n_{k} \leq 4m_{k} + 3$ (and $(2,3)$ is ruled out by the assumption that $n_{k} \geq \frac{1}{2}(3m_{k}+1)$).

We now estimate using Corollary \ref{sum} again, knowing that $r_{k} = 0$.
By Corollary \ref{sum} and Lemma \ref{pandsbounds},
\begin{align*}
P_{k} &>
1 + \frac{(r_{k+1}-1)\len{v_{k+1}} + \len{u_{k+1}} + (n_{k}-m_{k}-1)\len{v_{k}} + (n_{k-1} - m_{k-1} - 1 + \bbone_{r_{k-1}})\len{v_{k-1}}}{2(r_{k+1}-1)\len{v_{k+1}} + \len{u_{k+1}} + \len{v_{k+1}} + m_{k}\len{v_{k}} + (m_{k-1}+2)\len{v_{k-1}}} \\
&= \frac{3}{2} + \frac{\frac{1}{2}\len{u_{k+1}} - \frac{1}{2}\len{v_{k+1}} + (n_{k} - \frac{3}{2}m_{k} - 1)\len{v_{k}} + (n_{k-1} - \frac{3}{2}m_{k-1} - 2 + \bbone_{r_{k-1}})\len{v_{k-1}}}{(2r_{k+1}-1)\len{v_{k+1}} + \len{u_{k+1}} + m_{k}\len{v_{k}} + (m_{k-1}+2)\len{v_{k-1}}}
\end{align*}
and, as $\len{u_{k+1}} - \len{v_{k+1}} = (n_{k}-m_{k})\len{v_{k}}$, then
\[
\left(\frac{3}{2}n_{k} - 2m_{k} - 1\right)\len{v_{k}} + (n_{k-1} - \frac{3}{2}m_{k-1} - 2 + \bbone_{r_{k-1}})\len{v_{k-1}} < 0. \tag{$\dagger$}
\]

Consider first when $(m_{k},n_{k}) = (3,5)$.  Then $(\dagger)$ gives that $\frac{1}{2}\len{v_{k}} + (n_{k-1} - \frac{3}{2}m_{k-1} - 2 + \bbone_{r_{k-1}})\len{v_{k-1}} < 0$ and, as $\len{v_{k}} > m_{k-1}\len{v_{k-1}}$, then $n_{k-1} - m_{k-1} - 2 + \bbone_{r_{k-1}} < 0$ meaning $r_{k-1} = 0$ and $n_{k-1} = m_{k-1} + 1$ so we are in case (\ref{rk35}).

Assume from here on that $(m_{k},n_{k}) = (1,2)$.  If $n_{k-1} \leq 2m_{k-1}$ then we are in case (\ref{rk12a}) so we may also assume $n_{k-1} = 2m_{k-1} + 1 + a$ for some $a \geq 0$.  Then $(\dagger)$ gives that $n_{k-1} - \frac{3}{2}m_{k-1} - 2 + \bbone_{r_{k-1}} < 0$ meaning that  that $\frac{1}{2}m_{k-1} + a - 1 + \bbone_{r_{k-1}} < 0$.  Then $r_{k-1} = 0$ and $a = 0$ and $m_{k-1} = 1$ and so $n_{k-1} = 3$.
By Corollary \ref{sum}, as $n_{k} - m_{k} - 1 = 0$ and $n_{k-1} - m_{k-1} - 1 = 1$ and $r_{k} = 0$ and $r_{k-1} = 0$, and Lemma \ref{pandsbounds},
\begin{align*}
P_{k} &> 1 + \frac{(r_{k+1}-1)\len{v_{k+1}} + \len{u_{k+1}} + \len{v_{k-1}} + (n_{k-2} - m_{k-2} - 1 + \bbone_{r_{k-2}})\len{v_{k-2}}}{2(r_{k+1}-1)\len{v_{k+1}} + \len{u_{k+1}} + \len{v_{k+1}} + \len{v_{k}} + \len{v_{k-1}} + (m_{k-2} + 2)\len{v_{k-2}}} \\
&= \frac{3}{2} + \frac{\frac{1}{2}\len{u_{k+1}} - \frac{1}{2}\len{v_{k+1}} - \frac{1}{2}\len{v_{k}} + \frac{1}{2}\len{v_{k-1}} + (n_{k-2} - \frac{3}{2}m_{k-2} - 2 + \bbone_{r_{k-2}})\len{v_{k-2}}}{(2r_{k+1}-1)\len{v_{k+1}} + \len{u_{k+1}} + \len{v_{k}} + \len{v_{k-1}} + (m_{k-2}+2)\len{v_{k-2}}}
\end{align*}
and, as $\len{u_{k+1}} - \len{v_{k+1}} = \len{v_{k}}$ and $\len{v_{k-1}} > m_{k-2}\len{v_{k-2}}$, then $n_{k-2} - m_{k-2} - 2 + \bbone_{r_{k-2}} < 0$.  Therefore $r_{k-2} = 0$ and $n_{k-2} = m_{k-2} + 1$, putting us in case (\ref{rk12b}).
\end{proof}

\begin{remark}
Any specific substitution of the form $\tau_{m,n,r}$ in Corollary \ref{taus}, with parameters compatible with Propositions \ref{svp} and \ref{rk}, can be used infinitely often in the construction of a subshift with $\limsup \frac{p(q)}{q} < 1.5$.  Indeed, preceding that specific substitution by enough substitutions of the form $\tau_{m,m+1,0}$ for appropriate $m$ will provide such a subshift; we do not elaborate further as we do not make use of this.
\end{remark}

\begin{remark}
The above reasoning can also be used to show that certain substitutions are ruled out at various complexity cutoffs (we omit proofs since we will not use these facts):
\begin{itemize}
\item $\tau_{m,2m+2,0}$ cannot occur when $\limsup \frac{p(q)}{q} < 1.4$;
\item $\tau_{n,m,r}$, $r > 0$ cannot occur when $\limsup \frac{p(q)}{q} < \frac{4}{3}$ (c.f.~\cite{creutzpavlov} and \cite{aber}); and
\item $\tau_{m,2m+1,0}$ cannot occur when $\limsup \frac{p(q)}{q} < 1.25$.
\end{itemize}
We also note that Aberkane proved a slightly different substitutive structure for $\limsup \frac{p(q)}{q} < \frac{4}{3}$ in \cite{aber}.
\end{remark}

Our last fact regarding the substitutive structure is that $\limsup \frac{p(q)}{q} < 1.5$ imposes a bound on $\frac{n_{k}}{m_{k}}$.  Recall that $\frac{p(q)}{q} < (1.5 - \nu)q$ for all $q \geq \len{v_{0}}$.

\begin{proposition}\label{thedelta}
Let $X$ be an infinite minimal subshift with $\limsup \frac{p(q)}{q} < 1.5$ and $\nu > 0$ and the sequences of words $\{ u_{k} \}, \{ v_{k} \}, \{ p_{k} \}, \{ s_{k} \}$ and integers $\{ m_{k} \}, \{ n_{k} \}, \{ r_{k} \}$ be those from Proposition \ref{words}.

There exists $\delta > 0$ and $N \in \mathbb{N}$ such that for all $k \geq 2$ where $m_{k} \geq N$, if $r_{k+1} > 0$ then $n_{k} < \frac{3-\delta}{2}m_{k}$ and if $r_{k+1} = 0$ then $n_{k} < (2 - \delta)m_{k}$.
\end{proposition}
\begin{proof}
By Proposition \ref{rk}, if $r_{k+1} > 0$ then $n_{k} \leq \frac{4}{3}m_{k} + 1$ so if $m_{k} \geq 8$ then $\frac{n_{k}}{m_{k}} \leq \frac{4}{3} + \frac{1}{8} = \frac{3}{2} - \frac{1}{24}$.  Take $\frac{1}{24} > \delta > 0$ and $N \geq 8$ such that $\frac{\delta N + 3}{2(2 - \delta)N + 2} \leq \nu$.

Suppose that there exists $k$ with $r_{k} = 0$ and $m_{k} \geq N$ and $n_{k} \geq (2 - \delta)m_{k}$.  Then, since $\frac{\frac{1}{2}n - m - \frac{3}{2}}{n+1}$ is increasing with $n$, by Corollary \ref{sum} and Remark \ref{pands},
\begin{align*}
\frac{p(\len{s_{k}v_{k}^{n_{k}-2}p_{k}})}{\len{s_{k}v_{k}^{n_{k}-2}p_{k}}}
&> 1 + \frac{(n_{k} - m_{k} - 1)\len{v_{k}}}{(n_{k}-2)\len{v_{k}} + 3\len{v_{k}}} 
= \frac{3}{2} + \frac{\frac{1}{2}n_{k} - m_{k} - \frac{3}{2}}{n_{k} + 1} 
\geq \frac{3}{2} - \frac{\delta m_{k} + 3}{2(2-\delta)m_{k} + 2}.
\end{align*}
Therefore, as $\frac{\delta m + 3}{2(2-\delta)m+2}$ is decreasing with $m$,
$
\frac{p(\len{s_{k}v_{k}^{n_{k}-2}p_{k}})}{\len{s_{k}v_{k}^{n_{k}-2}p_{k}}} > \frac{3}{2} - \frac{\delta N + 3}{2(2-\delta)N + 2} \geq \frac{3}{2} - \nu
$
contradicting that $p(q) < (1.5-\nu)q$ for all $q \geq \len{v_{0}}$.
\end{proof}

\section{Discrete spectrum and balancedness}\label{discspecsec}

In this section, we use the substitutive structure and inequalities established in Section~\ref{ineqs} to prove that infinite minimal low complexity subshifts have (measurably) discrete spectrum and are balanced on words.

Although it will not be directly necessary for our proofs, it is instructive to describe our S-adic structure in terms of incidence matrices and associated products, as is standard in, for instance, papers on S-adic Pisot conjecture like \cite{bertheetal}.
 
For any substitution $\tau: A \rightarrow A$, the \textbf{incidence matrix} of $\tau$ is a square $|A| \times |A|$ matrix $M(\tau)$ with $m_{ij}$ equal to $|\tau(j)|_i$, where $|w|_a$ represents the number of times that a letter $a$ appears in a word $w$. For instance, in our setting,
\[
M(\tau_{m,n,0}) = \begin{pmatrix} m-1 & n-1\\ 1 & 1 \end{pmatrix} \textrm{ and } 
M(\tau_{m,n,r}) = \begin{pmatrix} (m-1)+(r-1) & (n-1)+(r-1)\\ 2 & 2 \end{pmatrix}.
\] 
For any subshift $X$ as in Proposition~\ref{words} with substitutions as in Corollary~\ref{taus}, define 
$M_k = M(\tau_{m_k,n_k,r_k})$ for $k \geq 0$. Then,
\[
D_k := |\det\left(M_{0} \cdots M_{k-1}\right)| = \prod_{i=0}^{k-1} |\det(M_i)| = 2^{\sum_{j=0}^{k-1}\bbone_{r_{j}}} \prod_{j=0}^{k-1} (n_{j} - m_{j}).
\]
In the case where $\pi$ does not appear (i.e., is the identity substitution), the entries of the matrix $M_0 \cdots M_{k-1}$ simply measure the numbers of $0$,$1$ appearing in $u_k, v_k$ respectively. In the general case, since $\pi$ is applied after all 
$\tau_k$, the entries measure the numbers of $u_0, v_0$ concatenated in the creation of $u_k, v_k$.

In some sense, this quantity $D_k$ can then be thought of both as a quantitative measure of differences involving the words $u_k$ and $v_k$ (see Lemma~\ref{dvu}), and also as a geometric property of the actual linear operator/cocycle corresponding to the application of the first $k$ substitutions (the latter viewpoint is useful when applying multiplicative ergodic theorems, as in \cite{BST23}).

\subsection{Bounds on the determinants $D_k$}

The key ingredient in the proof of discrete spectrum is a proof of exponential decay $D_k/|v_k|$, which will be used literally as a means to prove mean almost periodicity, but which can also be thought of as a Lyapunov exponent estimation in the spirit of \cite{bertheetal} and \cite{BST23}. We note that similar ideas were used in Host's \cite{MR873430} proof of the existence of eigenfunctions for subshifts coming from certain single substitutions.

\begin{proposition}\label{Pmain}
Let $X$ be an infinite minimal subshift with $\limsup \frac{p(q)}{q} < 1.5$.  Let $m_{k}$, $n_{k}$ and $r_{k}$ be the sequences from Proposition \ref{words}.  Then there exists $\epsilon_{k}$ with $\sum_{k=0}^{\infty} \epsilon_{k} < \infty$ such that for all $k$,
\[
\frac{D_k}{\len{v_k}} = \frac{2^{\sum_{j=0}^{k-1}\bbone_{r_{j}}} \prod_{j=0}^{k-1} (n_{j} - m_{j})}{\len{v_{k}}} < \epsilon_{k}.
\]
\end{proposition}

In this proposition, we can see directly the heuristic mentioned earlier; large values of $n_j - m_j$ (relative to $m_j$) and nonzero values of $r_j$ induce rapid growth of $D_k$, and our proof in some sense relies on showing low word complexity forces these to be `rare enough' to still allow for exponential decay.

The proof of Proposition~\ref{Pmain} will first require a few technical lemmas.
Throughout this section, let $X$ be an infinite minimal subshift with $\limsup \frac{p(q)}{q} < 1.5$ and $u_{k}$ and $v_{k}$ be the words from Proposition \ref{words}.

We define some auxiliary sequences which will be crucial throughout the remainder of the paper. For all $k \geq 0$, define
\begin{equation}\label{ab}
a_{k+1} = 2^{\bbone_{r_{k+1}}}(n_{k} - m_{k}), \quad b_{k} = m_{k} + r_{k}, \textrm{ and } a_{0} = 2^{\bbone_{r_{0}}}.
\end{equation}

Also set $\beta_{k} = \frac{a_{k+1}\len{v_{k}}}{\len{v_{k+1}}} > 0$ for $k \geq 0$.

\begin{lemma}\label{a2}
For all $k \geq 2$, $a_{k+1} \leq b_{k} + 2$.  If $a_{k+1} = b_{k} + 2$ then $r_{k+1} = 0$ and $n_{k} = 2m_{k} + 2$.
\end{lemma}
\begin{proof}
By Proposition \ref{svp},  $n_{k} \leq 2m_{k} + 2$ so if $r_{k+1} = 0$ then $a_{k+1} = n_{k} - m_{k} \leq m_{k} + 2 \leq b_{k} + 2$.

If $r_{k+1} > 0$ then by Proposition \ref{rk}, either $n_{k} \leq \frac{3}{2} m_{k}$ or $(m_{k},n_{k}) = (1,2)$ or $(m_{k},n_{k}) = (3,5)$, all of which preclude $a_{k+1} = b_{k} + 2$.  If $n_{k} \leq \frac{3}{2}m_{k}$ then $a_{k+1} = 2(n_{k} - m_{k}) \leq m_{k} \leq b_{k}$.  If $(m_{k},n_{k}) = (1,2)$ then $a_{k+1} = 2(2-1) = 2 = m_{k} + 1 \leq b_{k} + 1$.   If $(m_{k},n_{k}) = (3,5)$ then $a_{k+1} = 2(5-3) = 4 = 3 + 1 \leq b_{k} + 1$.
\end{proof}

\begin{lemma}\label{vrec}
For all $k \geq 1$,
\[
\len{v_{k+1}} = b_{k}\len{v_{k}} + a_{k}\len{v_{k-1}}.
\]
\end{lemma}
\begin{proof}
Since $\len{u_{k+1}} - \len{v_{k+1}} = (n_{k} - m_{k})\len{v_{k}}$,
if $r_{k} > 0$ then
\begin{align*}
\len{v_{k+1}} &= (m_{k}+r_{k}-2)\len{v_{k}} + 2\len{u_{k}} = (m_{k}+r_{k})\len{v_{k}} + 2(\len{u_{k}} - \len{v_{k}}) \\
&= (m_{k}+r_{k})\len{v_{k}} + 2(n_{k-1}-m_{k-1})\len{v_{k-1}} = b_{k}\len{v_{k}} + a_{k}\len{v_{k-1}}
\end{align*}
and if $r_{k} = 0$ then $\len{v_{k+1}} = (m_{k}-1)\len{v_{k}} + \len{u_{k}} = m_{k}\len{v_{k}} + (n_{k-1} - m_{k-1})\len{v_{k-1}} = b_{k}\len{v_{k}} + a_{k}\len{v_{k-1}}$.
\end{proof}

\begin{lemma}\label{2.4}
For $k \geq 1$,
\[
\beta_{k} = \frac{a_{k+1}}{b_{k} + \beta_{k-1}}.
\]
\end{lemma}
\begin{proof}
$
\beta_{k} = \frac{a_{k+1}\len{v_{k}}}{b_{k}\len{v_{k}} + a_{k}\len{v_{k-1}}}
= \frac{a_{k+1}}{b_{k} + a_{k}\frac{\len{v_{k-1}}}{\len{v_{k}}}} = \frac{a_{k+1}}{b_{k} + \beta_{k-1}}
$.
\end{proof}

The next several lemmas establish that the $\beta_{k}$ or products of them are always less than one, the first step in establishing the desired exponential decay.

\begin{lemma}\label{2n}
If $k \geq 1$ and $a_{k+1} \leq b_{k}$ then $\beta_{k} < 1$.
\end{lemma}
\begin{proof}
Since $\beta_{k-1} > 0$, by Lemma \ref{2.4},
$
\beta_{k} = \frac{a_{k+1}}{b_{k} + \beta_{k-1}} < \frac{a_{k+1}}{b_{k}} \leq 1
$.
\end{proof}

\begin{lemma}\label{trick}
If $k \geq 1$ and $a_{k+1} \leq b_{k} + 1$ and $\beta_{k-1} < 1$ then $\beta_{k}\beta_{k-1} < 1 - \frac{1 - \beta_{k-1}}{2} < 1$.
\end{lemma}
\begin{proof}
Since $\beta_{k-1} < 1$, $1 - \beta_{k-1} > 0$ and since $\frac{b(1-\beta_{k-1})}{b+1}$ is then increasing with $b$, by Lemma \ref{2.4},
\[
\beta_{k}\beta_{k-1} = \frac{a_{k+1}\beta_{k-1}}{b_{k} + \beta_{k-1}} \leq \frac{(b_{k}+1)\beta_{k-1}}{b_{k} + \beta_{k-1}}
= 1 - \frac{b_{k}(1 - \beta_{k-1})}{b_{k} + \beta_{k-1}} < 1 - \frac{b_{k}(1 - \beta_{k-1})}{b_{k}+1} \leq 1 - \frac{1 - \beta_{k-1}}{2} < 1. \qedhere
\]
\end{proof}

\begin{lemma}\label{2m+1}
If $k \geq 2$ and $n_{k} = 2m_{k} + 1$ and $n_{k-1} \leq 2m_{k-1}$ then $a_{k} \leq b_{k-1}$ and $\beta_{k}\beta_{k-1} < 1 - \frac{1 - \beta_{k-1}}{2} < 1$.
\end{lemma}
\begin{proof}
By Proposition \ref{rk}, $r_{k+1} = 0$.
By Proposition \ref{svp} (cases (\ref{svp+1a}) and (\ref{svp+1b})), $r_{k} = 0$ so by Lemma \ref{2n}, $\beta_{k-1} < 1$.  As $a_{k+1} = n_{k} - m_{k} = m_{k} + 1 = b_{k} + 1$, Lemma \ref{trick} gives the claim.
\end{proof}

\begin{lemma}\label{2m+1b}
If $k \geq 2$ and $n_{k} = 2m_{k} + 1$ and $n_{k-1} > 2m_{k-1}$ then $a_{k-1} \leq b_{k-2}$ and $\beta_{k}\beta_{k-1}\beta_{k-2} < 1 - \frac{1 - \beta_{k-2}}{2} < 1$.
\end{lemma}
\begin{proof}
By Proposition \ref{rk}, $r_{k+1} = 0$.  By Proposition \ref{svp} 
case (\ref{svp+1b}), $r_{k} = 0$ and $m_{k-1} = 1$ and $n_{k-1} = 3$ and $r_{k-1} = 0$ and $n_{k-2} = m_{k-2} + 1$.  So $a_{k+1} = b_{k} + 1$ and $a_{k} = 2$ and $b_{k-1} = 1$ and $a_{k-1} = 1$.  By Lemma \ref{2n}, then $\beta_{k-2} < 1$.  Then by Lemma \ref{2.4},
\begin{align*}
\beta_{k}\beta_{k-1}\beta_{k-2} &= \frac{b_{k}+1}{b_{k} + \frac{2}{1 + \beta_{k-2}}} \frac{2}{1 + \beta_{k-2}} \beta_{k-2} 
= \frac{(2b_{k} + 2)\beta_{k-2}}{b_{k} + b_{k}\beta_{k-2} + 2}
= 1 - \frac{(b_{k}+2)(1 - \beta_{k-2})}{b_{k} + b_{k}\beta_{k-2} + 2}
\end{align*}
and since $\beta_{k-2} < 1$ implies $b_{k} + 2 + b_{k}\beta_{k-2} < 2b_{k} + 2 < 2(b_{k} + 2)$, then
$
\beta_{k}\beta_{k-1}\beta_{k-2} < 1 - \frac{1 - \beta_{k-2}}{2} < 1
$.
\end{proof}

\begin{lemma}\label{rcase}
If $k \geq 2$ and $r_{k+1} > 0$, $n_{k} > \frac{3}{2}m_{k}$ and $n_{k-1} \leq 2m_{k-1}$ then $a_{k} \leq b_{k-1}$ and $\beta_{k}\beta_{k-1} < 1 - \frac{1 - \beta_{k-1}}{2} < 1$.
\end{lemma}
\begin{proof}
By Proposition \ref{rk}, $r_{k} = 0$ and either $(m_{k},n_{k})=(1,2)$ or $(m_{k},n_{k})=(3,5)$.  By Lemma \ref{2n}, since $n_{k-1} \leq 2m_{k-1}$, then $\beta_{k-1} < 1$.  When $m_{k}=1,n_{k}=2$, we have $a_{k+1} = 2$ and $b_{k} = 1$ and when $m_{k}=3,n_{k}=5$, we have $a_{k+1} = 4$ and $b_{k} = 3$ so Lemma \ref{trick} gives the claim.
\end{proof}

\begin{lemma}\label{rcaseb}
If $k \geq 2$ and $r_{k+1} > 0$, $n_{k} > \frac{3}{2}m_{k}$ and $n_{k-1} > 2m_{k-1}$ then $a_{k-1} \leq b_{k-2}$ and $\beta_{k}\beta_{k-1}\beta_{k-2} < 1 - \frac{3(1 - \beta_{k-2})}{4} < 1$.
\end{lemma}
\begin{proof}
By Proposition \ref{rk} case (\ref{rk12b}), $m_{k} = 1$ and $n_{k} = 2$ and $m_{k-1} = 1$ and $n_{k-1} = 3$ and $n_{k-2} = m_{k-2} + 1$ and $r_{k} = 0$ and $r_{k-1} = 0$.  So $a_{k+1} = 2$ and $b_{k} = 1$ and $a_{k} = 2$ and $b_{k-1} = 1$ and, by Lemma \ref{2n}, $\beta_{k-2} < 1$.  Then
\begin{align*}
\beta_{k}\beta_{k-1}\beta_{k-2} &= \frac{2}{1 + \frac{2}{1 + \beta_{k-2}}} \frac{2}{1 + \beta_{k-2}} \beta_{k-2}
= \frac{4\beta_{k-2}}{3 + \beta_{k-2}} = 1 - \frac{3 - 3\beta_{k-2}}{3 + \beta_{k-2}}
\end{align*}
and since $\beta_{k-2} < 1$ implies $3 + \beta_{k-2} < 4$, we have $\beta_{k}\beta_{k-1}\beta_{k-2} < 1 - \frac{3(1 - \beta_{k-2})}{4} < 1$.
\end{proof}

\begin{lemma}\label{2m+2lemma}
If $k \geq 2$ and $n_{k} = 2m_{k}+2$ then $a_{k-1} \leq b_{k-2}$ and $\beta_{k}\beta_{k-1}\beta_{k-2} < 1 - \frac{2(1 - \beta_{k-2})}{3} < 1$.
\end{lemma}
\begin{proof}
By Proposition \ref{rk}, $r_{k+1} = 0$.
By Proposition \ref{svp} case (\ref{svp+2}), $r_{k} = 0$ and $r_{k-1} = 0$ and $n_{k-1} = m_{k-1} + 1$, so $a_{k} = 1$, and $n_{k-2} \leq \frac{4}{3}m_{k-2} + 1$ so $n_{k-2} \leq 2m_{k-2}$.  Therefore $a_{k+1} = b_{k} + 2$ and $a_{k} = 1$ and $\beta_{k-2} < 1$ by Lemma \ref{2n}.  Observe that
\[
\beta_{k}\beta_{k-1}\beta_{k-2} = \frac{b_{k} + 2}{b_{k} + \frac{1}{b_{k-1} + \beta_{k-2}}} \frac{1}{b_{k-1} + \beta_{k-2}}\beta_{k-2}
= \frac{(b_{k} + 2)\beta_{k-2}}{b_{k-1}b_{k} + b_{k}\beta_{k-2} + 1}
\]
which is decreasing in both $b_{k}$ and $b_{k-1}$ so
\[
\beta_{k}\beta_{k-1}\beta_{k-2} \leq \frac{(b_{k}+2)\beta_{k-2}}{b_{k} + b_{k}\beta_{k-2}+1} \leq \frac{3\beta_{k-2}}{2 + \beta_{k-2}} = 1 - \frac{2(1 - \beta_{k-2})}{2 + \beta_{k-2}} < 1 - \frac{2(1 - \beta_{k-2})}{3}. \qedhere
\]
\end{proof}

We now combine all of the above lemmas bounding $\beta_{k}$ or products of them by $1$ into a single statement.

\begin{lemma}\label{allbetas}
For every $k\geq 2$ there exists $0 \leq i_{k} \leq 2$ such that $a_{k-i_{k}+1} \leq b_{k-i_{k}}$ and $\prod_{j=k-i_{k}}^{k} \beta_{j} < 1 - \frac{1}{2}(1 - \beta_{k-i_{k}}) < 1$.
\end{lemma}
\begin{proof}
For $k$ such that $a_{k+1} \leq b_{k}$, set $i_{k} = 0$.  Lemma \ref{2n} gives that $\beta_{k} < 1$.  Then $1 - \frac{1}{2}(1 - \beta_{k-i_{k}}) = \frac{1}{2} + \frac{\beta_{k}}{2} > \beta_{k} = \prod_{j=k-i_{k}}^{k} \beta_{j}$.

By Lemma \ref{a2}, $a_{k+1} \leq b_{k} + 2$ and they are only equal when $n_{k} = 2m_{k} + 2$ and $r_{k+1} = 0$.  For $k$ such that $a_{k+1} = b_{k} + 2$, set $i_{k} = 2$ and the claim follows from Lemma \ref{2m+2lemma}.

Let $k$ such that $a_{k+1} = b_{k} + 1$.  Consider first when $r_{k+1} > 0$.  Then $2(n_{k} - m_{k}) = b_{k} + 1 \geq m_{k} + 1$ so $n_{k} > \frac{3}{2}m_{k}$.  If $n_{k-1} \leq 2m_{k-1}$ then set $i_{k} = 1$ and the claim follows from Lemma \ref{rcase}; if $n_{k-1} > 2m_{k-1}$ then set $i_{k} = 2$ and the claim follows from Lemma \ref{rcaseb}.

Now consider when $r_{k+1} = 0$ so $n_{k} - m_{k} = m_{k} + 1$.  If $n_{k-1} \leq 2m_{k-1}$ then set $i_{k} = 1$ and the claim follows from Lemma \ref{2m+1}; if $n_{k-1} > 2m_{k-1}$ then set $i_{k} = 2$ and the claim follows from Lemma \ref{2m+1b}.
\end{proof}

Our next pair of lemmas reframes the
bound on $\frac{n_{k}}{m_{k}}$ established in Proposition~\ref{thedelta} in terms of $a_{k},b_{k}$, and $\beta_{k}$.

\begin{lemma}\label{toinfty}
For $k \geq 2$, if $a_{k + 1} \leq b_{k}$ and $\beta_{k} \geq 1 - \delta$ and $\beta_{k-1} \geq 1 - \delta$ then $b_{k} \geq \frac{(1 - \delta)^{2}}{\delta}$.
\end{lemma}
\begin{proof}
Since $1 - \delta \leq \beta_{k} \leq \frac{b_{k}}{b_{k} + \beta_{k-1}} \leq \frac{b_{k}}{b_{k} + 1 - \delta}$, then $b_{k}(1 - \delta) + (1-\delta)^{2} \leq b_{k}$ so
$(1 - \delta)^{2} \leq \delta b_{k}$.
\end{proof}

\begin{lemma}\label{delta}
There exists $\delta > 0$ such that for $k \geq 2$, if $a_{k+1} \leq b_{k}$ then at least one of $\beta_{k} < 1 - \delta$ or $\beta_{k-1} < 1 - \delta$.
\end{lemma}
\begin{proof}
By Proposition \ref{thedelta}, there exists $\delta_{0} > 0$ and $N$ such that for $k \geq 2$ and $m_{k} \geq N$, if $r_{k+1} > 0$ then $n_{k} < \frac{3-\delta_{0}}{2}m_{k}$ in which case $a_{k+1} = 2(n_{k}-m_{k}) < (1 - \delta_{0})m_{k} \leq (1 - \delta_{0})b_{k}$ and if $r_{k} = 0$ then $n_{k} < (2 - \delta_{0})m_{k}$ in which case $a_{k+1} = n_{k} - m_{k} < (1 - \delta_{0})m_{k} \leq (1 - \delta_{0})b_{k}$.
So for $k$ such that $m_{k} \geq N$, by Lemma \ref{2.4}, $\beta_{k} = \frac{a_{k+1}}{b_{k} + \beta_{k-1}} < 1 - \delta_{0}$.

Let $0 < \delta \leq \delta_{0}$ such that $\frac{(1-\delta)^{2}}{2\delta} \geq N$.  Let $k$ such that $a_{k+1} \leq b_{k}$ and $\beta_{k-1} \geq 1- \delta$.  By the above, if $m_{k} \geq N$ then $\beta_{k} < 1 - \delta_{0} \leq 1 - \delta$.  If $m_{k} < N$ then, as $r_{k} < m_{k}$, $b_{k} < 2m_{k} < 2N \leq \frac{(1-\delta)^{2}}{\delta}$ so, by Lemma \ref{toinfty}, $\beta_{k} < 1 - \delta$.
\end{proof}

We are now ready to prove exponential decay of the $\beta_k$, from which Proposition~\ref{Pmain} quickly follows.

\begin{lemma}\label{betaconvexp}
There exists $0 < \kappa < 1$ and $C > 0$ so that for all $k$, we have
$
\prod_{j=0}^{k} \beta_{j} < C\kappa^{k}.
$
\end{lemma}
\begin{proof}
By Lemma \ref{delta}, there exists $\delta > 0$ such that for $k \geq 2$, if $a_{k+1} \leq b_{k}$ then at least one of $\beta_{k} < 1 - \delta$ or $\beta_{k-1} < 1 - \delta$.
Let $k \geq 3$ such that $\beta_{k} \geq 1 - \delta$.  By Lemma \ref{allbetas}, there exists $0 \leq i_{k} \leq 2$ such that $\prod_{j=k-i_{k}}^{k}\beta_{j} < 1 - \frac{1}{2}(1 - \beta_{k-i_{k}}) < 1$ and $a_{k-i_{k}+1} \leq b_{k - i_{k}}$.

If $i_{k} = 0$ then we have $a_{k+1} \leq b_{k}$ so $\beta_{k-1} < 1 - \delta$ (since $\beta_{k} \geq 1 - \delta$). By Lemma \ref{2n}, we have $\beta_{k} < 1$ and therefore $\beta_{k}\beta_{k-1} < 1 - \delta$.

If $i_{k} > 0$ then $a_{k-i_{k}+1} \leq b_{k-i_{k}}$ so at least one of $\beta_{k-i_{k}} < 1 - \delta$ or $\beta_{k-i_{k}-1} < 1 - \delta$ holds.  If $\beta_{k-i_{k}} < 1 - \delta$ then $\prod_{j=k-i_{k}}^{k} \beta_{j} < 1 - \frac{\delta}{2}$ and if $\beta_{k-i_{k}-1} < 1 - \delta$ then, as Lemma \ref{2n} implies $\beta_{k-i_{k}} < 1$, we have $\prod_{j=k-i_{k}-1}^{k}\beta_{j} < \beta_{k-i_{k}-1} < 1 - \delta$.

So for all $k \geq 5$, there exists $0 \leq i_{k}^{\prime} \leq 3$ such that $\prod_{j=k-i_{k}^{\prime}}^{k} \beta_{j} < 1 - \frac{\delta}{2}$.  Set $\kappa_{0} = 1 - \frac{\delta}{2}$.

Let $C > \max\{ \prod_{j=0}^{k}\beta_{j}\kappa_{0}^{-k/4} : 0 \leq k \leq 5 \}$.  Then $\prod_{j=0}^{k} \beta_{j} < C\kappa_{0}^{k/4}$ for $0 \leq k \leq 5$.

Assume now that for some $k \geq 6$ we have $\prod_{j=0}^{k^{\prime}} \beta_{j} < C\kappa_{0}^{k^{\prime}/4}$ for all $k^{\prime} < k$.  Then, since $i_{k}^{\prime} \leq 3$,
\[
\prod_{j=0}^{k}\beta_{j} = \Big{(}\prod_{j=k-i_{k}^{\prime}}^{k}\beta_{j}\Big{)} \Big{(}\prod_{j=0}^{k-i_{k}^{\prime}-1} \beta_{j}\Big{)}
< \kappa_{0} \cdot C \kappa_{0}^{(k-i_{k}^{\prime} - 1)/4}
= C \kappa_{0}^{(k + 4 - i_{k}^{\prime} - 1)/4}
\leq C \kappa_{0}^{k/4}
\]
so the claim follows by induction and setting $\kappa = \kappa_{0}^{1/4}$.
\end{proof}

\begin{proof}[Proof of Proposition \ref{Pmain}]
Since $a_{k+1} = 2^{\bbone_{r_{k+1}}}(n_{k} - m_{k})$, 
\[
2^{\sum_{j=0}^{k}\bbone_{r_{j}}}\prod_{j=0}^{k-1} (n_{j} - m_{j}) = a_0 \prod_{j=0}^{k-1} a_{j+1}
= \frac{a_0 \len{v_{k}}}{\len{v_{0}}} \prod_{j=0}^{k-1} \frac{a_{j+1}\len{v_{j}}}{\len{v_{j+1}}}
= \frac{a_0 \len{v_{k}}}{\len{v_{0}}} \prod_{j=0}^{k-1} \beta_{j}.
\]
So, by Lemma \ref{betaconvexp}, there exists $C > 0$ and $0 < \kappa < 1$ such that for all $k$,
\[
\frac{2^{\sum_{j=0}^{k}\bbone_{r_{j}}}\prod_{j=0}^{k-1} (n_{j} - m_{j})}{\len{v_{k}}}
= \frac{a_0}{\len{v_{0}}}\prod_{j=0}^{k-1} \beta_{j}
< \frac{Ca_0}{\len{v_{0}}} \kappa^{k-1}
\]
meaning $\epsilon_{k} := \frac{Ca_0}{\len{v_{0}}} \kappa^{k-1}$ proves the claim.
\end{proof}

There are two interpretations (in addition to the motivation via determinants) of the quantity $D_k$ that will be useful for our proofs. The first (which will be useful for discrete spectrum) is a bound on how much the words $v_{k}u_{k}$ and $u_{k}v_{k}$ differ, which we will use to verify mean almost periodicity (see \cite{MR2569181}). 

\begin{lemma}\label{dvu}
For all $k \geq 0$, the words $u_{k}v_{k}$ and $v_{k}u_{k}$ differ on a number of locations less than
\[
2\len{u_0} D_k = 2\len{u_{0}} 2^{\sum_{j=0}^{k-1}\bbone_{r_{j}}}(n_{0} - m_{0})(n_{1}-m_{1})\cdots (n_{k-1}-m_{k-1}).
\]
\end{lemma}
\begin{proof}
Let $d$ be the Hamming distance: the metric defined on pairs of words of the same length by $d(w,x) = |\{ 0 \leq t < \len{w} : w_{t} \neq x_{t} \}|$.  Note that for words $w$ and $x$ of the same length and any words $p$ and $s$, we have $d(pws,pxs) = d(w,x)$.

Since $\len{u_{0}v_{0}} < 2\len{u_{0}}$, the claim is immediate for $k = 0$.  Assume the claim holds for $k$.

Consider first the case when $r_{k} = 0$.  Then, using the triangle inequality,
\begin{align*}
d(v_{k+1}u_{k+1},u_{k+1}v_{k+1})
&= d(v_{k}^{m_{k}-1}u_{k}v_{k}^{n_{k}-1}u_{k}, v_{k}^{n_{k}-1}u_{k}v_{k}^{m_{k}-1}u_{k}) \\
&= d(u_{k}v_{k}^{n_{k}-m_{k}}, v_{k}^{n_{k}-m_{k}}u_{k}) \\
&\leq \sum_{j=0}^{n_{k}-m_{k}-1} d(v_{k}^{j}u_{k}v_{k}^{n_{k}-m_{k}-j}, v_{k}^{j+1}u_{k}v_{k}^{n_{k}-m_{k}-j-1}) \\
&= \sum_{j=0}^{n_{k}-m_{k}-1} d(u_{k}v_{k},v_{k}u_{k}) \\
&< (n_{k}-m_{k}) 2\len{u_{0}} 2^{\sum_{j=0}^{k-1}\bbone_{r_{j}}}(n_{0} - m_{0})(n_{1}-m_{1})\cdots (n_{k-1}-m_{k-1}) \\
&= 2\len{u_{0}} 2^{\sum_{j=0}^{k}\bbone_{r_{j}}}(n_{0} - m_{0})(n_{1}-m_{1})\cdots (n_{k}-m_{k}).
\end{align*}
Now consider the case when $r_{k} > 0$.  Here
\begin{align*}
d(v_{k+1}u_{k+1},u_{k+1}v_{k+1})
&= d(v_{k}^{m_{k-1}}u_{k}v_{k}^{r_{k}-1}u_{k}v_{k}^{n_{k}-1}u_{k}v_{k}^{r_{k}-1}u_{k}, v_{k}^{n_{k}-1}u_{k}v_{k}^{r_{k}-1}u_{k}v_{k}^{m_{k}-1}u_{k}v_{k}^{r_{k}-1}u_{k}) \\
&= d(u_{k}v_{k}^{r_{k}-1}u_{k}v_{k}^{n_{k}-m_{k}}, v_{k}^{n_{k}-m_{k}}u_{k}v_{k}^{r_{k}-1}u_{k}) \\
&\leq \sum_{j=0}^{n_{k}-m_{k}-1} \Big{(}d(v_{k}^{j}u_{k}v_{k}^{r_{k}-1}u_{k}v_{k}^{n_{k}-m_{k}-j}, v_{k}^{j}u_{k}v_{k}^{r_{k}}u_{k}v_{k}^{n_{k}-m_{k}-j-1}) \\
&\quad\quad\quad\quad\quad\quad + d(v_{k}^{j}u_{k}v_{k}^{r_{k}}u_{k}v_{k}^{n_{k}-m_{k}-j-1}, v_{k}^{j+1}u_{k}v_{k}^{r_{k}-1}u_{k}v_{k}^{n_{k}-m_{k}-j-1}) \Big{)} \\
&= \sum_{j=0}^{n_{k}-m_{k}-1} 2d(u_{k}v_{k},v_{k}u_{k}) \\
&< 2(n_{k}-m_{k}) 2\len{u_{0}} 2^{\sum_{j=0}^{k-1}\bbone_{r_{j}}}(n_{0} - m_{0})(n_{1}-m_{1})\cdots (n_{k-1}-m_{k-1}) \\
&= 2\len{u_{0}} 2^{\sum_{j=0}^{k}\bbone_{r_{j}}}(n_{0} - m_{0})(n_{1}-m_{1})\cdots (n_{k}-m_{k}).
\end{align*}
Therefore the claim follows by induction.
\end{proof}

Our other observation is that the quantities $D_k$ are precisely the difference between successive convergents of a particular generalized continued fraction.

\begin{lemma}\label{contlem}
If we define any generalized continued fraction
\[
\lambda = \cfrac{a_0}{b_0+\frac{a_1}{b_1+\cdots}},
\]
and define its convergents as $\frac{p_k}{q_k}$, then 
\[
\left| \frac{p_{k-1}}{q_{k-1}} - \frac{p_k}{q_k} \right| = \frac{2^{\bbone_{r_k}} D_k}{q_k q_{k-1}}.
\]
In particular, $p_k/q_k \rightarrow \lambda$. The same convergence holds for any continued fraction $\lambda'$ differing from $\lambda$ on only finitely many digits. 
\end{lemma}

\begin{proof}
It's well known that
\[
\left| \frac{p_{k-1}}{q_{k-1}} - \frac{p_k}{q_k} \right| = \frac{a_0 a_1 a_2 \cdots a_k}{q_k q_{k-1}},
\]
and it was shown in the proof of Proposition~\ref{Pmain} that $2^{\bbone_{r_k}} D_k = a_0 a_1 a_2 \cdots a_k$, proving the first equation. Also, by Lemma~\ref{vrec}, the lengths $|v_k|$ satisfy the same recursion $|v_{k+1}| = b_k |v_k| + a_k |v_{k-1}|$ as the denominators $q_k$. Therefore, $|v_k|/q_k$ is bounded above by a constant, meaning that $\frac{D_k}{q_k q_{k-1}}$ is summable, completing the proof. Clearly changing finitely many digits changes nothing in the argument. 
\end{proof}

In fact, as we will see later (Proposition~\ref{2piqalpha}), continued fractions of this sort are actually what gives rise to our family of irrational additive continuous eigenvalues.

\subsection{Discrete spectrum}

We are now in a position to prove discrete spectrum.

\begin{theorem}\label{discspec}
Every infinite minimal subshift with $\limsup \frac{p(q)}{q} < 1.5$ has discrete spectrum.
\end{theorem}

\begin{proof}
Let $X$ be an infinite minimal subshift with $\limsup \frac{p(q)}{q} < 1.5$.  Let $v_{k}$ and $u_{k}$ be the words from Proposition \ref{words} with corresponding $m_{k}$, $n_{k}$ and $r_{k}$.

A subshift $(X, \sigma)$ is \emph{mean almost periodic} if for all $\epsilon > 0$ and all $x \in X$, there exists a syndetic set $S$ so that for all $s \in S$, $x$ and $\sigma^s x$ differ on a set of locations with upper density\footnote{The upper density of $D \subseteq \mathbb{N}$ is $\limsup_{N,M} \frac{1}{N}|D \cap \{ M+1, M+2, \ldots, M+N\}|$.} less than $\epsilon$. Mean almost periodicity implies discrete spectrum; see e.g.~Theorem 2.8 \cite{MR2569181}.

Let $x \in X$.  Then $x$ can be written as a bi-infinite concatenation of the words $u_{k}$ and $v_{k}$.  Without loss of generality, we may assume that $x$ contains $u_{k+1}$ starting at the origin, since for any $i, j$, the set of locations where
$\sigma^i x$ and $\sigma^j(\sigma^i x)$ differ is just a shift of the set of locations where
$x$ and $\sigma^j x$ differ. Then, decomposing $x$ into $u_k$ and $v_k$, we have
\begin{align*}
x &= \ldots v_k w_1 w_2 \ldots \\
\sigma^{\len{v_{k}}}x &= \ldots w_1 w_2 \ldots
\end{align*}
where each $w_{j} \in \{u_k, v_k\}$.

By definition, $x$ does not contain three consecutive $u_k$, i.e. there does not exist $j$ so that $w_j = w_{j+1} = w_{j+2} = u_k$. 
We can then decompose $x$ into blocks of the form $v_k^i u_k^j$, $i > 0$, $j \in \{1,2\}$, and then each such block corresponds to a block $v_k^{i-1} u_k^j v_k$ (of the same length) within $\sigma^{\len{v_k}} x$. Therefore, these blocks occur at the same locations in
$x$ and $\sigma^{\len{v_k}} x$, and the set of locations at which $x$ and $\sigma^{\len{v_k}} x$ differ is the union of 
such locations in these pairs of blocks. This number of differences in such a pair is $d(v_k^i u_k^j, v_k^{i-1} u_k^j v_k) = d(v_k u_k^j, u_k^j v_k)$, which is bounded from above by 
$4\len{u_{0}} 2^{\sum_{j=0}^{k-1}\bbone_{r_{j}}}(n_{0} - m_{0})(n_{1}-m_{1})\cdots (n_{k-1}-m_{k-1})$ by Lemma~\ref{dvu}. Since each of $u_{k+1}$ and $v_{k+1}$ contains at most two occurrences of $u_k$, the density of locations
where a $u_k$ starts in $x$ is bounded from above by $\frac{2}{|v_{k+1}|}$. Putting all of this together, 
\[
\overline{d}(\{ t : x_{t} \ne (\sigma^{\len{v_{k}}}x)_{t} \}) \leq
\frac{8\len{u_{0}} 2^{\sum_{j=0}^{k-1}\bbone_{r_{j}}}(n_{0} - m_{0})(n_{1}-m_{1})\cdots (n_{k-1}-m_{k-1})}{\len{v_{k+1}}}
\]
so by Proposition \ref{Pmain},
\[
\overline{d}(\{ t : x_{t} \ne (\sigma^{\len{v_{k}}}x)_{t} \}) < \frac{8\len{u_{0}} \len{v_{k}}}{\len{v_{k+1}}}\epsilon_{k}
\]
where $\sum_{k=0}^{\infty} \epsilon_{k} < \infty$.
Let
\[
S_{k} = \Big{\{} \sum_{i=k}^{\ell} p_{i}\len{v_{i}} : \ell > k, 0 \leq p_{i} < \frac{\len{v_{i+1}}}{\len{v_{i}}} \Big{\}}
\]
and observe that for all $t \geq 1$, there exists $0\leq m \leq \len{v_{k}}$ such that $t-m \in S_{k}$ so $S_{k}$ is syndetic.
Write $T_{s} = \{ t : x_{t} \ne (\sigma^{s}x)_{t} \}$.  For $s \in S_{k}$,
by the subadditivity of $\overline{d}$,
\[
\overline{d}(T_{s}) \leq \sum_{i=k}^{\ell} \overline{d}(T_{p_{i}\len{v_{i}}}) \leq \sum_{i=k}^{\ell} p_{i} 
\overline{d}(T_{\len{v_{i}}})
< \sum_{i=k}^{\ell} \frac{\len{v_{i+1}}}{\len{v_{i}}} \frac{8\len{u_{0}}\len{v_{i}}}{\len{v_{i+1}}}\epsilon_{i} \leq 8\len{u_{0}}\sum_{i=k}^{\infty} \epsilon_{i}.
\]
Since $\sum \epsilon_{k} < \infty$, then $\lim_{k} \sup_{s\in S_{k}} \overline{d}(T_{s}) = 0$ so $X$ is mean almost periodic, and therefore has discrete spectrum.
\end{proof}

\subsection{The balanced property}

The bounds on $D_k$ from Proposition~\ref{Pmain} can also be used to prove that $X$ is balanced.

\begin{theorem}\label{balanced}
Any infinite minimal subshift $X$ with $\limsup p(q)/q < 1.5$ is balanced on words.
\end{theorem}
\begin{proof}
We again apply our S-adic decomposition from Corollary~\ref{taus}, and will assume without loss of generality that 
$\pi$ is the identity (i.e. has no effect) for this proof, since the image of a subshift balanced for words under a single substitution is still balanced for words (see Theorem 3.4 of \cite{PS24}).

Theorem 5.8 from \cite{MR3330561} gives a way to view balancedness for letters in terms of rate of convergence along a so-called generalized eigenvector for products of incidence matrices of the substitutions. However, its hypotheses as stated are too restrictive, and so we instead adapt the technique of their proof, which relies on the so-called Dumont-Thomas prefix-suffix decomposition (\cite{DT89}). Specifically, for any $X$ generated by a sequence $(\tau_k)$ of substitutions (here $\tau_k = \tau_{m_k, n_k, r_k}$) and any word $w \in L(X)$, there exist $n$ and words $p_j, s_j$, $0 \leq j \leq n$, where each $p_j$ is a suffix of some $\tau_j(a)$ and each $s_j$ is a prefix of some $\tau_j(a)$, so that
\[
w = p_0 \tau_0(p_1 \tau_1( \cdots p_{n-1} \tau_{n-1}(p_n s_n) s_{n-1}) \cdots s_1) s_0.
\]

Let $d_{k} = \len{v_{k+1}}$ so that $d_{k} = b_{k}d_{k-1} + a_{k}d_{k-2}$ for all $k \geq 0$ (setting $d_{-2} = 0$).  Let $g_{k} = \len{v_{k+1}}_{1}$ %and $g_{-2} = \len{u_{0}}_{1} - \len{v_{0}}_{1}$ 
so that $g_{k} = b_{k}g_{k-1} + a_{k}g_{k-2}$ (letting $g_{-2} = 1$).  By Lemma~\ref{contlem}, $\frac{g_{k}}{d_{k}}$ converges, say to a limit $\alpha^{\star}$.

Therefore, the frequency of $1$s in $v_k$ approaches $\alpha^{\star}$. The next step is to verify that $X$ is balanced for letters, which will complete the proof since each incidence matrix $M_k$ is invertible and the associated substitutions $\tau_k$ are each right-proper, and so Corollary 4.3 from \cite{PS24} will imply that $X$ is balanced on words. Since each $\tau_k$ has alphabet $\{0,1\}$, we need only show that the discrepancy $\big||w| \alpha^\star - |w|_1\big|$ is uniformly bounded over $w$. In the proof of Theorem 5.8 in \cite{MR3330561}, they give the following formula:
\begin{equation}\label{DT}
|w|_1 - |w|\alpha^\star
=
\sum_{k=0}^{n}
\left(
|\tau_{[0,k)}(p_k)|_1
- |\tau_{[0,k)}(p_k)|\alpha^\star
+ |\tau_{[0,k)}(s_k)|_1
- |\tau_{[0,k)}(s_k)|\alpha^\star
\right),
\end{equation}
where $\tau_{[0,k)}$ is the composition $\tau_0 \circ \cdots \tau_{k-1}$. We first derive discrepancy bounds for 
$\tau_{[0,k)}(0) = v_k$, which essentially come from Proposition~\ref{Pmain}. Specifically, 
\[
\big| |v_k| \alpha^* - |v_k|_1 \big| = |g_{k-1} - d_{k-1}\alpha^{\star}| = d_{k-1} \left| \frac{g_{k-1}}{d_{k-1}} - \alpha^{\star} \right| 
< d_{k-1} \left|\frac{g_{k-1}}{d_{k-1}} - \frac{g_{k}}{d_{k}}\right|
= \frac{|g_{k-1}d_{k} - g_{k}d_{k-1}|}{d_{k}}
\]

It is easily checked (by induction or by using incidence matrices $M_k$) that $g_{k-1}d_{k} - g_{k}d_{k-1} = \newline (-1)^{k}a_{0}a_{1} \cdots a_{k}(g_{-2}d_{-1} - g_{-1}d_{-2})$. Therefore, if we set $D = |g_{-2}d_{-1} - g_{-1}d_{-2}|$,
\[
\big| |v_k| \alpha^* - |v_k|_1 \big|= \left| g_{k-1} - d_{k-1}\alpha^{\star} \right| < \frac{Da_{0}a_{1}\cdots a_{k}}{d_{k}} 
= D \frac{d_{k-1}}{d_{k}} \frac{2^{\sum_{j=0}^{k}\bbone_{r_{j}}}\prod_{j=0}^{k-1} (n_{j} - m_{j})}{\len{v_{k}}}
< D \frac{d_{k-1}}{d_{k}} \epsilon_{k}
\]
where $\epsilon_{k}$ is as in Proposition \ref{Pmain}.

Using this, we can get a discrepancy bound for $\tau_{[0,k)}(1) = u_k$. Specifically, since $u_k = v_{k-1}^{n_{k-1} - m_{k-1}} v_k$ for $k > 0$, we get
\begin{multline*}
\big| |u_k| \alpha^* - |u_k|_1 \big| \leq 
\big| |v_k| \alpha^* - |v_k|_1 \big| + (n_{k-1} - m_{k-1}) \big| |v_{k-1}| \alpha^* - |v_{k-1}|_1 \big| \\
< D \frac{d_{k-1}}{d_{k}} \epsilon_{k} + D \frac{(n_{k-1} - m_{k-1})d_{k-2}}{d_{k-1}} \epsilon_{k-1}
< D\left(\frac{d_{k-1}}{d_{k}} \epsilon_{k} + 3 \epsilon_{k-1}\right).
\end{multline*}

Finally, we return to (\ref{DT}), noting that each $p_k$ and $s_k$ contains at most two $1$s and at most $4m_k$ $0$s. Therefore,
\begin{multline*}
\big| |w|_1 - |w|\alpha^\star \big|
\leq
\sum_{k=0}^{n}
\left(
\big| |\tau_{[0,k)}(p_k)|_1
- |\tau_{[0,k)}(p_k)|\alpha^\star \big|
+ \big| |\tau_{[0,k)}(s_k)|_1
- |\tau_{[0,k)}(s_k)|\alpha^\star \big|
\right) \\
< |p_0| + |s_0| + \sum_{k=1}^{n} 4D\left(\frac{d_{k-1}}{d_{k}} \epsilon_{k} + 3 \epsilon_{k-1}\right) + 8m_k D \frac{d_{k-1}}{d_{k}} \epsilon_{k}
< |u_1| + |v_1| + 4D \sum_{k=1}^{n} \left(\frac{(2m_k + 1) d_{k-1}}{d_{k}} \epsilon_{k} + 3 \epsilon_{k-1}\right)\\ 
< |u_1| + |v_1| + 12D \sum_{k=1}^{n} \epsilon_k + \epsilon_{k-1}.
\end{multline*}

Since this series is summable and $w$ was arbitrary, $X$ is balanced on letters. Finally, as noted earlier, Corollary 4.3 from \cite{PS24} implies that $X$ is balanced on words, completing the proof. 
\end{proof}

\begin{remark}
By analyzing the bounds from Proposition~\ref{Pmain}, we can turn this into an explicit bound on letter discrepancy for $X$.
The proof of Proposition~\ref{Pmain} shows that $\epsilon_k < C(1 - \delta/2)^{k/4-1}$, where $C$ is a constant that can always be taken less than $2^6$, and $\delta > 0$ comes from nearness of $\limsup p(q)/q$ to $1.5$. 

Therefore, letter discrepancies of $X$ are bounded by a constant plus $12D \sum_{k=1}^{\infty} \epsilon_k + \epsilon_{k-1} 
< 24D \sum_{k=0}^{\infty} 2^6 (1 - \delta/2)^{k/4-1} < 20000/\delta$. We recall that this was under the assumption that $\pi = \textrm{id}$, and so in the general case an additional factor appears dependent on $\pi$.
\end{remark}

Balancedness on words already has several very useful implications. A main one is that Birkhoff sums coming from (normalized) characteristic functions of cylinder sets are bounded in absolute value, which implies by the classical Gottschalk-Hedlund theorem (\cite{GH}) that 
such functions are coboundaries. Then, as noted in Remark 3.9 of \cite{BCE}, in fact we know that the additive eigenvalue group 
$E_X$ is the same as the so-called image group $(I_X, +)$ generated by measures of cylinder sets. This immediately allows for analysis of orbit equivalence and dimension groups.

\subsection{Dimension groups and orbit equivalence}\label{orb}

\textbf{Orbit equivalence} and \textbf{strong orbit equivalence} are two weakened versions of isomorphism which are well-studied in dynamical systems. It was proved by Giordano, Putnam, and Skau in \cite{MR1363826} that for minimal TDS on a Cantor set, the so-called \textbf{dimension group} (a unital ordered group $K^{0}(X,\sigma)$) is a complete invariant for strong orbit equivalence, and the \textbf{reduced dimension group} (a unital ordered group $\widehat{K^{0}}(X,\sigma)$) is a complete invariant for orbit equivalence. As we do not make use of any nuanced properties of the dimension groups, we omit definitions and refer the reader to e.g. \cite{MR4228544} for definitions and details. 

We can now see (mostly using results from \cite{MR4228544}) that the dimension groups of any low-complexity infinite minimal subshift are equal to the additive eigenvalue group, and thereby characterize orbit equivalence and strong orbit equivalence.

\begin{theorem}\label{dimgroup}
For any infinite minimal subshift $X$ with $\limsup p(q)/q < 1.5$, the dimension group and reduced dimension group are both equal to $(E_{X}, E_{X} \cap \mathbb{R}^{+}, 1)$.
\end{theorem}
\begin{proof}
By section 2.4 of \cite{MR4228544}, since $X$ is minimal and uniquely ergodic, the dimension group $K^{0}(X,\sigma)$ and its group of infinitesimals $\mathrm{Inf}(K^{0}(X,\sigma))$ have the property that the reduced dimension group $\widehat{K^{0}}(X,\sigma) = K^{0}(X,\sigma) / \mathrm{Inf}(K^{0}(X,\sigma))$ is isomorphic to the image group $(I(X,\sigma),I(X,\sigma) \cap \mathbb{R}^{+},1)$.  

As noted above, $I(X,\sigma) = E_{X}$. Finally, $X$ is balanced on words by Theorem \ref{balanced}, so Proposition 5.4 of \cite{MR4228544} implies there are no infinitesimals.  Then $\widehat{K^{0}}(X,\sigma) = K^{0}(X,\sigma) = (E_{X}, E_{X} \cap \mathbb{R}^{+}, 1)$.
\end{proof}

The following corollary is now immediate, modulo the simple observation that if $G, G'$ are additive subgroups of $\mathbb{R}$ containing $1$, then $(G, G \cap \mathbb{R}^{+}, 1)$ and $(G', G' \cap \mathbb{R}^{+}, 1)$ are isomorphic as unital ordered groups iff $G = G'$.

\begin{corollary}\label{SOE}
Two infinite minimal subshifts with complexity satisfying $\limsup p(q)/q < 1.5$ are orbit equivalent if and only if they are strong orbit equivalent if and only if they have the same additive eigenvalue group.
\end{corollary}

There is one more immediate observation we can make; the main theorem of \cite{tijdeman} relates word complexity to dimension of 
$I(X)$ as a vector space over $\mathbb{Q}$, and in particular implies that even $\liminf p(q)/q < 2$ implies dimension bounded above by $2$. Since balancedness implied that $E_X = I_X$ and since the MEF is the character group of $E_X$, we can already conclude that the maximal equicontinuous factor of such subshifts is zero- or one-dimensional. However, for more nuanced analysis, we will need more hands-on techniques, which are the focus of the next section.

\section{The additive eigenvalue group}\label{eigs}

In this section, we explicitly compute the additive continuous eigenvalue group for low complexity minimal subshifts in terms of the $a_{k}$ and $b_{k}$ defined in Section \ref{discspecsec}, which is the first step in characterizing the maximal equicontinuous factor. As described in Remark 3.2 of \cite{creutzpavlov}, the unique irrational eigenvalue $\alpha$ (up to rational multiples) can be described via a generalized continued fraction; see Theorem~\ref{evalgroup} and Corollary~\ref{circlefactor} below. This in fact can be thought of as an entry of the generalized Perron eigenvector in the sense of \cite{MR3330561}, corresponding to the measure of a cylinder set, which is automatically an eigenvalue by balancedness and Gottschalk-Hedlund as mentioned earlier.

The main tools are the exponential decay already established and approximation arguments along a similar line of reasoning as in \cite{creutzpavlov}, though more complex.

Throughout this section, let $X$ be an infinite minimal subshift with $\limsup p(q)/q < 1.5$, which therefore satisfies the conclusions of Propositions \ref{words}, \ref{svp}, \ref{rk} and \ref{Pmain}.

Let $u_{k}$ and $v_{k}$ be the words from Proposition \ref{words} and $(a_{k})$ and $(b_{k})$ as in (\ref{ab}). 
Any reference to measure refers to the unique $\sigma$-invariant measure $\mu$.  By minimality, the measure of any nonempty open set is positive.

We first introduce the following notation for subgroups of $(\mathbb{Q}, +)$.  Here $\mathbb{P}$ is the set of primes.

\begin{definition}
Let $0 \leq \ell_{p} \leq \infty$ for each prime $p \in \mathbb{P}$.  The $(\ell_{p})$-subgroup of $\mathbb{Q}$ is
\[
Q_{(\ell_{p})} = \{ q \in \mathbb{Q} : \exists p_{1},\ldots,p_{t} \in \mathbb{P}~\text{such that}~p_{1}\cdots p_{t}q \in \mathbb{Z}~\text{and}~|\{ 1 \leq i \leq t : p_{i} = p \}| \leq \ell_{p}~\text{for all $p \in \mathbb{P}$} \}
\]
\end{definition}

That $Q_{(\ell_{p})}$ is a group under addition is easily verified.

The purpose of this section is to prove Theorem~\ref{evalgroup}, which provides a description of the eigenvalue group in terms of subgroups $L_X, R_X$ of $\mathbb{Q}$ describable in terms of the substitutions $\tau_{m_k,n_k,r_k}$.  We need to introduce the following standard notation.

\begin{notation}
For a prime $p$ and $a \in \mathbb{Q}_{p}$ a $p$-adic number, the \textbf{$p$-adic fractional part} is
\[
\{ a \}_{p} = \sum_{t = -m}^{-1} a_{t}p^{t}
\]
where $a = \sum_{t=-m}^{\infty} a_{t}p^{t}$ is the $p$-adic expansion of $a_{p}$.
\end{notation}

Note that $\{ a + a^{\prime} \}_{p} = \{ a \}_{p} + \{ a^{\prime} \}_{p}\pmod{\mathbb{Z}}$ and that for $q \in \mathbb{Q}$, $q = \sum_{p} \{ q \}_{p}\pmod{\mathbb{Z}}$.

\begin{theorem}\label{evalgroup}
For each prime $p \in \mathbb{P}$, let, writing $v_{k}$ for the words from Section \ref{subs},
\begin{align*}
L_{X}(p) &= \sup \{ t \geq 0 : p^{t}~\text{divides}~\frac{\len{v_{0}}a_{0}\cdots a_{k}}{\gcd(\len{v_{k}},\len{v_{k+1}})}~\text{for some $k \geq 0$} \} \\
R_{X}(p) &= \sup \{ t \geq 0 : p^{t}~\text{divides}~\gcd(\len{v_{k}},\len{v_{k+1}})~\text{for some $k \geq 0$} \}
\end{align*}
and let $Q_{X}$ be the $(L_{X}(p))$-subgroup of $\mathbb{Q}$ and $R_{X}$ be the $(R_{X}(p))$-subgroup of $\mathbb{Q}$.
Let
\[
\alpha = \frac{\lambda}{\len{u_{0}}\lambda + \len{v_{0}}(1 - \lambda)}
\quad\text{where}\quad
\lambda = \cfrac{a_0}{b_0+\frac{a_1}{b_1+\cdots}}.
\]
Then there exist $e_{p} \in \mathbb{Q}_{p}$ for each prime $p$
such that $e_{p} = 0$ for primes $p$ such that $R_{X}(p) = \infty$ and $e_{p} \in p^{-R_{X}(p)}\mathbb{Z}_{p}$ for all other primes and
\[
E_{X} = \left\{ q\alpha + \sum_{p} \{ \{ q \}_{p} e_{p} \}_{p} + r : q \in Q_{X}, r \in R_{X} \right\}.
\]
In addition, all measurable eigenfunctions are continuous. 
\end{theorem}

The $e_{p}$ represent the fact that in a suitable sense the eigenvalues include $p$-adic parts. More precisely, the adele $(\alpha,(e_{p}))$ (see Section \ref{MEF}) plays the role of the fundamental eigenvalue just as $\alpha$ is the fundamental eigenvalue when looking at $S^{1}$ under rotation by $\alpha$. 

For Example 1.4, $e_{2}$ is the only relevant $p$-adic rotation and both $\alpha$ and $-e_{2}$ are $(-3 + \sqrt{17})/4$ (it can be checked quickly that $e_{2}$ is in fact a $2$-adic integer).

While beyond the scope of this paper, the $e_{p}$ can be expressed as $p$-adic continued fractions and in fact are the same continued fraction as $\alpha$ but treated as converging instead in the $p$-adic norm. In the cases when $0 < \ell_{p} < \infty$, $e_{p} \in p^{-\ell_{p}}\mathbb{Z}$ and is treated as an element of $\mathbb{Z} / p^{\ell_{p}}\mathbb{Z}$.

For Example 1.4, $Q_{X} = \mathbb{Z}[1/2]$ and $R_{X} = \mathbb{Z}$. The only nontrivial $e_{p}$ is $e_{2}$, which can be expressed as the $2$-adic continued fraction with the same formal expansion as $\alpha$. The eigenvalue group is then $\{ (m/2^{k})\alpha + \{ (m/2^{k})e_{2} \}_{2} : m,k \in \mathbb{Z} \}$.

Before proceeding with the proof, we establish that the $L_{X}(p)$ are integers.

\begin{lemma}\label{div}
For all $k \geq 0$, $\gcd(\len{v_{k}},\len{v_{k+1}})$ divides $\len{v_{0}}a_{0}\cdots a_{k}$.
\end{lemma}
\begin{proof}
Set $g_{0} = \gcd(\len{v_{0}},\len{u_{0}}-\len{v_{0}})$ and $g_{k} = \gcd(\len{v_{k}},\len{v_{k-1}})$ for $k \geq 1$.  Then $g_{0}$ divides $\len{v_{0}}$ and $g_{k+1} = g_{k}\gcd(b_{k}\frac{\len{v_{k}}}{g_{k}} + a_{k}\frac{\len{v_{k-1}}}{g_{k}},\frac{\len{v_{k}}}{g_{k}}) = g_{k} \gcd(a_{k}\frac{\len{v_{k-1}}}{g_{k}},\frac{\len{v_{k}}}{g_{k}})$ and since $\gcd(\frac{\len{v_{k}}}{g_{k}},\frac{\len{v_{k-1}}}{g_{k}}) = 1$, then $g_{k+1}$ divides $g_{k}a_{k}$ so by induction $g_{k+1}$ divides $\len{v_{0}}a_{0}\cdots a_{k}$ for all $k$.
\end{proof}

\subsection{Additive continuous eigenvalues}

Our first step is establishing the existence of a family of irrational additive continuous eigenvalues, all of which are explicitly defined in terms of generalized continued fractions using $a_{k}$ and $b_{k}$. This also follows from 4.2 \cite{bertheetal} and balancedness, we include the proof for completeness.

\begin{proposition}\label{irrat}
$\lambda$ and $\alpha$ as defined in Theorem~\ref{evalgroup} exist are irrational.
\end{proposition}
\begin{proof}
Existence of both follows from Lemma \ref{contlem}.
Suppose that $\lambda \in \mathbb{Q}$ so $0 < \lambda = \frac{p}{q}$ for some $p,q \in \mathbb{Z}$ with $p,q > 0$.  Define the sequence $(p_{k})$ by $p_{-2} = q$ and $p_{-1} = p$ and for $k \geq -1$, $p_{k+1} = -b_{k+1}p_{k} + a_{k+1}p_{k-1}$.  By construction, $\lambda = \lambda_{0} = \frac{p_{-1}}{p_{-2}}$.  Assume that $\lambda_{k+1} = \frac{p_{k}}{p_{k-1}}$.  Then, since $\lambda_{k+1} = \frac{a_{k+1}}{b_{k+1} + \frac{a_{k+2}}{b_{k+2} + \cdots}} =
\frac{a_{k+1}}{b_{k+1} + \lambda_{k+2}}$,
\[
\frac{p_{k+1}}{p_{k}} = -b_{k+1} + a_{k+1}\frac{p_{k-1}}{p_{k}} = -b_{k+1} + \frac{a_{k+1}}{\lambda_{k+1}} = \lambda_{k+2}
\]
so by induction, $\lambda_{k} = \frac{p_{k}}{p_{k-1}}$ for all $k$.  In particular, $p_{k} > 0$ for all $k$ since $\lambda_{k} > 0$ and $p_{-1} > 0$ so if there were a minimal $k$ such that $p_{k} < 0$ then that $\lambda_{k} < 0$).

Now observe that
\begin{align*}
p_{k+1}\len{v_{k+1}} + p_{k}\len{v_{k+2}} &= -b_{k+1}p_{k}\len{v_{k+1}} + a_{k+1}p_{k-1}\len{v_{k+1}} + p_{k}b_{k+1}\len{v_{k+1}} + p_{k}a_{k+1}\len{v_{k}} \\
&= a_{k+1}(p_{k}\len{v_{k}} + p_{k-1}\len{v_{k+1}})
\end{align*}
and since $p_{-1}(\len{u_{0}} - \len{v_{0}}) + p_{-2}\len{v_{0}} = p\len{u_{0}} + (q-p)\len{v_{0}}$, by induction then
\[
p_{k+1}\len{v_{k+1}} + p_{k}\len{v_{k+2}} = (p\len{u_{0}} + (q-p)\len{v_{0}})a_{0}a_{1}\cdots a_{k+1}
< (p\len{u_{0}} + (q-p)\len{v_{0}})\epsilon_{k+1}\len{v_{k+1}}
\]
where $\epsilon_{k+1}$ is as in Proposition \ref{Pmain}.  Since $p_{k}, p_{k+1} \geq 1$,
\[
\len{v_{k+1}} < p_{k+1}\len{v_{k+1}} + p_{k}\len{v_{k+2}} < (p\len{u_{0}} + (q-p)\len{v_{0}})\epsilon_{k+1}\len{v_{k+1}}
\]
but then
\[
\frac{1}{p\len{u_{0}} + (q-p)\len{v_{0}}} < \epsilon_{k+1} \to 0
\]
which is impossible. Therefore $\lambda \notin \mathbb{Q}$ hence $\alpha \notin \mathbb{Q}$ (as $\len{u_{0}} > \len{v_{0}}$).
\end{proof}

\begin{definition}
For all $k \geq 0$, define
\[
\lambda_{k} = \frac{a_{k}}{b_{k} + \frac{a_{k+1}}{b_{k+1} + \cdots}} \textrm{ and } 
\alpha_{k} = \frac{1}{\len{v_{k}} + \len{v_{k-1}}\lambda_{k}}
\]
(where $\len{v_{-1}}$ is defined as $\len{u_{0}} - \len{v_{0}}$, obtained by reversing the recursion from Lemma~\ref{vrec}). 

The \textbf{eigenvalue family} is the set $\{\alpha_k\}_{k \geq 0}$.
\end{definition}

We adapt the argument of Host \cite{MR873430} using `approximate eigenfunctions' and deduce convergence to an actual eigenfunction from the exponential decay of Proposition~\ref{Pmain}. The notion of approximate eigenfunction roughly amounts to defining functions which behave like eigenfunctions on Rokhlin towers of measure arbitrarily close to full, and the deduction of an actual eigenfunction follows from the approximation.

\begin{lemma}\label{TT}
Define Rokhlin towers by, setting $u_{k}^{\prime}$ such that $u_{k} = u_{k}^{\prime}v_{k}$,
\begin{align*}
B_k &= \{ x \in X : x~\text{as a $(k+1)$-concatenation has $v_{k+1}$ at the origin, possibly as a suffix of $u_{k+1}$} \}, \\
B_k^{\prime} &= \{ x \in X : x~\text{as a $(k+1)$-concatenation has $u_{k+1}^{\prime}$ at the origin, as a prefix of $u_{k+1}$} \},
\end{align*}
and $T_k = \bigsqcup_{j = 0}^{\len{v_{k+1}} - 1} \sigma^j B_k$ and $T_k^{\prime} =  \bigsqcup_{j = 0}^{\len{u_{k+1}} - \len{v_{k+1}} - 1} \sigma^j B_k^{\prime}$.

Then for all $k$, $T_{k} \sqcup T_{k}^{\prime} = X$ and $\mu(T_{k}) \geq \frac{1}{4}$.
\end{lemma}
\begin{proof}
Every $x \in X$ is uniquely decomposable as a concatenation of $u_{k+1}$ and $v_{k+1}$ hence of $u_{k+1}^{\prime}$ and $v_{k+1}$ so the levels of the towers are disjoint and union to the entire space.  Since $n_{k+1} \leq 2m_{k+1} + 2$, $\len{u_{k+1}} - \len{v_{k+1}} = (n_{k}-m_{k})\len{v_{k}} \leq (m_{k} + 2)\len{v_{k}}$ and $\len{v_{k+1}} = (m_{k}-1)\len{v_{k}} + \len{u_{k}} > m_{k}\len{v_{k}}$, then $\frac{\len{u_{k+1}} - \len{v_{k+1}}}{\len{v_{k+1}}} < \frac{m_{k}+2}{m_{k}} \leq 3$.  Therefore $\mu(T_k) \geq \frac{1}{4}\mu(T_{k} \sqcup T_{k}^{\prime})$.
\end{proof}

\begin{proposition}\label{2piqalpha}
For $k \geq 0$, $\alpha_{k}$ is an additive continuous eigenvalue.
\end{proposition}
\begin{proof}
Fix $k_{0} \geq 0$ and let $k \geq k_{0}$.
Every $x \in X$ can be written in a unique way as a concatenation of $u_{k}$ and $v_{k}$; we will refer to such as a $k$-concatenation.  Let $B_{k}$ be the set of $x \in X$ such that $x$ has $u_{k}$ or $v_{k}$ at the origin when written as a $k$-concatenation.  Let $j(x,k)$ be the minimal $j \geq 0$ such that $\sigma^{-j}x \in B_{k}$.

Consider $x$ which has $u_{k}$ at the origin.  If $r_{k} = 0$ then one of $\sigma^{-(m_{k}-1)\len{v_{k}}}x$ or $\sigma^{-(n_{k}-1)\len{v_{k}}}x$ is in $B_{k+1}$.  If $r_{k} > 0$ then one of $\sigma^{-(m_{k}+r_{k}-2)\len{v_{k}} - \len{u_{k}}}x$ or $\sigma^{-(n_{k}+r_{k}-2)\len{v_{k}} - \len{u_{k}}}x$ is in $B_{k+1}$.
For $x$ that has $v_{k}$ at the origin, there exists $1 \leq p < n_{k}+r_{k}$ such that $\sigma^{-p\len{v_{k}}}x$ or $\sigma^{-p\len{v_{k}} - \len{u_{k}}}x$ is in $B_{k+1}$.  Therefore for every $x \in X$, we have
$j(x,k+1) - j(x,k) = p\len{v_{k}}$ or $j(x,k+1) - j(x,k) = p\len{v_{k}} + \len{u_{k}}$
for some $1 \leq p < n_{k} + r_{k}$.

Since $\len{u_{k}} = \len{v_{k}} + (n_{k-1}-m_{k-1})\len{v_{k-1}}$, in the latter case, $j(x,k+1) - j(x,k) = (p+1)\len{v_{k}} + (n_{k-1}-m_{k-1})\len{v_{k-1}}$.  Therefore
\[
j(x,k+1) - j(x,k) = p\len{v_{k}} + p^{\prime}\len{v_{k-1}} \tag{$\ddagger$}
\]
 for some $1 \leq p \leq n_{k} + r_{k}$ and $p^{\prime} = 0$ or $p^{\prime} = n_{k-1} - m_{k-1} < n_{k-1}$.

Let $f_{k}(x) = \exp(2\pi i \alpha_{k_{0}} j(x,k))$.
Each $f_{k}$ is `approximately' an eigenfunction: $f_{k}(\sigma x) = \exp(2\pi i \alpha_{k_{0}})f_{k}(x)$ except when $\sigma x \in B_{k}$ and $\mu(B_{k}) \to 0$ (since $\sigma^{i}B_{k}$ are disjoint for at least $0 \leq i < \len{v_{k}}$ and $\len{v_{k}} \to \infty$).  

Define the notation $\langle \cdot \rangle$ to be the distance to the nearest integer.
Observe that
\begin{align*}
|f_{k}(x) - f_{k+1}(x)| &= |\exp(2\pi i \alpha_{k_{0}} j(x,k)) - \exp(2\pi i \alpha_{k_{0}} j(x,k+1))| \\
&= |\exp(2\pi i \alpha_{k_{0}} j(x,k))(1 - \exp(2\pi i \alpha_{k_{0}} (p \len{v_{k}} + p^{\prime}\len{v_{k-1}}))| \\
&= |1 - \exp(2\pi i \alpha_{k_{0}} (p \len{v_{k}} + p^{\prime}\len{v_{k-1}})| \\
&\leq |1 - \exp(2\pi i \alpha_{k_{0}} p \len{v_{k}})| + |\exp(2\pi i \alpha_{k_{0}} p \len{v_{k}}) - \exp(2\pi i \alpha_{k_{0}} (p \len{v_{k}} + p^{\prime}\len{v_{k-1}})| \\
&= |1 - \exp(2\pi i \alpha_{k_{0}} p \len{v_{k}})| + |1 - \exp(2\pi i \alpha_{k_{0}} p^{\prime}\len{v_{k-1}})| \\
&\leq 2\pi \langle \alpha_{k_{0}} p \len{v_{k}}\rangle + 2\pi \langle \alpha_{k_{0}} p^{\prime} \len{v_{k-1}}\rangle.
\end{align*}

Suppose we knew that there exist $\epsilon_{k}^{\prime} > 0$ with $\sum \epsilon_{k}^{\prime} < \infty$ such that $\max_{1 \leq p \leq n_{k}+r_{k}} \langle \alpha_{k_{0}}p\len{v_{k}}\rangle < \epsilon_{k}^{\prime}$.
Since $\langle \alpha_{k_{0}} p^{\prime} \len{v_{k-1}}\rangle = 0$ when $p^{\prime} = 0$ and $\langle \alpha_{k_{0}}p^{\prime}\len{v_{k-1}} \rangle < \epsilon_{k-1}^{\prime}$ when $p^{\prime} > 0$, then we would have $\sum_{k=K}^{\infty} |f_{k+1}(x) - f_{k}(x)| < \sum_{k=K}^{\infty} (\epsilon_{k}^{\prime} + \epsilon_{k-1}^{\prime})$ which tends to zero uniformly over $x \in X$.  So then the $f_{k}(x)$ are uniformly Cauchy in the sup norm, and as each $f_{k}(x)$ is continuous, they converge to a continuous limit $f(x)$.  Since $f_{k}(\sigma x) = \exp(2\pi i \alpha_{k_{0}})f_{k}(x)$ on sets approaching full measure (and the unique invariant measure necessarily has full support), by continuity $f(\sigma x) = \exp(2\pi i \alpha_{k_{0}})f(x)$ for all $x$.  We will now show that such $\epsilon_{k}^{\prime}$ exist.

Set $d_{k} = \len{v_{k+k_{0}+1}}$ for $k \geq -2$ (if $k_{0}=0$ then set $d_{-2} = \len{u_{0}} - \len{v_{0}}$ as $v_{-1}$ is undefined).  Then $d_{k+1} = b_{k+k_{0}+1}d_{k} + a_{k+k_{0}+1}d_{k-1}$.  Define sequences $(c_{k})$ and $(e_{k})$ by $c_{-2} = 1$, $c_{-1}=0$, $e_{-2}=0$, $e_{-1}=1$ and the same recursion relation $c_{k+1} = b_{k+k_{0}+1}c_{k} + a_{k+k_{0}+1}c_{k-1}$ and $e_{k+1} = b_{k+k_{0}+1}e_{k} + a_{k+k_{0}+1}e_{k-1}$. By Lemma~\ref{contlem}, $\frac{c_{k}}{e_{k}} \to \lambda_{k_{0}}$.  Since the sequences are all defined by the same linear recurrence relation, $d_{k} = d_{-2}c_{k} + d_{-1}e_{k}$ for all $k$.  Then
\[
\lim \frac{e_{k}}{d_{k}} = \lim \left(d_{-1} + d_{-2}\frac{c_{k}}{e_{k}}\right)^{-1} = (d_{-1} + d_{-2} \lambda_{k_{0}})^{-1} = \frac{1}{d_{-1} + d_{-2}\lambda_{k_{0}}} = \alpha_{k_{0}}
\]

It is easily verified by induction that $e_{k+1}d_{k} - e_{k}d_{k+1} = (-1)^{k}\len{v_{k_{0}-1}}a_{k_{0}}a_{k_{0}+1} \cdots a_{k_{0}+k+1}$ for all $k$.  

Since $e_{k+1}d_{k} - e_{k}d_{k+1}$ alternates sign, $\frac{e_{2k}}{d_{2k}}$ approaches $\alpha_{k_{0}}$ from below and $\frac{e_{2k+1}}{d_{2k+1}}$ approaches $\alpha_{k_{0}}$ from above.  Therefore 
\[
\left|\alpha_{k_{0}} - \frac{e_{k}}{d_{k}}\right| < \left|\frac{e_{k+1}}{d_{k+1}} - \frac{e_{k}}{d_{k}}\right|
= \frac{\len{v_{k_{0}-1}}a_{k_{0}} \cdots a_{k_{0}+k+1}}{d_{k}d_{k+1}}
\]

Then $\left|\alpha_{k_{0}}d_{k} - e_{k}\right| < \frac{\len{v_{k_{0}-1}}a_{k_{0}} \cdots a_{k_{0}+k+1}}{d_{k+1}}$ so for any $p$,
\[
\langle p\alpha_{k_{0}}d_{k}\rangle < \frac{p\len{v_{k_{0}-1}}a_{k_{0}}\cdots a_{k+k_{0}+1}}{d_{k+1}}
\]
By Proposition \ref{Pmain}, there exists $\epsilon_{k}$ with $\sum \epsilon_{k} < \infty$ such that $\frac{a_{0}\cdots a_{k+k_{0}+1}}{d_{k}} < \epsilon_{k+k_{0}}$.

For $p \leq n_{k+1} + r_{k+1} \leq 2m_{k+1} + 2 + r_{k+1} \leq 2b_{k+1} + 2$, we have $\frac{pd_{k}}{d_{k+1}} \leq \frac{(2b_{k+1}+2)d_{k}}{b_{k+1}d_{k}} \leq 4$.  Then 
\[
\langle p\alpha_{k_{0}}d_{k}\rangle < \frac{pd_{k}\len{v_{k_{0}-1}}\epsilon_{k+k_{0}}}{d_{k+1} a_{0} \cdots a_{k_{0}-1}}
\leq 4 \frac{\len{v_{k_{0}-1}}}{a_{0}\cdots a_{k_{0}-1}}\epsilon_{k+k_{0}}.
\]
Setting $\epsilon_{k}^{\prime} = 4\len{v_{k_{0}-1}}(a_{0}\cdots a_{k_{0}-1})^{-1}\epsilon_{k+k_{0}}$ completes the proof.
\end{proof}

\begin{corollary}\label{alphairrat}
$\alpha$ is an additive continuous eigenvalue.
\end{corollary}
\begin{proof}
Since $\lambda = \frac{a_{0}}{b_{0} + \lambda_{1}}$, 
\begin{align*}
\alpha &= \frac{\lambda}{\len{v_{0}} + (\len{u_{0}}-\len{v_{0}})\lambda}
= \frac{1}{\lambda^{-1}\len{v_{0}} + \len{u_{0}}-\len{v_{0}}}
= \frac{1}{\frac{b_{0}+\lambda_{1}}{a_{0}}\len{v_{0}} + \len{u_{0}} - \len{v_{0}}} \\
&= \frac{a_{0}}{b_{0}\len{v_{0}} + \lambda_{1}\len{v_{0}} + a_{0}(\len{u_{0}} - \len{v_{0}})}
= \frac{a_{0}}{\len{v_{1}} + \lambda_{1}\len{v_{0}}} = a_{0}\alpha_{1}. 
\end{align*}
By Proposition \ref{2piqalpha}, $\alpha_{1}$ is a continuous additive eigenvalue so $\alpha$ is as well.
\end{proof}

\begin{corollary}\label{circlefactor}
There is a continuous factor map $(X,\sigma) \to (S^{1},R_{\alpha})$ where $R_{\alpha}$ denotes rotation by $\exp(2\pi i \alpha)$.  The same holds for $(S^{1},R_{\alpha_{k}})$ for each $k$.
\end{corollary}
\begin{proof}
Let $f_{\alpha} : X \to S^{1}$ be a continuous eigenfunction for $\exp(2\pi i \alpha)$.  Then $f_{\alpha}(\sigma x) = \exp(2\pi i \alpha)f_{\alpha}(x)$ so $f_{\alpha}$ is the factor map.  The same reasoning applies to $\alpha_{k}$.
\end{proof}

The next step is to establish the rational eigenvalues.

\begin{proposition}\label{rational}
A rational number $\nicefrac{m}{n}$ is an additive continuous eigenvalue if $n$ eventually divides the lengths of both $u_{k}$ and $v_{k}$, equivalently the lengths of both $v_{k}$ and $v_{k+1}$.
\end{proposition}
\begin{proof}
Assume that $n$ divides the length of $u_{k}$ and $v_{k}$ for some $k$.  Let $B$ be the clopen set of $x\in X$ such that as a $k$-concatenation, $\sigma^{sn}x$ has $v_{k}$ or $u_{k}$ at the origin for some integer $s$.  Then $\sigma^{n}B = B$, and so $e^{2\pi i/n}$ has continuous eigenfunction $\sum_{k = 0}^{n-1} \chi_{\sigma^k B} e^{2\pi i k/n}$.  Therefore $n^{-1}$ is an additive continuous eigenvalue so $\nicefrac{m}{n}$ also is.
\end{proof}

Using the above, we prove that every element of $Q_{X}\alpha$ is, up to a rational, an element of the additive continuous eigenvalue group; this essentially just comes from Proposition~\ref{2piqalpha}.

\begin{proposition}\label{gen}
For all $q \in Q_{X}$ there exists $r_{q} \in \mathbb{Q}$ such that $q\alpha + r_{q} \in E_{X}$.
\end{proposition}
\begin{proof}
Since $\lambda_{k} = \frac{a_{k}}{b_{k} + \lambda_{k+1}}$, we have that $\lambda_{k+1} = a_{k}\lambda_{k}^{-1} - b_{k}$.  Therefore, for $k \geq 1$,
\[
\alpha_{k+1} = \frac{1}{\len{v_{k+1}} + \len{v_{k}}\lambda_{k+1}}
= \frac{1}{b_{k}\len{v_{k}} + a_{k}\len{v_{k-1}} + a_{k}\len{v_{k}}\lambda_{k}^{-1} - b_{k}\len{v_{k}}}
= \frac{\lambda_{k}}{a_{k}(\len{v_{k}} + \len{v_{k-1}}\lambda_{k})} = \frac{\alpha_{k}\lambda_{k}}{a_{k}}.
\]

We claim now that $\alpha_{k}\lambda_{k} = (-1)^{k}\frac{\len{v_{k}}}{\len{v_{0}}a_{0}\cdots a_{k-1}}\alpha + r_{k}$ for some $r_{k} \in \mathbb{Q}$.
Clearly $\alpha_{0}\lambda_{0} = \alpha = \frac{\len{v_{0}}}{\len{v_{0}}}\alpha + 0$.  Observe that
\[
\alpha_{1}\lambda_{1} = \frac{\lambda_{1}}{\len{v_{1}} + \len{v_{0}}\lambda_{1}}
= \frac{a_{0}\lambda^{-1} - b_{0}}{b_{0}\len{v_{0}} + a_{0}(\len{u_{0}} - \len{v_{0}}) + a_{0}\lambda^{-1}\len{v_{0}} - b_{0}\len{v_{0}}}
= \frac{a_{0} - b_{0}\lambda}{a_{0}(\len{v_{0}} + (\len{u_{0}}-\len{v_{0}})\lambda)}
\]
so we have
\begin{align*}
\alpha_{1}\lambda_{1} + \frac{\len{v_{1}}}{\len{v_{0}}a_{0}}\alpha - \frac{1}{\len{v_{0}}}
&= \frac{\len{v_{0}}(a_{0} - b_{0}\lambda) + \len{v_{1}}\lambda - a_{0}(\len{v_{0}} + (\len{u_{0}}-\len{v_{0}})\lambda)}{a_{0}\len{v_{0}}(\len{v_{0}} + (\len{u_{0}}-\len{v_{0}})\lambda)} \\
&= \frac{-b_{0}\lambda\len{v_{0}} + b_{0}\lambda\len{v_{0}} + a_{0}\lambda(\len{u_{0}}-\len{v_{0}}) - a_{0}(\len{u_{0}}-\len{v_{0}})\lambda}{a_{0}\len{v_{0}}(\len{v_{0}} + (\len{u_{0}}-\len{v_{0}})\lambda)}
= 0.
\end{align*}
Assume that $\alpha_{k}\lambda_{k} = (-1)^{k}\frac{\len{v_{k}}}{\len{v_{0}}a_{0}\cdots a_{k-1}}\alpha + r_{k}$ and likewise for $k-1$.  Then
\begin{align*}
\alpha_{k+1}\lambda_{k+1} &= \frac{a_{k}\lambda_{k}^{-1} - b_{k}}{b_{k}\len{v_{k}} + a_{k}\len{v_{k-1}} + a_{k}\len{v_{k}}\lambda_{k}^{-1} - b_{k}\len{v_{k}}}
= \frac{a_{k} - b_{k}\lambda_{k}}{a_{k}(\len{v_{k}} + \len{v_{k-1}}\lambda_{k})}
= \alpha_{k} - \frac{b_{k}\alpha_{k}\lambda_{k}}{a_{k}} \\
&= \frac{\alpha_{k-1}\lambda_{k-1}}{a_{k-1}} - \frac{b_{k}\alpha_{k}\lambda_{k}}{a_{k}}
= (-1)^{k-1}\frac{\len{v_{k-1}}}{\len{v_{0}}a_{0}\cdots a_{k-1}}\alpha + \frac{r_{k-1}}{a_{k-1}} - (-1)^{k}\frac{b_{k}\len{v_{k}}}{\len{v_{0}}a_{0}\cdots a_{k}}\alpha - \frac{b_{k}r_{k}}{a_{k}} \\
&= (-1)^{k+1} \frac{a_{k}\len{v_{k-1}} + b_{k}\len{v_{k}}}{\len{v_{0}}a_{0}\cdots a_{k}}\alpha + \frac{r_{k-1}}{a_{k-1}} - \frac{b_{k}r_{k}}{a_{k}} = (-1)^{k+1} \frac{\len{v_{k+1}}}{\len{v_{0}}a_{0}\cdots a_{k}}\alpha + \frac{r_{k-1}}{a_{k-1}} - \frac{b_{k}r_{k}}{a_{k}} 
\end{align*}
so by induction, the claim holds.  Then
\[
\alpha_{k+1} = \frac{\alpha_{k}\lambda_{k}}{a_{k}} = (-1)^{k}\frac{\len{v_{k}}}{\len{v_{0}}a_{0}\cdots a_{k}}\alpha + \frac{r_{k}}{a_{k}}, 	
\]
and linear combinations of these multiples of $\alpha$ can clearly create all of $Q_X$.
$\qedhere$
\end{proof}

\begin{proposition}\label{contones}
The group of additive continuous eigenvalues $E_X$ contains $\{q \alpha + r_q + r \ : \ q \in Q_X, r \in R_X\}$.
\end{proposition}

\begin{proof}
This is an immediate consequence of Propositions~\ref{gen} and \ref{rational}.
\end{proof}

\subsection{Additive measurable eigenvalues}

We now prove that every additive measurable eigenvalue is contained in $Q_X \alpha + \mathbb{Q}$.

\begin{proposition}\label{2pigamma}
Let $\gamma$ be an additive measurable eigenvalue. Then there exists $q \in Q_{X}$ and $r \in \mathbb{Q}$ such that $\gamma = q\alpha + r$.
\end{proposition}
\begin{proof}
Let $f$ be a measurable eigenfunction with eigenvalue $\exp(2\pi i \gamma)$.  Let $B_{k}$ and $B_{k}^{\prime}$ as in Lemma \ref{TT}.
For each $k$, define 
\begin{align*}
f_k(x) &=\sum_{j=0}^{\len{v_{k+1}}-1}  \frac{1}{\mu(B_k)} \left(\int_{\sigma^j B_k} f \ d\mu\right) \bbone_{\sigma^j B_{k}}(x) 
+ \sum_{j=0}^{\len{u_{k+1}}-\len{v_{k+1}}-1}  \frac{1}{\mu(B_k^{\prime})} \left(\int_{\sigma^j B_k^{\prime}} f \ d\mu\right) \bbone_{\sigma^j B_{k}^{\prime}}(x).
\end{align*}
Let $\mathcal{F}_{k}$ be the $\sigma$-algebra generated by the sets $\sigma^{j}B_{k}$, $0 \leq j < \len{v_{k+1}}$, and $\sigma^{j}B_{k}^{\prime}$, $0 \leq j < \len{u_{k+1}^{\prime}}$. Note that $\mu(T_k \sqcup T_k^{\prime}) = 1$, and that for any $L$, 
all levels of $T_k$ and $T_k^{\prime}$ except the top and bottom $L$ have diameter less than $2^{-L}$. Since the heights of the towers $\len{u_{k+1}}-\len{v_{k+1}} \geq \len{v_{k}} \to \infty$, these top and bottom $L$ levels have measures approaching $0$. Since $L$ was arbitrary, the $\sigma$-algebras $\mathcal{F}_{k}$ converge to the $\sigma$-algebra of all measurable sets.  Since each $f_{k}$ is $\mathcal{F}_{k}$-measurable and $\mathbb{E}[f|\mathcal{F}_{k}] = \mathbb{E}[f_{k+1}|\mathcal{F}_{k}] = f_{k}$, by the Martingale Convergence Theorem, $f_{k}$ converge almost everywhere to $f$.

For all $t \geq 0$, we have that $\sigma^{\len{v_{k+t+1}}}$ takes every occurrence of $v_{k+t+1}$ to an occurrence of $v_{k+t+1}$ except those immediately followed by an occurrence of $u_{k+t+1}$.  Therefore $\sigma^{\len{v_{k+t+1}}}$ takes every occurrence of $v_{k+1}$ in a $v_{k+t+1}$ to an occurrence of $v_{k+1}$ except for those in a $v_{k+t+1}$ immediately followed by a $u_{k+t+1}$.  Likewise, for $0 < i_{k+t+1}$, $\sigma^{i_{k+t+1}\len{v_{k+t+1}}}$ takes every occurrence of a $v_{k+1}$ in a $v_{k+t+1}$ to an occurrence of $v_{k+1}$ except for those in a $v_{k+t+1}$ less than $i_{k+t+1}$ words prior to a $u_{k+t+1}$.

For any $(k+t+2)$-concatenation, since $u_{k+t+2}$ has $v_{k+t+2}$ as a suffix, the concatenation is a concatenation of $v_{k+t+1}^{m_{k+t+1}}u_{k+t+1}^{\prime}$ and $v_{k+t+1}^{n_{k+t+1}}u_{k+t+1}^{\prime}$ and, if $r_{k+t+1} > 0$, $v_{k+t+1}^{r_{k+t+1}}u_{k+t+1}^{\prime}$ where $u_{k+t+1}^{\prime}$ is the prefix of $u_{k+t+1}$ such that $u_{k+t+1} = u^{\prime}v_{k+t+1}$.  Since $n_{k+t+2} \leq 2m_{k+t+2} + 2$, then $\len{u_{k+t+1}^{\prime}} \leq \len{v_{k+t+1}} + 2\len{v_{k+t}} < 3\len{v_{k+t+1}}$ so at least $\frac{1}{4}$ of the $v_{k+1}$ appearing in a $(k+t+1)$-concatenation are in a $v_{k+t+1}$.

Let $\{ i_{k} \}$ such that $0 < i_{k+t} \leq \max(1,0.5b_{k+t})$.  Write $d_{k} = \len{v_{k+1}}$.  For $k+t$ such that $b_{k+t} > 1$, then
\[
\mu(\sigma^{i_{k+t}d_{k+t}}B_{k} \cap B_{k}) \geq \frac{b_{k+t}-i_{k+t}}{b_{k+t}} \left( \frac{1}{4} \mu(B_{k})\right)
\geq \frac{1}{8}\mu(B_{k}).
\]
For $k$ such that $b_{k+t} = 1$, meaning $r_{k+t} = 0$ and $m_{k+t} = 1$, we have that $\sigma^{\len{v_{k+t+1}}} = \sigma^{\len{u_{k+t}}}$ takes every occurrence of $v_{k+t}$ which precedes a $u_{k+t}$ to the $v_{k+t}$ which is a suffix of 
that $u_{k+t}$.  Since $n_{k+t} \leq 4$, at least $\frac{1}{4}$ of the words in a $(k+t+1)$-concatenation are $u_{k+t}$ so at least $\frac{1}{4}$ of the $v_{k+t}$ are taken to a $v_{k+t}$ by $\sigma^{\len{v_{k+t+1}}}$ (since $u_{k+t}$ is always preceded by $v_{k+t}$, possibly as a suffix of another $u_{k+t}$).  Then,
\[
\mu(\sigma^{d_{k+t}}B_{k} \cap B_{k}) \geq  \frac{1}{4} \mu(B_{k}).
\]
Then $f_{k}(\sigma^{i_{k+t}d_{k+t}}x) = f_{k}(x)$ for a set of measure at least $\frac{1}{8}\mu(T_{k}) \geq \frac{1}{32}$.  Since $f_{k} \to f$ almost everywhere, there is then a positive measure set such that for any sufficiently small $\epsilon > 0$ and almost every $x$ in the set, there exists $k$ so that for all $t$, $|f(\sigma^{i_{k+t}d_{k+t}}x) - f(x)| < \epsilon$.  Therefore $\exp(2\pi i \gamma i_{k}d_{k}) \to 1$.

For large enough $k$ (say $k \geq k_0$), $\langle i d_{k} \gamma \rangle < \frac{1}{160}$ for all $0 < i \leq \max(1,0.5b_{k+1})$.  Suppose that for all $c \in \mathbb{Z}$, we have $|c - d_{k}\gamma| \geq 0.05\frac{d_{k}}{d_{k+1}} \geq 0.05(2b_{k+1}+2)^{-1}$ (using that $d_{k+1} = b_{k+1}d_{k} + a_{k+1}d_{k-1} \leq b_{k+1}d_{k} + (b_{k+1}+2)d_{k-1}$).  
Then $|\max(1,\floor{0.5b_{k+1}})c - \max(1,\floor{0.5b_{k+1}})d_{k}\gamma| \geq 0.025\frac{\max(1,\floor{0.5(b_{k+1}})}{b_{k+1}+1} \geq \frac{1}{160}$, a contradiction.
This implies that for all $k \geq k_0$, there exists $c'_k \in \mathbb{Z}$,
so that $\left|\gamma - \frac{c'_k}{d_k} \right| < 0.05(d_{k+1})^{-1}$.

We will prove that $c'_{k+1} = b_{k+1} c'_k + a_{k+1} c'_{k-1}$ for all $k > k_0$.

For $k > k_0$, let $c''_{k+1} = b_{k+1} c'_k + a_{k+1} c'_{k-1}$. By the above,
\[
\left|\gamma - \frac{c'_{k-1}}{d_{k-1}} \right| < 0.05(d_k)^{-1} \textrm{ and }
\left|\gamma - \frac{c'_k}{d_k} \right| < 0.05(d_{k+1})^{-1} \textrm{ so } 
\left|d_{k+1} \gamma - c'_k \frac{d_{k+1}}{d_k} \right| < 0.05.
\]

Since $|d_{k-1}\gamma - c_{k-1}^{\prime}| < 0.05d_{k-1}(d_{k})^{-1}$,
\[
\left|a_{k+1} d_{k-1} \gamma - a_{k+1} c'_{k-1}\right| < \frac{0.05a_{k+1} d_{k-1}}{d_{k}} \leq \frac{0.05(b_k+2) d_{k-1}}{d_{k}} < 0.05(3) = 0.15.
\]

Similarly, since $|d_{k}\gamma - c_{k}^{\prime}| < 0.05d_{k}(d_{k+1})^{-1}$,
\[
\left|a_{k+1} d_{k-1} \gamma - c'_k a_{k+1} \frac{d_{k-1}}{d_k}\right| 
= \left|a_{k+1}d_{k-1}\left(\gamma - \frac{c_{k-1}^{\prime}}{d_{k}}\right)\right|
< \frac{0.05a_{k+1} d_{k-1}}{d_{k+1}} \leq \frac{0.05(b_k +2)d_{k-1}}{d_{k+1}} < 0.15.
\]

Therefore,
\begin{align*}
\left|c'_k \frac{d_{k+1}}{d_k} - c''_{k+1}\right| &= \left|c'_k\left(b_{k+1} + \frac{a_{k+1} d_{k-1}}{d_k}\right) - b_{k+1} c'_k - a_{k+1} c'_{k-1}\right| = \left| c'_k a_{k+1} \frac{d_{k-1}}{d_k} - a_{k+1} c'_{k-1} \right| \\
&\leq \left| c'_k a_{k+1} \frac{d_{k-1}}{d_k} - a_{k+1}d_{k-1}\gamma\right| + \left|a_{k+1}d_{k-1}\gamma - a_{k+1} c'_{k-1} \right|
< 0.3.
\end{align*}

Combining with $|d_{k}\gamma - c_{k}^{\prime}| < 0.05d_{k}(d_{k+1})^{-1}$ via the triangle inequality yields
\[
\left|d_{k+1} \gamma - c''_{k+1}\right| \leq \left|d_{k+1}\gamma - c_{k}^{\prime}\frac{d_{k+1}}{d_{k}}\right| + \left|c_{k}^{\prime}\frac{d_{k+1}}{d_{k}} - c_{k}^{\prime\prime}\right|
< \frac{d_{k+1}}{d_{k}}|d_{k}\gamma - c_{k}^{\prime}| + 0.3 < 0.35.
\]
Recall that by definition, 
\[
\left|\gamma - \frac{c'_{k+1}}{d_{k+1}}\right| < 0.05 (d_{k+2})^{-1}, \textrm{ and so }
\left|d_{k+1} \gamma - c'_{k+1}\right| < 0.05 \frac{d_{k+1}}{d_{k+2}} < 0.05.
\]
This implies that $c'_{k+1} = c''_{k+1}$ (since they are both integers).

For $-2 \leq k < k_{0}$, define $c_{k}^{\prime} \in \mathbb{Q}$ using the recursion relation $c_{k+1}^{\prime} = b_{k+1}c_{k}^{\prime} + a_{k+1}c_{k-1}^{\prime}$ in reverse.
Since the recurrence relations defining $c_{k}, e_{k}, d_{k}$ and $c_{k}^{\prime}$ are the same linear relation, $c_{k}^{\prime} = c_{-2}^{\prime} c_{k} + c_{-1}^{\prime}e_{k}$ so
\[
\gamma = \lim \frac{c_{k}^{\prime}}{d_{k}} = \lim \frac{c_{-2}^{\prime}c_{k} + c_{-1}^{\prime}e_{k}}{e_{k}} \frac{e_{k}}{d_{-2}c_{k} + d_{-1}e_{k}} = \frac{c_{-2}^{\prime}\lambda + c_{-1}^{\prime}}{d_{-2}\lambda + d_{-1}}.
\]
Then
\begin{align*}
\gamma &= c_{-2}^{\prime}\alpha + \frac{c_{-1}^{\prime}}{d_{-2}\lambda + d_{-1}} = c_{-2}^{\prime}\alpha + \frac{c_{-1}^{\prime}}{d_{-1}} \frac{d_{-1}}{d_{-2}\lambda + d_{-1}}
= c_{-2}^{\prime}\alpha + \frac{c_{-1}^{\prime}}{d_{-1}} \Big{(} 1 - \frac{d_{-2}\lambda}{d_{-2}\lambda + d_{-1}}\Big{)} \\
&= c_{-2}^{\prime}\alpha + \frac{c_{-1}^{\prime}}{d_{-1}}(1 - d_{-2}\alpha)
= \Big{(}c_{-2}^{\prime} - \frac{c_{-1}^{\prime}d_{-2}}{d_{-1}}\Big{)}\alpha + \frac{c_{-1}^{\prime}}{d_{-1}}
\end{align*}
meaning that
\[
d_{-1}\gamma = (c_{-2}^{\prime}d_{-1} - c_{-1}^{\prime}d_{-2})\alpha + c_{-1}^{\prime}.
\]

It is easily seen by induction that
\[
\frac{c_{k}^{\prime}d_{k+1} - c_{k+1}^{\prime}d_{k}}{c_{-2}^{\prime}d_{-1} - c_{-1}^{\prime}d_{-2}} = (-1)^{k} a_{0} \cdots a_{k+1}
= \frac{c_{k}^{\prime}c_{k+1} - c_{k+1}^{\prime}c_{k}}{c_{-2}^{\prime}c_{-1} - c_{-1}^{\prime}c_{-2}}
\]
and therefore, as $c_{k}^{\prime} \in \mathbb{Z}$ for $k \geq k_{0}$ and $c_{-1} = 0$ and $c_{-2} = 1$,
\[
(c_{-2}^{\prime}d_{-1} - c_{-1}^{\prime}d_{-2})\frac{a_{0}\cdots a_{k_{0}+1}}{\gcd(d_{k_{0}},d_{k_{0}+1})},~c_{-1}^{\prime}a_{0}\cdots a_{k_{0}} \in \mathbb{Z}.
\]
Then
$
\gamma = q\alpha + r
$ for some $q \in \frac{\gcd(d_{k_{0}},d_{k_{0}+1})}{d_{-1}a_{0}\cdots a_{k_{0}+1}}\mathbb{Z}$ and $r \in \mathbb{Q}$.
\end{proof}

\subsection{Rational additive measurable eigenvalues}

Next, we establish that the only rational additive measurable eigenvalues are those in $R_X$. 

\begin{proposition}\label{ratmeas}
If a rational number $\nicefrac{m}{n}$ in lowest terms is an additive measurable eigenvalue then $n$ eventually divides 
the lengths of both 
$u_{k}$ and $v_{k}$, equivalently the lengths of both $v_{k}$ and $v_{k+1}$.
\end{proposition}

\begin{proof}
It suffices to prove the case when $m = 1$ and $n$ is a prime power.
Assume that $p^{-r}$ is an additive eigenvalue for a prime $p$ and integer $r \geq 1$.  
Then there exists a positive measure set $A$ such that $\sigma^{p^{r}}A = A$ and $\sigma^{p^{r-1}}A$ is disjoint from $A$.  Let $B_{k}$ and $B_{k}^{\prime}$ as in Lemma \ref{TT}.
Since cylinder sets generate the algebra of measurable sets, there exists 
$S_{k}^{v} \subseteq \{ 0, \ldots, \len{v_{k+1}}-1 \}$ and $S_{k}^{u} \subseteq \{ 0, \ldots, \len{u_{k+1}^{\prime}}-1 \}$ such that $A_{k} = \bigsqcup_{j \in S_{k}^{v}}\sigma^{j}B_{k} \sqcup \bigsqcup_{j \in S_{k}^{u}}\sigma^{j}B_{k}^{\prime}$ has $\mu(A_{k} \symdiff A) \to 0$.
Since $\sigma^{p^{r}}A = A$, $\sup_{t \in \mathbb{Z}} \mu(A_{k} \symdiff \sigma^{p^{r}t}A_{k}) \to 0$.

Set $A_{k}^{v} = A_{k} \cap T_{k}$.  Since $\mu(T_{k}) \geq \frac{1}{4}$ and $\| \bbone_{\sigma T_{k}} - \bbone_{T_{k}} \|_{2} \to 0$, by Lemma 3.6 \cite{danilenkopr}, $\liminf \mu(A \cap T_{k}) \geq \frac{1}{4}\mu(A)$.  Then  $\mu(A_{k}^{v})$ is uniformly bounded above zero for sufficiently large $k$.

For $p^{r} \leq j < \len{v_{k+1}}$, if $j \in S_{k}^{v}$ and $j-p^{r} \notin S_{k}^{v}$ then $\sigma^{j}B_{k} \subseteq A_{k}^{v}$ and $\sigma^{j-p^{r}}B_{k} \cap A_{k} = \emptyset$ so $\sigma^{j}B_{k} \subseteq A_{k}^{v} \setminus \sigma^{p^{r}}A_{k}$.  Therefore, since $\mu(A_{k}^{v}) = |S_{k}^{v}|\mu(B_{k})$,
\[
\frac{1}{|S_{k}^{v}|}\left|\{ j < \len{v_{k+1}} : j \in S_{k}^{v}, j-p^{r} \notin S_{k}^{v} \}\right| \leq \frac{p^{r}}{|S_{k}^{v}|} + \frac{\mu(A_{k}^{v} \setminus \sigma^{p^{r}}A_{k})}{\mu(A_{k}^{v})} \leq \frac{p^{r}}{|S_{k}^{v}|} + \frac{\mu(A_{k} \setminus \sigma^{p^{r}}A_{k})}{\mu(A_{k}^{v})} \to 0
\]
so $\frac{1}{|S_{k}^{v}|}|\{ j \in S_{k}^{v} : j - p^{r} \in S_{k}^{v} \}| \to 1$.

Choose $t_{k} \in \mathbb{Z}$ such that $p^{r}t_{k} = \len{v_{k+1}} + \ell_{k}$ for some $0 < \ell_{k} \leq p^{r}$.  Then $\sigma^{p^{r}t_{k}}B_{k} \subseteq \sigma^{\ell_{k}}B_{k} \sqcup \sigma^{\ell_{k}}B_{k}^{\prime}$.  Since the set of $x \in X$ such that $x$ has $v_{k}^{2}$ at the origin is positive measure (as otherwise every $x$ would be a multiple of $u_{k}$), the same reasoning as above gives that $\frac{1}{|S_{k}^{v}|}|\{ j \in S_{k}^{v} : j + \ell_{k} \in S_{k}^{v} \}| \to 1$.

Since $0 < \ell_{k} \leq p^{r}$, there exists a constant $0 < \ell \leq p^{r}$ such that $\ell_{k_{i}} = \ell$ for infinitely many $k_{i}$ and we may assume $\ell$ is the minimal such constant.  Let $0 \leq z \leq r$ maximal such that $p^{z}$ divides $\ell$.  Then there exist integers $a < 0$ and $b > 0$ such that $ap^{r} + b\ell = p^{z}$.  As $a$ and $b$ are fixed and $\frac{1}{|S_{k_{i}}^{v}|}|\{ j \in S_{k_{i}}^{v} : j+\ell,j-p^{r} \in S_{k_{i}}^{v} \}| \to 1$, we have $\frac{1}{|S_{k_{i}}^{v}|}|\{ j \in S_{k_{i}}^{v} : j + p^{z} \in S_{k_{i}}^{v} \}| \to 1$.

Let $S_{k_{i}}^{\prime} = \{ j \in S_{k_{i}}^{u} : j = j_{0} + \len{u_{k_{i}+1}^{\prime}} - \len{v_{k_{i}+1}}~\text{for some $0 \leq j_{0} < \len{v_{k_{i}+1}} - p^{z}$} \}$ and $S_{k_{i}}^{\prime\prime} = \{ j \in S_{k_{i}}^{u} : j = j_{0} + \len{u_{k_{i}+1}^{\prime}} - 2\len{v_{k_{i}+1}}~\text{for some $0 \leq j_{0} < \len{v_{k_{i}+1}} - 2p^{z}$} \}$.  Since $\len{u_{k_{i}+1}} < 3\len{v_{k_{i}+1}}$ (Remark \ref{pands}), $|S_{k_{i}}^{u} \setminus (S_{k_{i}}^{\prime} \sqcup S_{k_{i}}^{\prime\prime})| \leq 3p^{z}$.

Since $u_{k_{i}+1}^{\prime}$, when it appears at the start of a $u_{k_{i}+1}$ in a concatenation, is always followed by $v_{k_{i}+1}$, $\sigma^{\len{u_{k_{i}+1}^{\prime}}}B_{k_{i}}^{\prime} \subseteq B_{k_{i}}$.  
For $j = j_{0} + \len{u_{k_{i}+1}^{\prime}} - \len{v_{k_{i}+1}} \in S_{k_{i}}^{\prime}$, then $\sigma^{p^{r}t_{k_{i}}}\sigma^{j}B_{k_{i}}^{\prime} \subseteq \sigma^{p^{z} + j_{0}}B_{k_{i}}$ which is a level in $T_{k_{i}}$ (as $j_{0} < \len{v_{k_{i}+1}}-p^{z}$) and for $j = j_{0} + \len{u_{k_{i}+1}^{\prime}} - 2\len{v_{k_{i}+1}} \in S_{k_{i}}^{\prime\prime}$, then $\sigma^{2p^{r}t_{k_{i}}}\sigma^{j}B_{k_{i}}^{\prime} \subseteq \sigma^{2p^{z} + j_{0}}B_{k_{i}}$ which is also a level in $T_{k_{i}}$.

Since $\mu(\sigma^{p^{r}t_{k_{i}}}A_{k_{i}} \symdiff A_{k_{i}}) \to 0$, then $\frac{1}{|S_{k_{i}}^{\prime}|}|\{ j \in S_{k_{i}}^{\prime} : p^{r}t_{k_{i}} + j - \len{u_{k_{i}+1}^{\prime}} \in S_{k_{i}}^{v} \}| \to 1$ and $\frac{1}{|S_{k_{i}}^{\prime\prime}|}|\{ j \in S_{k_{i}}^{\prime\prime} : 2p^{r}t_{k_{i}} + j - \len{u_{k_{i}+1}^{\prime}} \in S_{k_{i}}^{v} \}| \to 1$.  As $\frac{1}{|S_{k_{i}}^{v}|}|\{ j \in S_{k_{i}}^{v} : j + p^{z} \in S_{k_{i}}^{v} \}| \to 1$, then $\frac{1}{|S_{k_{i}}^{\prime}|}|\{ j \in S_{k_{i}}^{\prime} : j + p^{z} \in S_{k_{i}}^{\prime} \}| \to 1$ and likewise for $S_{k_{i}}^{\prime\prime}$ so $\frac{1}{|S_{k_{i}}^{u}|}|\{ j \in S_{k_{i}}^{u} : j + p^{z} \in S_{k_{i}}^{v} \}| \to 1$.

Then $\mu(A_{k_{i}} \symdiff \sigma^{p^{z}}A_{k_{i}}) \to 0$ meaning that $\mu(A \symdiff \sigma^{p^{z}}A) = 0$.  By choice of $A$ then $z = r$.  Therefore $p^{r}t_{k_{i}} = \len{v_{k_{i}+1}} + p^{r}$ so $p^{r}$ divides $\len{v_{k_{i}+1}}$.  As $\ell$ was chosen minimally, then $p^{r}t_{k} = \len{v_{k+1}} + p^{r}$ for all sufficiently large $k$ so $p^{r}$ divides $\len{v_{k}}$ for all sufficiently large $k$.
Since $\len{u_{k+1}} = \len{v_{k+1}} + (n_{k} - m_{k})\len{v_{k}}$, then $p^{r}$ divides $\len{u_{k}}$ for all sufficiently large $k$ as well.
\end{proof}

\subsection{The structure of the additive eigenvalue group}

We are now ready to establish the relationship between $Q_{X}$, $R_{X}$ and $E_{X}$. The key observation is that the map from $q$ to $r_q$ from Proposition~\ref{gen}, viewed modulo $R_X$, is in fact a homomorphism, and so we can describe its structure a bit more explicitly using $p$-adic valuations.

%The principal idea in the following is that $E_{X} / \mathbb{Z}$ is isomorphic to $Q_{X} \times R_{X} / \mathbb{Z}$ and the map $q \mapsto \sum_{p} \{ qe_{p} \}_{p}$ is a natural section for the short exact sequence $R_{X} / \mathbb{Z} \to E_{X} / \mathbb{Z} \to Q_{X}$.

\begin{proposition}\label{homomorph}
There exists a homomorphism $\phi : Q_{X} \to \mathbb{Q}/R_{X}$ with $\phi(1) = \phi(0)$ such that $q \mapsto q\alpha + \phi(q)$ is an isomorphism $Q_{X} \to E_{X}/R_{X}$.  
\end{proposition}
\begin{proof}
By Proposition \ref{gen}, for every $q \in Q_{X}$ there exists $r_{q} \in \mathbb{Q}$ such that $q\alpha + r_{q} \in E_{X}$.  Let $\phi(q) = r_{q} + R_{X}$.  If $r,r^{\prime} \in \mathbb{Q}$ such that $q\alpha + r, q\alpha + r^{\prime} \in E_{X}$ then $r - r^{\prime} \in E_{X} \cap \mathbb{Q} = R_{X}$ so for every $r \in \mathbb{Q}$ such that $q\alpha + r \in E_{X}$, we have $r \in \phi(q)$.  Since $\alpha \in E_{X}$, $\phi(1) = R_{X} = \phi(0)$.  Since $r_{q+q^{\prime}} - r_{q} - r_{q^{\prime}} = (q + q^{\prime})\alpha + r_{q+q^{\prime}} - (q\alpha + r_{q}) - (q^{\prime}\alpha + r_{q^{\prime}}) \in E_{X} \cap \mathbb{Q} = R_{X}$, $\phi$ is a homomorphism and therefore $q \mapsto q\alpha + \phi(q)$ is a homomorphism $Q_{X} \to E_{X}/R_{X}$.

By Proposition \ref{2pigamma}, every $\gamma \in E_{X}$ is of the form $q\alpha + r$ for some $q \in Q_{X}$ and $r \in \mathbb{Q}$ so $q \mapsto q\alpha + \phi(q)$ is onto.  Since $\alpha \notin \mathbb{Q}$, the kernel of $q \mapsto q\alpha + \phi(q)$ is $\{ 0 \}$ meaning $q \mapsto q\alpha + \phi(q)$ is an isomorphism.
\end{proof}

To characterize the structure of $E_{X}$, we need to establish the nature of such homomorphisms $\phi$.

\begin{proposition}\label{phi}
Let $0 \leq \ell_{p},r_{p} \leq \infty$ and
$\phi : Q_{(\ell_{p})} \to \mathbb{Q} / Q_{(r_{p})}$ be a homomorphism such that $\phi(1) = \phi(0)$.  Then there exist $e_{p} \in \mathbb{Q}_{p}$ for each prime $p$
such that for all $q \in Q_{(\ell_{p})}$,
\[
\phi(q) = \sum_{p} \{ \{ q \}_{p} e_{p} \}_{p} + Q_{(r_{p})}.
\]
\end{proposition}
\begin{proof}
Since $\phi(1) = \phi(0)$, there exists a homomorphism $\tilde{\phi} : Q_{(\ell_{p})} / \mathbb{Z} \to \mathbb{Q} / Q_{(r_{p})}$ such that $\phi(q) = \tilde{\phi}(q + \mathbb{Z})$.  As $Q_{(\ell_{p})} / \mathbb{Z}$ is an abelian torsion group (since it is a subgroup of $\mathbb{Q}/\mathbb{Z}$), it is isomorphic to the direct sum of its $p$-power torsion groups.  Concretely speaking, adopting the convention that $p^{-\infty}\mathbb{Z} = \mathbb{Z}\left[\nicefrac{1}{p}\right]$, the map $i : Q_{(\ell_{p})} / \mathbb{Z} \to \bigoplus_{p} p^{-\ell_{p}}\mathbb{Z} / \mathbb{Z}$ given by $i(q + \mathbb{Z}) = (\{ q \}_{p} + \mathbb{Z})_{p}$ is an isomorphism with inverse map given by $(x_{p} + \mathbb{Z})_{p} \mapsto \sum_{p} x_{p} + \mathbb{Z}$.  Likewise, $\mathbb{Q} / Q_{(r_{p})}$ is a torsion group and $j : \mathbb{Q} / Q_{(r_{p})} \to \bigoplus \mathbb{Z}\left[\nicefrac{1}{p}\right]/p^{-r_{p}}\mathbb{Z}$ by 
$j(r + Q_{(r_{p})}) = (\{ r \}_{p} + p^{-r_{p}}\mathbb{Z})_{p}$ is the isomorphism.  As $p$-power torsion elements must map to $p$-power torsion elements, there exist homomorphisms $\tilde{\phi}_{p} : p^{-\ell_{p}}\mathbb{Z} / \mathbb{Z} \to \mathbb{Z}\left[\nicefrac{1}{p}\right] / p^{-r_{p}}\mathbb{Z}$ such that $\tilde{\phi} = j^{-1} \circ \left( \bigoplus \tilde{\phi}_{p} \right) \circ i$.

For $p$ such that $r_{p} = \infty$, $\tilde{\phi}_{p}$ maps to the trivial group so $\tilde{\phi}_{p}(p^{-t}) = 0$ for all $t \leq \ell_{p}$ and we set $e_{p} = 0$.  For $p$ such that $r_{p} < \infty$ and $\ell_{p} = \infty$, for each $n > 0$, let $c_{p,n} \in \tilde{\phi}_{p}(p^{-n} + \mathbb{Z})$.  Then $p^{n}c_{p,n} \in \tilde{\phi}_{p}(\mathbb{Z}) = p^{-r_{p}}\mathbb{Z}$ and $p^{n+m}c_{p,n+m} - p^{n}c_{p,n} \in p^{n}(p^{m}\tilde{\phi}_{p}(p^{-n-m} + \mathbb{Z}) - \tilde{\phi}_{p}(p^{-n} + \mathbb{Z})) = p^{n}\tilde{\phi}_{p}(\mathbb{Z}) = p^{n-r_{p}}\mathbb{Z}$.  Then $p^{n}c_{p,n}$ is a $p$-adic Cauchy sequence so $p^{n}c_{p,n} \to e_{p} \in \mathbb{Q}_{p}$ and since $p^{n}c_{p,n} \in p^{-r_{p}}\mathbb{Z}$, $e_{p} \in p^{-r_{p}}\mathbb{Z}_{p}$.  

Let $e_{p} = \sum_{t=-r_{p}}^{\infty} e_{p,t}p^{t}$ and $p^{n}c_{p,n} = \sum_{t=-r_{p}}^{\infty} d_{p,n,t}p^{t}$ be the $p$-adic expansions.  Since $p^{n+m}c_{p,n+m} - p^{n}c_{p,n} \in p^{n-r_{p}}\mathbb{Z}$, $e_{p,t} = d_{p,n,t}$ for $t \leq n-r_{p}$ so $e_{p,t} = d_{p,n+r_{p},t}$ for $t \leq n$.  Then $p^{n}\{ p^{-n}e_{p,t} \}_{p} = \sum_{t=-r_{p}}^{n-1} e_{p,t}p^{t} = \sum_{t=-r_{p}}^{n-1} d_{p,n+r_{p},t}p^{t} = p^{n}\{ p^{-n}p^{n+r_{p}}c_{p,n+r_{p}} \}_{p}$ meaning that $\{ p^{-n}e_{p} \}_{p} = \{ p^{r_{p}}c_{p,n+r_{p}} \}_{p}$.  Now $p^{r_{p}}c_{p,n+r_{p}} - c_{p,n} \in \tilde{\phi}_{p}(\mathbb{Z}) = p^{-r_{p}}\mathbb{Z}$ so $\{ p^{r_{p}}c_{p,n+r_{p}} \}_{p} - c_{p,n} \in p^{-r_{p}}\mathbb{Z}$.  

Since we have dealt with all possibilities for $p$, for every $p$ and $n$ we have $\{ p^{-n}e_{p} \}_{p} \in \tilde{\phi}_{p}(p^{-n})$.

For $p$ such that $\ell_{p} < \infty$ and $r_{p} < \infty$, let $c_{p,\ell_{p}} \in \tilde{\phi}_{p}(p^{-\ell_{p}})$ and set $e_{p} = p^{\ell_{p}}c_{p,\ell_{p}} \in p^{-r_{p}}\mathbb{Z}$.  For $n \leq \ell_{p}$, $\{ p^{-n}e_{p} \}_{p} = \{ p^{-n}p^{\ell_{p}}c_{p,\ell_{p}} \}_{p} = \{ p^{\ell_{p}-n}c_{p,\ell_{p}} \}_{p} \in p^{\ell_{p}-n}\tilde{\phi}_{p}(p^{-\ell_{p}}) = \tilde{\phi}_{p}(p^{-n})$.  Therefore, for all $p$, there exist $e_{p} \in p^{-r_{p}}\mathbb{Z}_{p}$ such that $\{ p^{-n}e_{p} \}_{p} \in \tilde{\phi}_{p}(p^{-n})$ for all $n \leq \ell_{p}$ meaning that $\tilde{\phi}_{p}(x + \mathbb{Z}) = \{ xe_{p} \}_{p} + p^{-r_{p}}\mathbb{Z}$ for all $x \in p^{-\ell_{p}}\mathbb{Z}$.
Therefore for all $q \in Q_{(\ell_{p})}$,
\begin{align*}
\phi(q) &= \tilde{\phi}\left(q + \mathbb{Z}\right)
= j^{-1} \circ \left(\bigoplus \tilde{\phi}_{p}\right) \circ i\left(q + \mathbb{Z}\right) \\
&= j^{-1} \circ \left(\bigoplus \tilde{\phi}_{p}\right) \left(\left(\{ q \}_{p} + \mathbb{Z}\right)_{p}\right)
= j^{-1} \left( \left( \{ \{ q \}_{p} e_{p} \}_{p} + p^{-r_{p}}\mathbb{Z}\right)_{p}\right).
\end{align*}
Since $q - \{ q \}_{p} \in \mathbb{Z}_{p}$ and $\{ e_{p} \}_{p} \in p^{-r_{p}}\mathbb{Z}$, $\{ (q - \{ q \}_{p})e_{p} \}_{p} \in p^{-r_{p}}\mathbb{Z}$.  As $\{ qe_{p} \}_{p} = \{ (q - \{ q \}_{p})e_{p} \}_{p} + \{ \{ q \}_{p} e_{p} \}_{p} \pmod{\mathbb{Z}}$, then $\{ qe_{p} \}_{p} = \{ \{ q \}_{p} e_{p} \}_{p}\pmod{p^{-r_{p}}\mathbb{Z}}$.  Therefore
\[
\phi(q)
= j^{-1} \left( \left( \{ \{ q \}_{p} e_{p} \}_{p} + p^{-r_{p}}\mathbb{Z}\right)_{p}\right)
= j^{-1}  \left( \left( \{ \{ q \}_{p} e_{p} \}_{p} + p^{-r_{p}}\mathbb{Z}\right)_{p}\right)
= \sum \{ \{ q \}_{p} e_{p} \}_{p} + Q_{(r_{p})}. \qedhere
\]
\end{proof}

Now we are in a position to prove the explicit description of the eigenvalue group and verify that all eigenvalues are continuous.

\begin{proof}[Proof of Theorem \ref{evalgroup}]
Consider any additive (measurable) eigenvalue $\gamma$. By Proposition~\ref{2pigamma}, there exist $q \in Q_X$ and $r \in \mathbb{Q}$ so that $\gamma = q\alpha + r$. By Proposition~\ref{contones}, $q\alpha + r_q$ is an additive (even continuous) eigenvalue. Therefore, $r - r_q$ is a rational additive eigenvalue, which must be in $R_X$ by Proposition~\ref{ratmeas}, and so $\gamma = q\alpha + r_q + (r - r_q) \in 
\{q \alpha + r_q + r \ : \ q \in Q_X, r \in R_X\}$. Therefore, by Proposition~\ref{contones}, $\gamma$ is also an additive continuous eigenvalue. Since all eigenspaces are one-dimensional by ergodicity of the unique $\sigma$-invariant measure $\mu$, all eigenfunctions of $\gamma$ are continuous.

By Proposition \ref{homomorph}, $E_{X} = \{ q\alpha + r : r \in \phi(q) \}$ for some homomorphism $\phi : Q_{X} \to E_{X}/R_{X}$.  By Proposition \ref{phi}, there exists $e_{p} \in \mathbb{Q}_{p}$ such that
$\sum_{p} \{ \{ q \}e_{p} \}_{p} \in \phi(q)$ for all $q \in Q_{X}$.  Therefore $E_{X} = \{ q\alpha + \sum_{p} \{ \{ q \}_{p} e_{p} \}_{p} + r : q \in Q_{X}, r \in R_{X} \}$.
\end{proof}

\begin{remark}
We remark that many of the arguments from this section could be alternately phrased in terms of Bratteli-Vershik diagrams, since S-adic decompositions can be alternately viewed in this way. The results of \cite{DFM} would then yield alternate ways of deriving our characterizations. However, as we don't use the language of such diagrams elsewhere, we won't give a full description of that interpretation here.
\end{remark}

\section{Adelic nilmanifolds and central extensions}\label{adelicsection}

We now characterize the dual group of $Q_{(\ell_{p})}$ as an adelic abelian nilmanifold.

\begin{definition}
For $(\ell_{p})_{p \in \mathbb{P}}$ with $\ell_{p} \in \mathbb{N} \cup \{ \infty \}$, the associated groups are
\begin{align*}
&\text{the \textbf{$(\ell_{p})$-adeles}} &&\quad \mathbb{A}_{(\ell_{p})} = \mathbb{R} \times  \sideset{}{'}\prod_{p : \ell_{p} = \infty} \mathbb{Q}_{p}; \\
&\text{the \textbf{$(\ell_{p})$-lattice}} &&\quad \Gamma_{(\ell_{p})} = \mathbb{Z}[1/p~:~\ell_{p} = \infty]; \\
&\text{the \textbf{$(\ell_{p})$-torsion group}} &&\quad F_{(\ell_{p})} = \bigoplus_{p : 0 < \ell_{p} < \infty} \mathbb{Z} / p^{\ell_{p}}\mathbb{Z}; \text{and} \\
&\text{the \textbf{$(\ell_{p})$-compact group}} &&\quad K_{(\ell_{p})} = \prod_{p : 0 < \ell_{p} < \infty} \mathbb{Z} / p^{\ell_{p}}\mathbb{Z}.
\end{align*}
and the \textbf{$(\ell_{p})$-carry cocycle} $c_{(\ell_{p})} : F_{(\ell_{p})} \times F_{(\ell_{p})} \to \Gamma_{(\ell_{p})}$ is given by
\[
c_{(\ell_{p})}((x_{p}),(y_{p})) = \sum_{p : 0 < \ell_{p} < \infty} \left\{ \begin{array}{ll} 1 &\quad\text{if $(x_{p}\pmod p^{\ell_{p}}) + (y_{p}\pmod p^{\ell_{p}}) \geq p^{\ell_{p}}$} \\ 0 &\quad\text{otherwise} \end{array}\right.
\]
where the $(x_{p})$ and $(y_{p})$ are treated as elements of the sets $\{ 0, 1, \ldots, p^{\ell_{p}} - 1 \}$.
\end{definition}

Here we are treating $\mathbb{A}_{(\ell_{p})}$ as a subgroup of $\mathbb{A}$ where each coordinate is $0$ for $p$ such that $\ell_{p} \ne \infty$ and the restricted product requires that all but finitely many coordinates be $p$-adic integers.  Then $\Gamma_{(\ell_{p})}$ is easily checked to be a lattice under the diagonal embedding $\gamma \mapsto (\gamma, (-\gamma))$.

\begin{notation}
For topological groups $H$ and $K$, we write $K \rtimes_{c} H$ for the central extension of $H$ by $K$ via the twisted action by a cocycle $c : H \times H \to K$.

Recall that a central extension of $H$ by $K$ means there is a short exact sequence $0 \to K \to K \rtimes_{c} H \to H \to 0$ and the cocycle $c$ is the ``correction factor" for a section $H \to K \times_{c} H$. The cohomology class of $c$ uniquely determines the structure of $K \rtimes_{c} H$ up to isomorphism so we will sometimes write $K \rtimes_{[c]} H$ where $[c]$ is the cohomology class.
\end{notation}

The adelic abelian nilmanifolds we will need are central extensions of compact abelian groups by adelic groups quotiented by lattices.

\begin{definition}\label{defnadelic}
The \textbf{$(\ell_{p})$-adelic abelian nilmanifold} is
\[
\mathcal{M}_{(\ell_{p})}  = \bigslant{\left(\mathbb{A}_{(\ell_{p})} \times K_{(\ell_{p})}\right)}{\Gamma_{(\ell_{p})}}
\]
where $\Gamma_{(\ell_{p})}$ is embedded diagonally via the map $\gamma \mapsto (\gamma, \{ -\gamma \}_{p}, \{ -\gamma\pmod{p^{\ell_{p}}} \}_{p})$.

A concrete realization of this group as a central extension is that
\[
\mathcal{M}_{(\ell_{p})} \simeq K_{(\ell_{p})} \rtimes_{[\hat{c}]} \left(\bigslant{\mathbb{A}_{(\ell_{p})}}{\Gamma_{(\ell_{p})}}\right)
\]
where $[\hat{c}]$ denotes the class of the topological extension obtained by Pontryagin dualizing the original discrete abelian extension via the $(\ell_{p})$-carry cocycle.
\end{definition}

The odometers we will need are also characterized in terms of $(\ell_{p})$.

\begin{definition}\label{defnlpodo}
The \textbf{$(\ell_{p})$-odometer} is
\[
\mathcal{O}_{(\ell_{p})} = \varprojlim \bigslant{\mathbb{Z}}{\prod_{p \leq k} p^{\min(k,\ell_{p})}\mathbb{Z}}.
\]
\end{definition}

\begin{proposition}\label{Qce}
$
Q_{(\ell_{p})} = \Gamma_{(\ell_{p})} \rtimes_{c} F_{(\ell_{p})}
$
where $c$ is the $(\ell_{p})$-carry cocycle.
\end{proposition}
\begin{proof}
Clearly $\Gamma_{(\ell_{p})}$ embeds in $Q_{(\ell_{p})}$ and there is a natural map $Q_{(\ell_{p})} \to F_{(\ell_{p})}$ given by $q \mapsto p^{\ell_{p}} \{q\}_{p}\pmod p^{\ell_{p}}$ where we treat each $\mathbb{Z}/p^{\ell_{p}}\mathbb{Z}$ as the set $\{ 0, 1, \ldots, p^{\ell_{p}} - 1 \}$. It is easily checked that $0 \to \Gamma_{(\ell_{p})} \to Q_{(\ell_{p})} \to F_{(\ell_{p})} \to 0$ is an exact sequence.

For the section $s : F_{(\ell_{p})} \to Q_{(\ell_{p})}$ given by $s((x_{p})) = \sum (x_{p} / p^{\ell_{p}})$ (which is convergent as $F_{(\ell_{p})}$ is a direct sum so only finitely many $x_{p}$ are nonzero),
\begin{align*}
s((x_{p})) + s((y_{p})) - s((x_{p}+y_{p}))
&= \sum \frac{x_{p}}{p^{\ell_{p}}} + \sum \frac{y_{p}}{p^{\ell_{p}}} - \sum \frac{x_{p} + y_{p}\pmod p^{\ell_{p}}}{p^{\ell_{p}}} \\
&= \sum \left\{ \begin{array}{ll} 1 &\quad\text{if $(x_{p}\pmod p^{\ell_{p}}) + (y_{p}\pmod p^{\ell_{p}}) \geq p^{\ell_{p}}$} \\ 0 &\quad\text{otherwise} \end{array}\right. \qedhere
\end{align*}
\end{proof}

\begin{proposition}\label{Qhatce}
$\widehat{Q_{(\ell_{p})}} = K_{(\ell_{p})} \rtimes_{[\hat{c}]} \widehat{\Gamma_{(\ell_{p})}}$.
\end{proposition}
\begin{proof}
This is just the fact that short exact sequences are reversed under Pontryagin duality and that the group structure is uniquely determined by the cohomology class of the cocycle.
\end{proof}

Our next task is to characterize the character groups of $Q_{(\ell_{p})}$ and $Q_{(r_{p})}/\mathbb{Z}$.  We begin with an observation connecting such characters to $p$-adic integers.

\begin{lemma}\label{achar}
Let $0 \leq \ell_{p} \leq \infty$ for each prime $p$ and
$\chi \in \widehat{Q_{(\ell_{p})}}$.  Then there exists a unique $\theta \in [0,1)$ and unique $z_{p} \in \mathbb{Z}_{p}$ with $0 \leq z_{p} < p^{\ell_{p}}$ when $\ell_{p} < \infty$ (in particular $z_{p} = 0$ when $\ell_{p} = 0$) such that for all $q \in Q_{(\ell_{p})}$,
\[
\chi(q) = \exp\left(2\pi i \left(q\theta + \sum \{ qz_{p} \}_{p}\right)\right).
\]
\end{lemma}
\begin{proof}
Let $\theta \in [0,1)$ be the unique value such that $\chi(1) = \exp(2\pi i \theta)$ and let $\chi^{\prime}(q) = \chi(q)/\exp(2\pi i q \theta)$.  Then $\chi^{\prime} \in \widehat{Q_{(\ell_{p})}}$ and $\chi^{\prime}(1) = 1$.  For $0 < t \leq \ell_{p}$,
\[
\left(\chi^{\prime}(p^{-t})\right)^{p^{t}} = \left(\frac{\chi(p^{-t})}{\exp(2\pi i p^{-t}\theta)}\right)^{p^{t}} = \frac{\chi(1)}{\exp(2\pi i \theta)} = 1
\]
so there exist unique integers $0 \leq z_{p,t} < p^{t}$ such that $\chi^{\prime}(p^{-t}) = \exp(2\pi i p^{-t}z_{p,t})$.  Since $(\chi^{\prime}(p^{-t-1}))^{p} = \chi^{\prime}(p^{-t})$, we have $z_{p,t+1} \pmod{p^{t}} = z_{p,t}$.  For $p$ such that $\ell_{p} < \infty$, set $z_{p,t} = z_{p,\ell_{p}}$ for $t > \ell_{p}$.  Then $z_{p,t} \to z_{p} \in \mathbb{Z}_{p}$ and $0 \leq z_{p} < p^{\ell_{p}}$ when $\ell_{p} < \infty$.

Since $\{ p^{-t}z_{p} \} = p^{-t}z_{p,t}$, then $\chi^{\prime}(p^{-t}) = \exp(2\pi i \{ p^{-t}z_{p} \}_{p})$ for all $p$ and $t \leq \ell_{p}$.  Let $q \in Q_{(\ell_{p})}$.  For each prime $p$, $q$ has $p$-adic expansion $\sum_{t=-m}^{\infty} q_{p,t}p^{t}$ for some $m \leq \ell_{p}$ so
\[
\chi^{\prime}(\{ q \}_{p})
= \prod_{t=-m}^{-1} \chi^{\prime}(q_{p,t}p^{t})
= \prod_{t=-m}^{-1} \exp(2\pi i \{ q_{p,t} p^{t} z_{p} \}_{p})
= \exp(2\pi i \{ qz_{p} \}_{p})
\]
and as $q = \sum \{ q \}_{p}\pmod{\mathbb{Z}}$,
\begin{align*}
\chi(q) &= \exp(2\pi i q \theta) \chi^{\prime}(q)
= \exp(2\pi i q \theta) \prod \chi^{\prime}(\{ q \}_{p}) \\
&= \exp(2\pi i q \theta) \prod \exp(2\pi i \{ qz_{p} \}_{p})
= \exp\left(2\pi i \left(q \theta + \sum \{ qz_{p} \}_{p}\right)\right). \qedhere
\end{align*}
\end{proof}

We can now characterize the character group of $Q_{(\ell_{p})}$ as a compact adelic abelian nilmanifold. The classical character group $\widehat{\mathbb{Q}} = \mathbb{A} / \mathbb{Q}$ is well-known, see e.g.~\cite{conrad}.

\begin{proposition}\label{Qhat}
$
\widehat{Q_{(\ell_{p})}} \simeq K_{(\ell_{p})} \rtimes_{[\hat{c}]} \left(\bigslant{\mathbb{A}_{(\ell_{p})}}{\Gamma_{(\ell_{p})}}\right) 
$ is the $(\ell_{p})$-adelic abelian nilmanifold.
\end{proposition}
\begin{proof}
For $(a_{\infty},(a_{p})) \in \mathbb{A}_{(\ell_{p})}$ and $(b_{p}) \in K_{(\ell_{p})}$, let $\chi_{a_{\infty},(a_{p})} \in \widehat{Q_{(\ell_{p})}}$ by
\[
\chi_{a_{\infty},(a_{p}),(b_{p})}(q) = \exp\left(2\pi i \left(qa_{\infty} + \sum \{ qa_{p} \}_{p} + \sum \{ q b_{p} \}_{p} \right)\right).
\]
where we treat $b_{p} \in \mathbb{Z} / p^{\ell_{p}}\mathbb{Z}$ as the representative in $\{ 0, 1, \ldots, p^{\ell_{p}}-1 \}$.
The mapping $\mathbb{A}_{(\ell_{p})} \times K_{(\ell_{p})} \to \widehat{Q_{(\ell_{p})}}$ is clearly a continuous homomorphism and by Lemma \ref{achar}, it is onto.

Let $(a_{\infty},(a_{p}),(b_{p})) \in \mathbb{A}_{(\ell_{p})} \times K_{(\ell_{p})}$ such that $\chi_{a_{\infty},(a_{p}),(b_{p})}$ is the trivial character.  Then $\exp(2\pi i (a_{\infty} + \sum \{ a_{p} \}_{p} + \sum \{ b_{p} \}_{p})) = 1$ so $a_{\infty} = - \sum \{ a_{p} \}_{p}\pmod{\mathbb{Z}}$ hence $a_{\infty} \in \mathbb{Q}$.  Then $a_{\infty} = \sum \{ a_{\infty} \}_{p}\pmod{\mathbb{Z}}$ so
\begin{align*}
\chi_{a_{\infty},(a_{p}),(b_{p})}(q) &= \exp\left(2\pi i \left(qa_{\infty} + \sum \{qa_{p} \}_{p} + \sum \{ qb_{p} \}_{p} \right)\right) \\
&= \exp\left(2\pi i \left(\sum \{ qa_{\infty} \}_{p} + \sum \{ qa_{p} \}_{p} + \sum \{ qb_{p} \}_{p} \right)\right) \\
&= \exp\left(2\pi i \sum \{ q(a_{\infty} + a_{p}) \}_{p}\right) \exp\left(2\pi i \sum \{ q(a_{\infty} + b_{p}) \}_{p}\right) \\
&= \chi_{0,(a_{p} + a_{\infty}),(b_{p} + a_{\infty}\pmod{p^{\ell_{p}}\mathbb{Z}})}(q).
\end{align*}
Now $a_{\infty} + a_{p} = a_{p} - \sum_{p^{\prime}} \{ a_{p^{\prime}} \}_{p^{\prime}}\pmod{\mathbb{Z}} = (a_{p} - \{ a_{p} \}_{p}) + \sum_{p^{\prime} \ne p} \{ a_{p^{\prime}} \}_{p^{\prime}}\pmod{\mathbb{Z}}$ so, as $\{ a_{p^{\prime}} \}_{p^{\prime}} \in \mathbb{Z}_{p}$ for $p^{\prime} \ne p$, we have $a_{\infty} + a_{p} \in \mathbb{Z}_{p}$.  By Lemma \ref{achar}, there is a unique $\theta$ and $z_{p}$ such that the trivial character is $\exp(2\pi i (q\theta + \sum \{ qz_{p} \}_{p}))$ which clearly must all be zero.  Then $a_{p} + a_{\infty} = 0$ and $b_{p} + a_{\infty} = 0\pmod{p^{\ell_{p}}\mathbb{Z}}$ for all $p$ which is precisely the statement that $(a_{\infty},(a_{p}),(b_{p})) = (a_{\infty},(-a_{\infty}),(-a_{\infty}\pmod{p^{\ell_{p}}\mathbb{Z}})) \in \Gamma_{(\ell_{p})}$ when embedded diagonally.

The fully split short exact sequence $0 \to K_{(\ell_{p})} \to \mathbb{A}_{(\ell_{p})} \times K_{(\ell_{p})} \to \mathbb{A}_{(\ell_{p})} \to 0$ then quotients to a short exact sequence $0 \to K_{(\ell_{p})} \to \widehat{Q_{(\ell_{p})}} \to \mathbb{A}_{(\ell_{p})} / \Gamma_{(\ell_{p})} \to 0$. Note, however, that $\Gamma_{(\ell_{p})}$ does not act trivially on the $K_{(\ell_{p})}$ coordinates under the diagonal embedding which is precisely what leads to the extension being nonsplit, i.e.~to the cocycle being nontrivial.
\end{proof}

We can characterize the character group of $Q_{(\ell_{p})}/\mathbb{Z}$ as an odometer.

\begin{proposition}\label{odo}
Let $0 \leq \ell_{p} \leq \infty$.  Then $\widehat{Q_{(\ell_{p})}/\mathbb{Z}}$ equipped with multiplication by the identity character is isomorphic as a topological dynamical system to the $(\ell_{p})$-odometer.
\end{proposition}
\begin{proof}
By Lemma \ref{achar}, any $\chi \in \widehat{Q_{(\ell_{p})}/\mathbb{Z}}$ corresponds uniquely to $\theta \in [0,1)$ and $a_{p} \in \mathbb{Z}_{p}$.  Since $\chi(1) = 1$, we have $\theta = 0$.  The $p$-adic expansions $a_{p} = \sum_{t=0}^{\infty} a_{p,t} p^{t}$ have the property that $a_{p,t+1}\pmod{p^{t}} = a_{p,t}$ so the values $a_{p,t}$ uniquely determine a point
$x \in \mathcal{O}_{(\ell_{p})}$ via the Chinese Remainder Theorem.

Conversely, given $x \in \mathcal{O}_{(\ell_{p})}$, if one defines $a_{p,t}$ as above, then $a_{p,t} \to a_{p} \in \mathbb{Z}_{p}$ which uniquely determine a character on $Q_{(\ell_{p})} / \mathbb{Z}$.  We have then described a one-one onto mapping from $\widehat{Q_{(\ell_{p})}/\mathbb{Z}}$ to $\mathcal{O}_{(\ell_{p})}$, which is easily checked to be continuous from the topology of pointwise convergence to the natural topology.

Let $(a_{p})$ correspond to $\chi$ and $\iota$ be the identity character.  Then
for $t \leq \ell_{p}$,
\begin{align*}
(\iota\chi)(p^{-t}) &= \exp(2\pi i p^{-t})\chi(p^{-t}) = \exp(2\pi i p^{-t})\exp(2\pi i \{ p^{-t}a_{p} \}_{p}) 
= \exp(2\pi i \{ p^{-t}(a_{p} + 1) \}_{p}).
\end{align*}
As the natural action on $\mathcal{O}_{(\ell_{p})}$ maps to the action $a_{p,t} \mapsto a_{p,t} + 1 \pmod{p^t}$, the claim follows.
\end{proof}

\section{The maximal equicontinuous factor}\label{MEF}

In this section, we characterize the maximal equicontinuous factor of any low complexity minimal subshift as the twisted product of a (possibly degenerate) odometer and a rotation of an adelic abelian nilmanifold (i.e., a circle or solenoid). 

\begin{definition}
Let $X$ be an infinite minimal subshift with $\limsup p(q)/q < 1.5$. Let $L_{X}(p)$ and $R_{X}(p)$ and $e_{p}$ be as in Theorem \ref{evalgroup}. Define
\begin{align*}
&\text{the \textbf{adelic abelian nilmanifold associated to $X$}} &&\quad \mathcal{M}_{X} = \mathcal{M}_{(L_{X}(p))} \\
&\text{the \textbf{odometer associated to $X$}} &&\quad \mathcal{O}_{X} = \mathcal{O}_{(R_{X}(p))} = \varprojlim_{k \to \infty} \bigslant{\mathbb{Z}}{\mathrm{gcd}(\len{v_{k+1}},\len{v_{k}})\mathbb{Z}} \\
&\text{the \textbf{adelic eigenvalue group for $X$}} &&\quad Q_{X} = Q_{(L_{X}(p))} \\
&\text{the \textbf{rational eigenvalue group for $X$}} &&\quad R_{X} = Q_{(R_{X}(p))}
\end{align*}
and the \textbf{eigenvalue cocycle for $X$} is $\omega : Q_{X} \times Q_{X} \to R_{X} / \mathbb{Z}$ given by
\[
\omega(q,q^{\prime}) = \sum_{p} \{ \{ q \}_{p} e_{p} \}_{p} + \{ \{ q^{\prime} \}_{p} e_{p} \}_{p} - \{ \{ q + q^{\prime} \}_{p} e_{p} \}_{p}.
\]
\end{definition}

Note that the eigenvalue cocycle is uniquely determined (up to cohomology class) by the $e_{p}$.

\begin{theorem}\label{MEFprod}
Let $X$ be an infinite minimal subshift with $\limsup p(q)/q < 1.5$. Let $L_{X}(p)$ and $R_{X}(p)$ and $e_{p}$ be as in Theorem \ref{evalgroup}.

Then $X$ is measurably isomorphic to its maximal equicontinuous factor
\[
\mathcal{M}_{X} \rtimes_{[\hat{\omega}]} \mathcal{O}_{X} 
\]
where $[\hat{\omega}]$ is the cohomology class of the topological extension dualizing the eigenvalue cocycle $\omega$.

If $0 < R_{X}(p) < \infty$ for only finitely many primes then the MEF is the direct product $\mathcal{M}_{X} \times \mathcal{O}_{X}$ and the action on the MEF is the product of the adelic rotation $(\alpha,(e_{p}))$ on $\mathcal{M}_{X}$ and the usual action of $\mathcal{O}_{X}$.
\end{theorem}

\begin{remark}
A more abstract description of the MEF, though one possibly easier to work with in some circumstances, is to define the section $s : Q_{X} \to \mathcal{E}_{X}$ by $s(q) = q\alpha + \sum \{ \{ q \}_{p} e_{p} \}\pmod{\mathbb{Z}}$ and then
\[
\widehat{\mathcal{E}_{X}} \simeq \left\{ (h,\rho) \in (S^{1})^{Q_{X}} \times \widehat{R_{X}/\mathbb{Z}}~:~h(q)h(q^{\prime}) = h(q + q^{\prime})\rho(\omega(q,q^{\prime}))~\text{for $q,q^{\prime} \in Q_{X}$} \right\}
\]
where the group is endowed with the subspace topology and coordinate-wise multiplication. The isomorphism is to set $h(q) = \chi(s(q))$ and $\rho = \chi\big{|}_{R_{X}/\mathbb{Z}}$ for $\chi \in \widehat{\mathcal{E}_{X}}$ and the rotation is
\[
(h,\rho) \mapsto (q \mapsto h(q)\exp(2\pi i s(q)), r \mapsto \rho(r)\exp(2\pi i r))
\]
making it clear that the $e_{p}$ determine the extension and $\alpha$ only enters into the rotation.
\end{remark}

We begin by describing the adelic abelian nilmanifolds in question.

\begin{notation}
We write $\mathbb{A}_{X} = \mathbb{A}_{(L_{X}(p))}$, $\Gamma_{X} = \Gamma_{(L_{X}(p))}$ and $K_{X} = K_{(L_{X}(p))}$.
\end{notation}

\begin{example}
The simplest example is when $L_{X}(p) = R_{X}(p) = 0$ for all primes $p$, which for instance happens for any Sturmian subshift.  Here $\mathbb{A}_{X} = \mathbb{R}$ and $Q_{X} = \Gamma_{X} = \mathbb{Z}$ so upon quotienting, $\mathcal{M}_{X} = \mathbb{A}_{X} / Q_{X} = \mathbb{R} / \mathbb{Z} = S^1$ and so the MEF is an irrational circle rotation.% and the corresponding subshifts are exactly the Sturmian subshifts.
\end{example}

\begin{example}
An example of a $p$-adic MEF is when $L_{X}(2) = \infty$ and $L_{X}(p) = 0$ for $p \ne 2$ and $R_{X}(p) = 0$ for all $p$. Here
$\mathbb{A}_{X} = \mathbb{R} \times \mathbb{Q}_2$ and $Q_{X} = \Gamma_{X} = \mathbb{Z}[\nicefrac{1}{2}]$ and $K_{X}$ is trivial so 
$\mathcal{M}_{X} = (\mathbb{A}_{X} \times \{1\}) / Q_{X} = \mathbb{A}_{X} / \Gamma_{X} =  (\mathbb{R} \times \mathbb{Q}_2) / \mathbb{Z}[\nicefrac{1}{2}] = \mathcal{M}_2$. Therefore, the MEF is a rotation of $\mathcal{M}_2$ as described in Section \ref{int6}. This MEF structure occurs for Example 1.4, but could also occur for a subshift where $r_{k}=1$ and $n_{k}=m_{k}+1$ for all $k$.
\end{example}

The group $\mathcal{M}_{2}$ is also referred to as the dyadic solenoid and can be viewed as the inverse limit of circle rotations, namely the inverse limit of the circle rotations isomorphic to the (fully nonsplit) central extensions of the circle by $\mathbb{Z} / 2^{k}\mathbb{Z}$.

The first such group in the inverse limit is the nonsplit central extension of $S^{1}$ by $\mathbb{Z} / 2\mathbb{Z}$ which can be viewed as the group $[0,1) \times \{ 0, 1 \}$ under the operation $(\theta,a) + (\theta^{\prime},a^{\prime}) = (\theta + \theta^{\prime}~\mod~1, (a + a^{\prime} + \{ 1~\text{if $\theta + \theta^{\prime} \geq 1$}; 0~\text{otherwise} \}~\mod~2)$. The  topology on the set $[0,1) \times \{ 0, 1\}$ is that obtained by identifying it with $\mathbb{R}/2\mathbb{Z}$ via $(\theta,a) = \theta + a$, i.e.~the point $(0,1)$ is ``glued" onto the $S^{1} \times \{ 0 \}$ copy of the circle as the ``endpoint" $(1,0)$.

This group is quickly seen to be homeomorphically and algebraically isomorphic to $S^{1}$ via the map $(\theta,a) \mapsto \theta/2 + a/2$. At the level of groups, this identification is immediate but for our purposes where we also have a distinguished element for the rotation, some care is needed as rotation by $(\alpha,0)$ in $(\mathbb{Z} / 2\mathbb{Z}) \rtimes_{c} S^{1}$ becomes rotation by $\alpha/2$ in $S^{1}$ but rotation by $(\alpha,1)$ becomes rotation by $(\alpha + 1)/2$.

\begin{example}
A concrete example of a situation where $\ell_{p} < \infty$ and therefore $K_{(\ell_{p})}$ is nontrivial is when $\ell_{2} = 1$ and $\ell_{p} = 0$ for all $p \ne 2$. In this case, $\mathcal{M}_{X}$ is precisely $(\mathbb{Z} / 2\mathbb{Z}) \rtimes_{c} S^{1}$, i.e.~the group described immediately above as the first stage in the inverse limit defining the dyadic solenoid.

The group is isomorphic, topologically as well as algebraically, to $S^{1}$. However the rotation of the MEF by $(\alpha,1)$ corresponds to rotation by $(\alpha + 1)/2$ under the isomorphism.

Similarly, if say, $\ell_{2} = 3$ and $\ell_{3} = 1$ and $\ell_{p} = 0$ for $p \ne 2,3$ then the MEF is $(\mathbb{Z} / 24\mathbb{Z}) \rtimes_{c} S^{1}$ under rotation by $(\alpha,a,b)$ which is isomorphic to $S^{1}$ under rotation by $(\alpha + 9a + 16b)/24$.

For completeness, we also mention the example where $\ell_{2} = 1$ and $\ell_{3} = \infty$ and $\ell_{p} = 0$ for $p \ne 2,3$. Here the MEF is $(\mathbb{Z} / 2\mathbb{Z}) \rtimes_{c} (\mathbb{R} \times \mathbb{Q}_{3} / \mathbb{Z}[1/3])$ which is isomorphic to $\mathbb{R} \times \mathbb{Q}_{3} / \mathbb{Z}[1/3]$ with the caveat that rotation by $(\alpha,e_{3},1)$ becomes rotation by $(\alpha/2 + 1/2, e_{3}/2 - 1/2)$ (note that $1/2$ is a $3$-adic integer).
\end{example}

The reason for the seemingly unnecessary complication of writing the groups as central extensions by compact abelian groups is that if $0 < \ell_{p} < \infty$ for infinitely many primes $p$ then there is no analogue of rotation by $\alpha/2$ available. It is worth noting that if $0 < \ell_{p} < \infty$ for only finitely many primes then the adelic abelian nilmanifold for $(\ell_{p})$ is isomorphic to the nilmanifold which treats all such primes as having $\ell_{p} = 0$ provided we replace $\alpha$ by $\alpha / \prod_{p : \ell_{p} < \infty} p^{\ell_{p}}$ (possibly plus a correction term $a_{p}/p^{\ell_{p}}$ for each $p$) just as we could have replaced $\alpha$ by $\alpha/2$ in the $(\mathbb{Z}/2\mathbb{Z}) \rtimes_{c} S^{1}$ situation.

The natural framework for expressing such behavior is quotients of Lie groups by lattices: if $(\theta,z), (\theta^{\prime},z^{\prime}) \in [0,1) \times \mathbb{Z}_{2}$ then, treated as elements of $\mathbb{R} \times \mathbb{Q}_{2}$, $(\theta + \theta^{\prime}~\mod~1, z + z^{\prime} + \{ 1~\text{ if $\theta + \theta^{\prime} \geq 1$}; 0~\text{otherwise} \})$ is equivalent to $(\theta + \theta^{\prime}, z + z^{\prime})$ in the quotient group $(\mathbb{R} \times \mathbb{Q}_{2})/\mathbb{Z}[1/2]$ so the skew product action (equivalently, the group operation in the central extension) is the natural product action in the $p$-adic Lie group quotiented by the lattice.

\begin{example}
For another concrete example, if $\ell_{p} = 1$ and $r_{p} = 0$ for all primes $p$ then $\mathbb{A}_{(\ell_{p})} = \mathbb{R}$ and $\Gamma_{(\ell_{p})} = \mathbb{Z}$ and $K_{(\ell_{p})} = \prod_{p} \mathbb{Z} / p\mathbb{Z}$ and the corresponding MEF is $K_{(\ell_{p})} \rtimes_{c} S^{1}$ which is the inverse limit of circle rotations where at each stage in the process, we divide by the next prime. This is exactly the same process as that for obtaining the dyadic solenoid with the modification that instead of taking circle rotations by $\alpha / 2^{k}$ as $k \to \infty$, we take circle rotations by $\alpha / (p_{1}p_{2} \ldots p_{k})$ ranging over the products of consecutive primes.
\end{example}

%\begin{definition}
%The \textbf{odometer associated to $X$} is
%$
%\mathcal{O}_{X} = \mathcal{O}_{(R_{X}(p))} = \varprojlim_{k \to \infty} \bigslant{\mathbb{Z}}{\mathrm{gcd}(\len{v_{k+1}},\len{v_{k}})\mathbb{Z}}
%$
%under the natural (coordinatewise) $+1$ action where $v_{k}$ and $v_{k+1}$ are the words from Proposition \ref{words}.
%\end{definition}

In general, the maximal equicontinuous factor will be expressed as a twisted product of the  adelic abelian nilmanifold associated to the subshift with the odometer associated to the subshift, i.e.~a central extension of the odometer by the adelic abelian nilmanifold.

%\begin{notation}
%For topological groups $G$ and $H$ and a cocycle $\omega : H \times H \to G$, the $\omega$-twisted product $G \rtimes_{\omega} H$ is the central extension of $H$ by $G$ with $\omega$ as the cocycle defining the action.
%\end{notation}
%
%The direct product $G \times H$ is precisely the $\omega$-twisted product using the trivial cocycle for $\omega$.

%\begin{theorem}\label{MEFprod}
%Let $X$ be an infinite minimal subshift with $\limsup p(q)/q < 1.5$.  Then $X$ is measurably isomorphic to its maximal equicontinuous factor $\mathcal{M}_{X} \rtimes_{[\hat{\omega}]} \mathcal{O}_{X}$ where the action and cocycle (up to cohomology class) are uniquely determined by the $e_{p}$.
%
%When the $e_{p}$ from Theorem \ref{evalgroup} satisfy $e_{p} \in \mathbb{Z}_{p}$ for all but finitely many primes, the maximal equicontinuous factor is the direct product $\mathcal{M} \times \mathcal{O}_{X}$ (and the rotation is the product of the natural rotations); in particular when $0 < R_{X}(p) < \infty$ for only finitely many primes then the maximal equicontinuous factor is the direct product.
%\end{theorem}

The remainder of this section is devoted to the proof of Theorem \ref{MEFprod}.  We start by characterizing the MEF as the group of characters on the multiplicative eigenvalue group. The MEF is naturally equipped with the action of multiplication by the identity character: if we define $\iota : E_{X} \to S^{1}$ by $\iota(\gamma) = \mathrm{exp}(2\pi i \gamma)$ then the action on the MEF is given by $\chi \mapsto \chi \cdot \iota$, i.e. $(T\chi)(\gamma) = \chi(\gamma)\iota(\gamma) = \chi(\gamma)\mathrm{exp}(2\pi i \gamma)$.

\begin{proposition}\label{almostoneone}
The maximal equicontinuous factor of $X$ is $\widehat{\mathcal{E}_{X}}$ equipped with Haar measure under the action of multiplication by the identity character.

Moreover, $(X,\sigma)$ is measurably isomorphic to $\widehat{\mathcal{E}_{X}}$ under the action of multiplication by the identity character.
\end{proposition}
\begin{proof}
By Theorem 2.21 in \cite{MR3160543}, the maximal equicontinuous factor is homeomorphic to $\widehat{\mathcal{E}_{X}}$ under multiplication by the identity character.  Since $X$ has discrete spectrum, Theorem \ref{isochars} implies $X$ is measurably isomorphic to $\widehat{\mathcal{E}_{X}}$ under that action.
\end{proof}

Next we establish that the space of characters is a twisted product of the spaces of characters on $Q_{X}$ and $R_{X}/\mathbb{Z}$.  By slight abuse of notation, for $\chi \in \widehat{\mathcal{E}_{X}}$ and $\gamma \in E_{X}$, we will write $\chi(\gamma)$ to mean $\chi(\exp(2\pi i \gamma))$ and treat $\chi$ as a character on $E_{X}$ which maps $\mathbb{Z}$ to $1$.

\begin{proposition}\label{shortexact}
$\mathcal{E}_{X} = (R_{X} / \mathbb{Z}) \rtimes_{\omega} Q_{X}$ and 
$\widehat{\mathcal{E}_{X}} = \widehat{Q_{X}} \times_{[\hat{\omega}]} \widehat{R_{X}/\mathbb{Z}}$.
\end{proposition}
\begin{proof}
By Theorem \ref{evalgroup}, $E_{X} = \{ q\alpha + \sum_{p} \{ qe_{p} \}_{p} + r : q \in Q_{X}, r \in R_{X} \}$. For $\chi \in \widehat{\mathcal{E}_{X}}$ and $q, q^{\prime} \in Q_{X}$,
\begin{align*}
\chi(q + q^{\prime}) &= \chi\left((q + q^{\prime})\alpha + \sum \{ \{ q + q^{\prime} \}_{p}e_{p} \}_{p}\right) \\
&= \chi(q) \chi(q^{\prime}) 
\chi\left(\sum \{ \{q+q^{\prime}\}_{p} e_{p} \}_{p} - \sum \{ \{ q \}_{p} e_{p} \}_{p} - \sum \{ \{ q^{\prime} \}_{p}e_{p} \}_{p}\right).
\end{align*}
Now $\{ e_{p} \}_{p} \in p^{-R_{X}(p)}\mathbb{Z}$ so
\[
\sum \{ {q+q^{\prime}}_{p} e_{p} \}_{p} - \sum \{ \{ q \}_{p} e_{p} \}_{p} - \sum \{ \{ q^{\prime} \}_{p}e_{p} \}_{p} \in R_{X}
\]
meaning that if we take $\chi_{Q} \in \widehat{Q_{X}}$ and define $\chi : \mathcal{E}_{X} \to S^{1}$ by
\[
\chi\left(q\alpha + \sum_{p} \{ \{ q \}_{p} e_{p} \}_{p} + r\right) = \chi_{Q}(q)
\]
for all $r \in R_{X}$ then this map $\widehat{Q_{X}} \to \widehat{\mathcal{E}_{X}}$ is well-defined. The restriction map $\widehat{E_{X}} \to \widehat{R_{X}/\mathbb{Z}}$ by $\chi \mapsto \chi\big{|}_{R}$ is a natural homomorphism and the image of the embedding of $\widehat{Q_{X}}$ is precisely the kernel of the restriction map. Therefore there is a short exact sequence as claimed (that the embedding is one-one is easily checked) and the displayed equality at the start of the proof shows that the cocycles involved are as claimed.
\end{proof}

%\begin{remark}
%In principle, one can identify the cocycle $\omega : \widehat{Q_{X}} \times \widehat{Q_{X}} \to \widehat{R_{X}/\mathbb{Z}}$ in terms of $\{ e_{p} \}_{p}$ by analyzing the behavior of $\sum \{ \{q+q^{\prime} \}_{p} e_{p} \}_{p} - \sum \{ \{ q \}_{p} e_{p} \}_{p} - \sum \{ \{ q^{\prime} \}_{p}e_{p} \}_{p}$ but we do not attempt that here.
%\end{remark}

We now consider the case when $e_{p} \in \mathbb{Z}_{p}$ for all but finitely many primes. The key consequence of this additional hypothesis is that we can write $\sum_{p} \{ \{q\}_{p} e_{p} \}_{p} = \sum_{p} \{ qe_{p} \}_{p}\pmod{\mathbb{R_{X}}}$ without potential convergence issues.

\begin{proposition}\label{product}
If $e_{p} \in \mathbb{Z}_{p}$ for all but finitely many primes then $\widehat{\mathcal{E}_{X}}$ is the direct product $\widehat{Q_{X}} \times \widehat{R_{X}/\mathbb{Z}}$ and the action of multiplication by the identity character on $\widehat{\mathcal{E}_{X}}$ maps to the action of multiplication by $\exp(2\pi i (q \alpha + \sum_{p} \{ qe_{p} \}_{p}))$ on $\widehat{Q_{X}}$ and multiplication by the identity character on $\widehat{R_{X}/\mathbb{Z}}$.
\end{proposition}
\begin{proof}
Let $\chi \in \widehat{\mathcal{E}_{X}}$.  For $q \in Q_{X}$, set $\chi_{Q}(q) = \chi(q\alpha + \sum_{p} \{ qe_{p} \}_{p})$ where we have used that $\sum \{ e_{p} \}_{p}$ is convergent to ensure this is well-defined.  Since $\sum \{ (q + q^{\prime})e_{p} \}_{p} = \sum \{ qe_{p} \}_{p} + \sum \{ q^{\prime}e_{p} \}_{p}\pmod{\mathbb{Z}}$, and since $\chi(1) = 1$,
\begin{align*}
\chi_{Q}(q + q^{\prime}) &= \chi\left(q\alpha + \sum \{ qe_{p} \}_{p}\right)\chi\left(q^{\prime}\alpha + \sum \{ q^{\prime}e_{p} \}_{p}\right)\chi\left(\sum \{ (q+q^{\prime})e_{p} \}_{p} - \sum \{ qe_{p} \}_{p} - \sum \{ q^{\prime}e_{p} \}_{p}\right) \\
&= \chi_{Q}(q)\chi_{Q}(q^{\prime})~1
\end{align*}
so $\chi_{Q} \in \widehat{Q_{X}}$.  Therefore for any $q \in Q_{X}$ and $r \in R_{X}$,
\[
\chi\left(q\alpha + \sum \{ qe_{p} \}_{p} + r\right) = \chi_{Q}(q)\chi(r)
\]
so $\chi \mapsto \chi_{Q} \cdot \chi \big{|}_{R_{X}}$ defines a homomorphism $\widehat{\mathcal{E}_{X}} \to \widehat{Q_{X}} \times \widehat{R_{X}/\mathbb{Z}}$.  As every such product of characters defines a character on $\mathcal{E}_{X}$, the homomorphism is onto and it is easily seen to be continuous and have trivial kernel. Put more succinctly: the map $\chi \mapsto \chi_{Q}$ shows that the short exact sequence in Proposition \ref{shortexact} necessarily splits directly, equivalently that the cocycle there is trivial.

The action of multiplication by the identity character on $\widehat{\mathcal{E}_{X}}$ on $\chi = \chi_{Q} \cdot \chi \big{|}_{R_{X}}$ is
\begin{align*}
(\chi\iota)\left(q\alpha + \sum \{ qe_{p} \}_{p} + r\right)
&= \chi\left(q\alpha + \sum \{ qe_{p} \}_{p} + r\right) \exp\left(2\pi i \left(q\alpha + \sum \{ qe_{p} \}_{p} + r\right)\right) \\
&= \chi_{Q}(q)\exp\left(2\pi i \left(q\alpha + \sum \{ qe_{p} \}_{p}\right)\right)~\chi \big{|}_{R_{X}}(r)\exp(2\pi i r). \qedhere
\end{align*}
\end{proof}

Finally we are in a position to prove the MEF has the claimed structure.

\begin{proof}[Proof of Theorem \ref{MEFprod}]
By Proposition \ref{almostoneone}, $X$ is measurably isomorphic to 
its maximal equicontinuous factor $\widehat{\mathcal{E}_{X}}$ under multiplication by the identity character.  By Proposition \ref{shortexact}, $\widehat{\mathcal{E}_{X}}$ under multiplication by the identity character is the $[\hat{\omega}]$-twisted product of $\widehat{Q_{X}}$ and $\widehat{R_{X}/\mathbb{Z}}$.

By Proposition \ref{odo}, $\widehat{R_{X}/\mathbb{Z}}$ is isomorphic as a topological dynamical system to $\mathcal{O}_{X}$.
By Proposition \ref{Qhat}, $\widehat{Q_{X}}$ is isomorphic to $\mathcal{M}_{X}$ equipped with the action of multiplication by the identity character.

When $e_{p} \in \mathbb{Z}_{p}$ for all but finitely many $p$, Proposition \ref{product} then gives that the twisted product is in fact a direct product and the action on the character $\chi_{a_{\infty},(a_{p})}$ is
\begin{align*}
\chi_{a_{\infty},(a_{p})}(q)\exp\left(2\pi i \left(q\alpha + \sum \{ qe_{p} \}_{p}\right)\right)
&= \exp\left(2\pi i \left( qa_{\infty} + \sum \{ qa_{p} \}_{p} + q\alpha + \sum \{ qe_{p} \}_{p}\right)\right) \\
&= \chi_{a_{\infty}+\alpha,(a_{p}+e_{p})}(q).
\end{align*}
so the action on $\mathcal{M}_{X}$ is $(a_{\infty},(a_{p}),(b_{p})) \mapsto (a_{\infty}+\alpha,(a_{p}+e_{p}),(b_{p} + e_{p}\pmod{p^{\ell_{p}}}))$, i.e.~translation by the adelic element $(\alpha,(e_{p}))$ times $(e_{p}) \in K_{X}$.
\end{proof}

\section{Existence of low complexity minimal subshifts for every odometer}\label{all}

We here demonstrate that there are no restrictions on the adelic abelian one-dimensional nilmanifolds $\mathcal{M}_{X}$ and odometers $\mathcal{O}_{X}$ which can appear in the MEF of an infinite low complexity subshift.  Other than the case when $\mathcal{M}_{X}$ is a finite group extension of $S^{1}$ and $\mathcal{O}_{X}$ is finite, we show that $\limsup \frac{p(q)}{q}$ can take any value in $[1, 1.5)$ for subshifts with that MEF.

\begin{theorem}\label{nobetter}
Let $\mathcal{O}$ be an odometer and $\mathcal{M}$ be an adelic abelian one-dimensional nilmanifold, i.e.~$\mathcal{M} = K_{(\ell_{p})} \rtimes_{\omega} (\mathbb{A}_{(\ell_{p})} / \Gamma_{(\ell_{p})})$ for some $(\ell_{p})_{p \in \mathbb{P}}$. Then there exists an infinite minimal subshift with $\lim p(q)/q = 1$ which has maximal equicontinuous factor the twisted product of $\mathcal{O}$ and a rotation on $\mathcal{M}$.

If there are only finitely many primes which appear in the construction of $\mathcal{O}$ with bounded degree then the maximal equicontinuous factor will in fact be the direct product.

If $\mathcal{M}$ is not a finite group extension of $S^{1}$ or $\mathcal{O}$ is infinite (or both), then for every 
$0 < \delta < \frac{1}{2}$, there exists a minimal uniquely ergodic subshift with $\limsup p(q)/q = 1 + \delta$ which 
has maximal equicontinuous factor of the same type.
\end{theorem}

\begin{remark}
We make two comments about Theorem~\ref{nobetter}. First of all, it's unavoidable that the final statement excludes the case where both $\mathcal{M}$ is a finite group extension of $S^{1}$ and $\mathcal{O}$ is finite; in that case, $a_k$ is eventually $1$, meaning that the substitutions are eventually of the form $\tau_{m_k,m_k+1,0}$ (say for $k > k_0$). In that case, $X$ is the image of a Sturmian subshift under the substitution $\rho_{k_0}$. Such a subshift is called \textbf{quasisturmian}, and is known to have 
$p(q) \leq q + C$ for a constant $C$ (\cite{quasisturm}), and so is forced to have $\limsup p(q)/q = 1$.

Secondly, we want to be clear that we are not characterizing the set of possible MEFs of infinite minimal low-complexity subshifts, since we only show that a single adele $(\alpha,(e_{p}))$ can occur together with a pair of a nilmanifold and an odometer; we do not currently know which rotations $(\alpha,(e_{p}))$ nor which cocycles can be associated with a specific group $\mathcal{M} \times_{\omega} \mathcal{O}$. Even in the case when no cocycle appears, we do not know which rotations can be associated to the specific groups. % can occur We suspect that in fact any triple $(\alpha,(e_{p})), \mathcal{M}, \mathcal{O}$ is possible provided $e_{p}\in\mathbb{Z}_{p}$ for all $p$, but were not able to show this.
\end{remark}

\begin{proof}
Define a sequence $(\delta_k)$ of positive reals as follows: when $\delta > 0$, set $\delta_{k} = \delta$, and when $\delta = 0$, let $\delta_{k}$ be any sequence approaching $0$.

\emph{First case: $\mathcal{M}$ is not a finite group extension of $S^{1}$.}

Let $0 \leq x_{p} \leq \infty$ such that $\mathcal{M} = K_{(x_{p})} \rtimes_{\omega} \mathbb{A}_{(x_{p})} / \Gamma_{(x_{p})}$.  Let $S_{k} = p_{k}^{q_{k}}$ for $p_{k} \in \mathbb{P} \cup \{ 1 \}$ and $q_{k} \geq 0$ nonnegative integers such that for each prime $p$, $\sum_{k : p_{k} = p} q_{k} = x_{p}$. Since $\mathcal{M}$ is not a finite extension of $S^{1}$, $S_{k} \to \infty$ as then either $\ell_{p} = \infty$ for at least one $p$ or $\ell_{p} > 0$ for infinitely many $p$. We assume $p_{0} > 1$ for the same reason.
Define $y_{k} \in \mathbb{P} \cup \{ 1 \}$ such that $\mathcal{O} = \varprojlim \mathbb{Z} / y_{0}\cdots y_{k}\mathbb{Z}$.

The goal is to define $m_{k}$ and $n_{k}$ inductively so that the orbit closure of $\pi \circ \tau_{m_{0},n_{0},0} \circ \tau_{m_{1},n_{1},0} \circ \cdots \circ \tau_{m_{k},n_{k},0}(0)$ has $\mathcal{M}$ as its MEF and its complexity satisfies $\limsup p(q)/q = 1 + \delta$.

The induction will proceed using auxiliary sequence $\ell_{k}$, $t_{k}$ and $j_{k}$. Note that $\ell_{k}$ are not the $\ell_{p}$ in the construction of the MEF, here we are using $x_{p}$ for the construction of the MEF and $\ell_{k}$ to refer to an auxiliary sequence described below.

The sequences $t_{k}$ and $j_{k}$ will control the construction of the odometer $\mathcal{O}$. The idea being that $t_{k}$ is a ``slowing down" of $y_{k}$ obtained by interspersing $1$'s as needed and $j_{k}$ is the sequence tracking the times in the construction when a non-$1$ occurs. This arrangement is designed so that 
$t_{0} \cdot t_{1} \cdot t_{2} \cdot \cdots \cdot t_{k} = y_0 \cdots y_{j_{k}-1}$ meaning that $\mathcal{O} = \varprojlim \mathbb{Z} / y_{0}\cdots y_{k}\mathbb{Z} = \varprojlim \mathbb{Z} / y_{0}\cdots t_{0}t_{1}t_{2}\cdots t_{k}\mathbb{Z}$.

The sequence $\ell_{k}$ will be chosen to ensure the primes in the $S_{k}$ do not contribute to rational eigenvalues, i.e.~do not influence the construction of the odometer $\mathcal{O}$, and instead contribute solely to the $\mathcal{M}$ component of the MEF. This is accomplished by controlling divisibility properties of the $m_{k}$ using the sequence $\ell_{k}$ to enforce relative primeness as needed.

We now proceed with the technical argument defining these sequences.

Set $\ell_{-1} = \ell_{0} = 1$ and $j_{0} = 0$ and $t_{0} = t_{1} = 1$ and $S_{-1} = 1$.  For all $k \geq 1$, we will set $\ell_{k+1} = m_{k}\ell_{k} + t_{k}S_{k-1}\ell_{k-1}$.  For ease of notation, write $g_{k} = t_{0} \cdots t_{k}$.

Choose $m_{0}$ such that $m_{0} > \ceil{(\delta_{0}^{-1}-1)t_{1}S_{0}}$ and $p_{0}$ does not divide $m_{0} + 1$.  Then $\ell_{1} = m_{0}\ell_{0} + t_{0}S_{-1}\ell_{-1} = m_{0} + 1$ so $t_{1}$ divides $\frac{\ell_{1}}{g_{0}}$ and $p_{0}$ does not divide $\frac{\ell_{1}}{g_{1}}$.  Also $\gcd(\ell_{1},\ell_{0}) = 1 = g_{0}$.

Assume that $t_{k}$ divides $\frac{\ell_{k}}{g_{k-1}}$ and $p_{k-1}$ does not divide $\frac{\ell_{k}}{g_{k}}$ and $\gcd(\ell_{k},\ell_{k-1}) = g_{k-1}$.  
If $y_{j_{k}}$ divides $\frac{\ell_{k}}{g_{k}}$ then set $t_{k+1} = 1$ and $j_{k+1} = j_{k}$.  If not, set $t_{k+1} = y_{j_{k}}$ and $j_{k+1} = j_{k} + 1$.  Set $m_{k}^{\prime} = \ceil{(\delta_{k}^{-1}-1)t_{k+1}S_{k}} > S_{k}t_{k+1}$.  The map $m \mapsto m\frac{\ell_{k}}{g_{k}} + S_{k-1}\frac{\ell_{k-1}}{g_{k-1}} \mod t_{k+1}$ is a cyclic onto homomorphism since $\gcd(\frac{\ell_{k}}{g_{k}},t_{k+1}) = 1$ (since $t_{k+1}$ is a prime power or $1$).  So there exists $0 \leq i < t_{k+1}$ such that $t_{k+1}$ divides $(m_{k}^{\prime}-i)\frac{\ell_{k}}{g_{k}} + S_{k-1}\frac{\ell_{k-1}}{g_{k-1}}$.

If $p_{k}t_{k+1}$ were to divide both $(m_{k}^{\prime}-i)\frac{\ell_{k}}{g_{k}} + S_{k-1}\frac{\ell_{k-1}}{g_{k-1}}$ and $(m_{k}^{\prime}-i-t_{k+1})\frac{\ell_{k}}{g_{k}} + S_{k-1}\frac{\ell_{k-1}}{g_{k-1}}$ then $p_{k}t_{k+1}$ divides $t_{k+1}\frac{\ell_{k}}{g_{k}}$ so $p_{k}$ divides $\frac{\ell_{k}}{g_{k}}$ but then $p_{k}$ divides both $\frac{\ell_{k}}{g_{k}}$ and $S_{k-1}\frac{\ell_{k-1}}{g_{k-1}}$ which is impossible as $\gcd(\frac{\ell_{k}}{g_{k}},S_{k-1}\frac{\ell_{k-1}}{g_{k-1}}) = 1$.  Therefore we may take $m_{k}$ such that $m_{k}^{\prime} - 2t_{k+1} \leq m_{k} < m_{k}^{\prime}$ so that $t_{k+1}$ divides $\frac{\ell_{k+1}}{g_{k}}$ and $p_{k}$ does not divide $\frac{\ell_{k+1}}{g_{k+1}}$.
We also have $\gcd(\ell_{k+1},\ell_{k}) = \gcd(m_{k}\ell_{k}+t_{k}S_{k-1}\ell_{k-1},\ell_{k}) = \gcd(t_{k}S_{k-1}\ell_{k-1},\ell_{k})
= g_{k-1}\gcd(t_{k}S_{k-1}\frac{\ell_{k-1}}{g_{k-1}},t_{k}\frac{\ell_{k}}{g_{k}})
= g_{k} \gcd(S_{k-1}\frac{\ell_{k-1}}{g_{k-1}},\frac{\ell_{k}}{g_{k}})
= g_{k}$.

Therefore the sequences exist by induction.  Note that if $y_{j_{k}} > 1$ and $t_{k+1} = 1$ then necessarily $t_{k+2} = y_{j_{k}}$ as otherwise $y_{j_{k}}$ divides $\frac{\ell_{k}}{g_{k}}$ and $y_{j_{k}}$ divides $\frac{\ell_{k+1}}{g_{k+1}}$ but $\gcd(\ell_{k+1},\ell_{k}) = g_{k}$.  By the construction of $j_{k}$, the sequence $(t_k)$ is just the sequence $(y_{j})$ with extra interspersed $1$s, and so the sequence $(t_k)$ induces the odometer $\mathcal{O}$.
Set $n_{k} = m_{k} + t_{k+1}S_{k}$.  Let $X$ be the orbit closure of $\lim \pi \circ \tau_{m_{0},n_{0},0} \circ \cdots \circ \tau_{m_{k},n_{k},0}(0)$ where $\pi(0) = 0$ and $\pi(1) = 01$ so $\len{v_{k}} = \ell_{k}$.
By Remark \ref{pands}, $\sum_{j=0}^{k-1} (n_{j} - m_{j} - 1)\len{v_{j}} < \sum_{j=0}^{k-1} (m_{j}-1)\len{v_{j}} = \len{p_{k}} < 3\len{v_{k}}$, so by Corollary \ref{sum},
\begin{align*}
\frac{p(\len{s_{k}v_{k}^{n_{k}-2}p_{k}})}{\len{s_{k}v_{k}^{n_{k}-2}p_{k}}}
&\leq 1 + \frac{(n_{k} - m_{k} - 1)\len{v_{k}} + \len{p_{k}} + C}{(n_{k}-2)\len{v_{k}}}
\leq 1 + \frac{t_{k+1}S_{k} + 3 + \frac{C}{\len{v_{k}}}}{\delta_{k}^{-1}t_{k+1}S_{k} -2t_{k+1} - 2} \to 1 + \delta
\end{align*}
since $S_{k} \to \infty$.
By Corollary \ref{sum} and Remark \ref{pands},
\[
\frac{p(\len{s_{k}v_{k}^{n_{k}-2}p_{k}})}{\len{s_{k}v_{k}^{n_{k}-2}p_{k}}}
\geq 1 + \frac{n_{k} - m_{k} - 1}{n_{k} + 1} \geq 1 + \frac{t_{k+1}S_{k} - 1}{\delta_{k}^{-1}t_{k+1}S_{k} + 2}
\to 1 + \delta.
\]
Since $\limsup \frac{p(q)}{q}$ is attained along the sequence $\len{s_{k}v_{k}^{n_{k}-2}p_{k}}$, $\limsup \frac{p(q)}{q} = 1 + \delta$.

By construction, $\gcd(\len{v_{k+1}},\len{v_{k}}) = t_{1} \cdots t_{k}$, and so $\mathcal{O}_X = \mathcal{O}$. Similarly, $a_0 \cdots a_k = (n_{0}-m_{0})\cdots (n_{k-1}-m_{k-1}) = S_{0}\cdots S_{k-1} t_{1} \cdots t_{k}$ so $\frac{\len{v_{0}}a_{0}\cdots a_{k}}{\gcd(\len{v_{k}},\len{v_{k+1}})} = S_{0}\cdots S_{k-1}$, implying that $\mathcal{M}_X = \mathcal{M}$. Therefore by Theorem \ref{MEFprod}, the subshift $X$ defined as the orbit closure of $\lim \pi \circ \tau_{m_{0},n_{0},0} \circ \cdots \circ \tau_{m_{k},n_{k},0}(0)$ has the claimed properties.

\emph{Second case: $\mathcal{M}$ is a finite group extension of $S^{1}$ and $\mathcal{O}$ is infinite}

A similar induction argument applies here.  Let $t_{k}$ such that $\mathcal{O} = \varprojlim \mathbb{Z}/t_{0}\cdots t_{k}\mathbb{Z}$ and let $\mathcal{M} = \mathbb{Z}/q\mathbb{Z} \rtimes S^{1}$.
Let $\pi(0) = 0^{q}$ and $\pi(1) = 0^{q}1$.  Let $j_{0} = 0$ and $S_{0} = 1$.  

We wish to inductively define a nondecreasing sequence $(j_k)$ (and then define $S_{k} = t_{j_{k-1}} t_{j_{k-1}+1} \cdots t_{j_{k}-1})$ and $m_k$ (and then define $n_k = m_k + S_k$) so that 
\begin{itemize}
\item $m_k \rightarrow \infty$
\item for all $k$, $S_0 \cdots S_{k-1}$ divides $\len{v_k}$
\item for all $k$, $\len{v_k}/(S_0 \cdots S_{k-1})$ and $\len{v_{k-1}}/(S_0 \cdots S_{k-2})$ are coprime
\item $\liminf \frac{m_k}{S_k} = \delta^{-1} - 1$
\end{itemize}

Suppose that these conditions are satisfied for $m_i$ and $S_i$ for $i < k$ (meaning that $S_0 \cdots S_{k-1}$ divides 
$\len{v_k}$). Then, define $J$ large enough that for
$S := t_{j_{k-1}} \cdots t_{J-1} > k$; call these our `desired' values, which we will achieve for 
$j_{k+L}$ and $S_{k+L}$ for some $L \leq 2$.
Choose $M' = \lfloor (\delta^{-1} - 1) S \rfloor$; there exists $M$ with $|M - M'| = o(S)$ and $\gcd(M, S) = 1$.
Now, if $G = \gcd(S, \frac{\len{v_k}}{S_0 \cdots S_{k-1}}) = 1$, define $j_k = j_{k-1}$ and $S_k = 1$, and
for convenience we denote $A = \frac{\len{v_k}}{S_0 \cdots S_{k-1}}$ and $B = \frac{\len{v_{k-1}}}{S_0 \cdots S_{k-2}}$
(which are coprime integers by inductive hypothesis). 
Now, choose $m_k > S$ such that 
$m_k A + B \equiv -A M^{-1} \pmod S$, and note that $m_k/S_k > S > k$ and
\[
\len{v_{k+1}} = m_k \len{v_k} + S_{k-1} \len{v_{k-1}} = S_0 \cdots S_{k-1} S_k (m_k A + B),
\]
where $m_k A + B \equiv -AM^{-1} \pmod S$ is relatively prime to $S$ (and also to $A$ since $\gcd(A,B) = 1$).

At the next step, define $j_{k+1} = J$, $S_{k+1} = S$, and
$m_{k+1} = M$. Then, $m_{k+1}/S_{k+1} = M/S = \delta^{-1} - 1 + o(1)$, and
\[
\len{v_{k+2}} = m_{k+1} \len{v_{k+1}} + S_{k} \len{v_{k}} = 
S_0 \cdots S_{k-1} S_k \left(M(m_k A + B) + A\right).
\]
Since the quantity in the parentheses is $0 \pmod S$ by definition of $m_{k+1}$, we have $S_0 \cdots S_{k+1}$ divides 
$\len{v_{k+2}}$, and $M(m_k A + B) + A$ and $m_k A + B$ are coprime since $A$ and $m_k A + B$ were, completing the check of the inductive hypotheses.

If instead $G = \gcd(S, \frac{\len{v_k}}{S_0 \cdots S_{k-1}}) = \gcd(S,A) > 1$, then we must `delay' for a step.
Since $A$ and $B$ are coprime, $\gcd(G, B) = 1$.
Define $S'$ to be the product of all prime powers appearing in the prime factorization of $S$ for primes which do not divide $B$.
Define $j_{k} = j_{k-1}$ (so $S_k = 1$) and $m_k$ to be any product of $S'$ by an integer $C$ coprime to $\gcd(S,B)$. Then, 
\[
\len{v_{k+1}} = m_k \len{v_k} + S_{k-1} \len{v_{k-1}} = S_0 \cdots S_k \left(m_k A + B\right),
\]
and by construction $\gcd(S, m_k A + B) = \gcd(S, CS' A + B) = 1$. In addition, $m_k A + B$ and $A$ are coprime since $A$ and 
$B$ were, verifying inductive hypotheses. So, at the next step, we have the desired GCD property (i.e., $\gcd\left(S, \frac{\len{v_{k+1}}}{S_0 \cdots S_k}\right) = 1$), and we can then proceed according to the previous paragraph. 

Finally, since all steps with $S_k > 1$ satisfy $m_k/S_k \rightarrow \delta^{-1} - 1$ and steps with $S_k = 1$ satisfy $m_k/S_k = m_k \rightarrow \infty$, $\liminf m_k/S_k = \delta^{-1} - 1$.

Then, as above, we have
\[
1 + \frac{S_{k} - 1}{m_k + S_{k} + 2} \leq
\frac{p(\len{s_{k}v_{k}^{n_{k}-2}p_{k}})}{\len{s_{k}v_{k}^{n_{k}-2}p_{k}}}
\leq 1 + \frac{(S_k + 3)}{(m_k + S_k - 2)},
\]
and so $\limsup \frac{p(q)}{q} = 1 + \delta$.  

By inductive hypotheses,  $\len{v_{0}}a_{0}\cdots a_{k} = qS_{0} \cdots S_{k-1}$, and 
$\gcd(\len{v_{k+1}},\len{v_{k}}) = S_{0} \cdots S_{k-1}$. For $X$ defined in the usual way, since the sequence of partial products of $(S_k)$ is a subsequence of partial products of $(T_k)$, $\mathcal{O}_X = \mathcal{O}$, and since 
$\frac{\len{v_{0}}a_{0}\cdots a_{k}}{\gcd(\len{v_{k}},\len{v_{k+1}})} = q$, $\mathcal{M}_X = \mathcal{M}$. Therefore, $X$ has the claimed properties by Theorem~\ref{MEFprod}.

\emph{Final case: both $\mathcal{M} =  \mathbb{Z}/y\mathbb{Z} \rtimes S^{1}$ and $\mathcal{O} = \mathbb{Z}/r\mathbb{Z}$.}

Let $\pi(0) = 0^{yr}$ and $\pi(1) = 0^{yr}1^{r}$.  Let $m_{k} = 1$ and $n_{k} = 2$ for all $k$. Then it's an immediate implication of Lemma~\ref{only} that $p(q+1) - p(q)$ is eventually $1$, and so $\limsup p(q)/q = 1$.

Also, $\len{v_{0}}a_{0}\cdots a_{k} = yr$ and $\gcd(\len{v_{k+1}},\len{v_{k}}) = \gcd(\len{v_{k}} + \len{v_{k-1}},\len{v_{k}}) = \gcd(\len{v_{k-1}},\len{v_{k}})$ for all $k$. Then, since $\gcd(\len{v_{1}},\len{v_{0}}) = \gcd(yr+yr+r,yr) = r$, $\gcd(\len{v_{k+1}},\len{v_{k}}) = r$ for all $k$. Therefore, $X$ defined as above has the claimed properties by Theorem \ref{MEFprod}.
\end{proof}

Concluding the paper, we address Examples 1.4-1.6 from the introduction. Example 1.4 is fairly straightforward; it is determined by substitutions with $|u_0| = |v_0| = 1$, $m_k = 3$ and $n_k = 5$. Therefore, by (\ref{ab}), $a_0 = 1$, all other $a_k = 2$, and all $b_k = 3$.
The verification that $\limsup p(q)/q < 3/2$ follows from Corollary \ref{sum} and Remark \ref{pands}. Namely, by Corollary~\ref{sum}, the limsup of $p(q)/q$ is achieved along the sequence $q_k=\len{s_{k}v_{k}^{n_{k}-2}p_{k}} = \len{s_k} + 3\len{v_k} + \len{p_k}$, which equals $\sum_{i=0}^{k-1} 3\len{v_i} + 4\len{v_k}$ by Remark \ref{pands}. The value of $p(q_k)$ is equal to
$q_k + \sum_{i = 0}^k (n_k - m_k - 1)\len{v_i} = \sum_{i=0}^{k-1} 4\len{v_i} + 5\len{v_k} + C$ for some constant $C$. Finally, we note that by the Perron-Frobenius theorem, the lengths $|v_i|$ grow exponentially with base the Perron eigenvalue of the incidence matrix 
$\left(\begin{smallmatrix} 2 & 4\\1 & 1 \end{smallmatrix}\right)$, which is $\kappa = \frac{3 + \sqrt{17}}{2}$. Therefore, 
$\frac{\sum_{i=0}^{k-1} \len{v_i}}{\len{v_k}} \rightarrow \frac{1}{\kappa - 1} = \frac{\sqrt{17} - 1}{8}$, and so
\[
\limsup p(q)/q = \lim p(q_k)/q_k = \lim \frac{\sum_{i=0}^{k-1} 4\len{v_i} + 5\len{v_k} + C}{\sum_{i=0}^{k-1} 3\len{v_i} + 4\len{v_k}} =
\frac{4(\sqrt{17}-1)/8 + 5}{3(\sqrt{17}-1)/8 + 4} = \frac{105 + \sqrt{17}}{86}  < 3/2.
\]
It remains to verify that the MEF is a rotation of $\mathcal{M}_{2} = (\mathbb{R} \times \mathbb{Q}_{2})/\mathbb{Z}\left[\nicefrac{1}{2}\right]$. As noted above, 
$a_0 = 1$ and all $a_k$ for $k > 0$ equal $2$. Also, $|v_0| = 1$, and it is easily checked by induction that all $|v_k|$ are odd; since
$\gcd(|v_{k+1}|, |v_k|)$ divides $\len{v_{0}} a_0 \ldots a_k = 2^k$ by Lemma~\ref{div}, all $\gcd(|v_{k+1}|, |v_k|) = 1$. Therefore, in the language of Theorem~\ref{MEFprod}, $\mathcal{O}_{X}$ is trivial and $\mathcal{M}_{X}$ is $\mathcal{M}_{2}$.  

Since the computations are significantly more unpleasant, we omit details of Example 1.5, except to note the following differences from Example 1.4. First, the limsup of $p(q)/q$ is now achieved along the sequence $q_k = \len{s_k v_k^{r_k - 1} u_k v_k^{r_k - 1} p_k} = \len{s_k u_k p_k}$. Second, now every $a_k$ for $k > 0$ is equal to $2^{r_k} (n_k - m_k) = 4$, and $\gcd(\len{v_k}, \len{v_{k+1}}) = 2^k$, which implies by Theorem~\ref{MEFprod} that $\mathcal{M}_{X} = \mathcal{M}_{2}$ and $\mathcal{O}_{X}$ is the dyadic odometer.

For Example 1.6, we cannot solve exactly for 
$\limsup p(q)/q$ since we do not have a closed form for
$m_k$ and $n_k$. However, we note that by Corollary~\ref{sum}, increasing
$m_k$ while keeping $n_k - m_k$ and $r_k$ constant can only decrease this limsup; since $(m_k, n_k)$ is always either
$(3, 5)$ or $(5, 7)$, this limsup is then clearly less than or equal to 
$\frac{105 + \sqrt{17}}{86}$ from Example 1.4. As in Example 1.4, $a_0 = 1$ and all other $a_k = 2$. It is easily checked by induction and the definition of the $\rho_k$ that for all $k$, $\len{v_k}$ is divisible by $2^k$, but not by $2^{k+1}$. Therefore, 
$\gcd(\len{v_k}, \len{v_{k+1}}) = 2^k = |v_0| a_0 \cdots a_k$, and so by Theorem~\ref{MEFprod}, 
$\mathcal{M}_{X}$ is $\mathbb{R}/\mathbb{Z} = S^{1}$ and $\mathcal{O}_{X}$ is the dyadic odometer.

\dbibliography{1.5}

\end{document}